\documentclass{article}

\usepackage[latin1]{inputenc}
\usepackage[T1]{fontenc}

\usepackage{bbm}
\usepackage{xypic}

\usepackage{amsmath}
\usepackage{amsfonts}
\usepackage{amssymb}
\usepackage{amsthm}
\usepackage{graphicx}
\usepackage{enumerate}

\newtheorem{theorem}{Theorem}

\newtheorem{corollary}[theorem]{Corollary}
\newtheorem{definition}[theorem]{Definition}
\newtheorem{lemma}[theorem]{Lemma}

\newtheorem{proposition}[theorem]{Proposition}
\newtheorem{remark}[theorem]{Remark}
\numberwithin{equation}{section}

\newcommand{\PP}{\mathbb{P}}
\newcommand{\PPP}{\hat{\mathbb{P}}}
\newcommand{\ppp}{\tilde{\mathbb{P}}}
\newcommand{\EE}{\mathbb{E}}
\newcommand{\EEE}{\hat{\mathbb{E}}}
\newcommand{\At}{\tilde{A}}
\newcommand{\Att}{\tilde{\tilde{A}}}
\newcommand{\Ab}{\bar{A}}

\newcommand{\Sj}{\tilde{\mathfrak{S}}_j}

\newcommand{\It}{\tilde{I}}

\begin{document}
\title{Near-critical percolation in two dimensions}
\author{Pierre Nolin}
\date{\'Ecole Normale Sup\'erieure and Universit\'e Paris-Sud}
\maketitle


\begin{abstract}
We give a self-contained and detailed presentation of Kesten's results that allow to relate critical and near-critical percolation on the triangular lattice. They constitute an important step in the derivation of the exponents describing the near-critical behavior of this model. For future use and reference, we also show how these results can be obtained in more general situations, and we state some new consequences.
\end{abstract}

\vspace{1.7cm}

\begin{center}
\includegraphics[width=12cm]{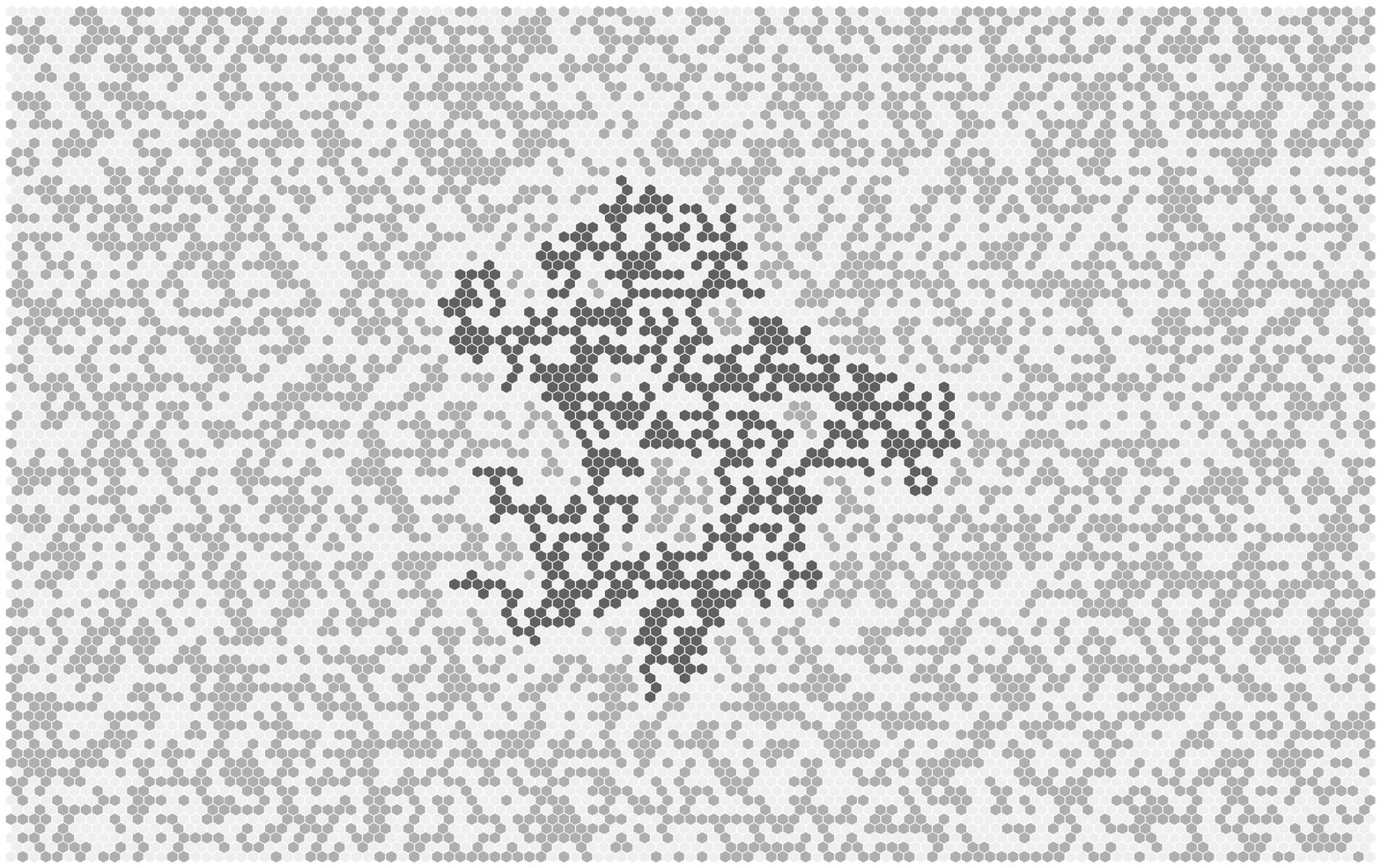}
\end{center}

\newpage

~

\vfill

D\'epartement de Math\'ematiques et Applications, \'Ecole Normale Sup\'erieure, 45 rue d'Ulm, 75230 Paris Cedex 05, France

\medskip

Laboratoire de Math\'ematiques, B\^atiment 425, Universit\'e Paris-Sud 11, 91405 Orsay Cedex, France

\bigskip
\bigskip
\bigskip

Research supported by the Agence Nationale pour la Recherche under the grant ANR-06-BLAN-0058.

\newpage

\tableofcontents

\newpage

~

\newpage

\section{Introduction}

Since 2000, substantial progress has been made on the mathematical understanding of percolation on the triangular lattice. In fact, it is fair to say that it is now well-understood. Recall that performing a percolation with parameter $p$ on a lattice means that each site is chosen independently to be black or white with probability $p$. We then look at the connectivity properties of the set of black sites (or the set of white ones). It is well-known that on the regular triangular lattice, when $p \leq 1/2$ there is almost surely no infinite black connected component, whereas when $p > 1/2$ there is almost surely one infinite black connected component. Its mean density can then be measured via the probability $\theta(p)$ that a given site belongs to this infinite black component.

Thanks to Smirnov's proof of conformal invariance of the percolation model at $p=1/2$ \cite{Sm}, allowing to prove that critical percolation interfaces converge toward $SLE_6$, and to the derivation of the $SLE_6$ critical exponents \cite{LSW1,LSW2} by Lawler, Schramm and Werner, it is possible to prove results concerning the behavior of the model when $p$ is exactly equal to $1/2$, that had been conjectured in the physics literature, such as the values of the arm exponents \cite{SmW,LSW4}. See e.g. \cite{W2} for a survey and references.

More than ten years before the above-mentioned papers, Kesten had shown in his 1987 paper \emph{Scaling relations for 2D-percolation} \cite{Ke4} that the behavior of percolation at criticality (\emph{ie} when $p=1/2$) and near criticality (\emph{ie} when $p$ is close to $1/2$) are very closely related. In particular, the exponents that describe the behavior of quantities such as $\theta(p)$ when $p \to 1/2^+$ and the arm exponents for percolation at $p=1/2$ are related via relations known as scaling (or hyperscaling) relations. At that time, it was not proved that any of the exponents existed (not to mention their actual value) and Kesten's paper explains how the knowledge of the existence and the values of some arm exponents allows to deduce the existence and the value of the exponents that describe ``near-critical'' behavior. Therefore, by combining this with the derivation of the arm exponents, we can for instance conclude \cite{SmW} that $\theta (p) = (p-1/2)^{5/36 + o(1)}$ as $p \to 1/2^+$.

Reading Kesten's paper in order to extract the statement that is needed to derive this result can turn out to be not so easy for a non-specialist, and the first goal of the present paper is to give a complete self-contained proof of Kesten's results that are used to describe near-critical percolation. We hope that this will be useful and help a wider community to have a clear and complete picture of this model.

It is also worth emphasizing that the proofs contain techniques (such as separation lemmas for arms) that are interesting in themselves and that can be applied to other situations. The second main purpose of the present paper is to state results in a more general setting than in \cite{Ke4}, for possible future use. In particular, we will see that the ``uniform estimates below the characteristic length'' hold  for an arbitrary number of arms and non-homogeneous percolation (see Theorem \ref{armsep} on separation of arms, and Theorem \ref{armnear} on arm events near criticality). Some technical difficulties arise due to these generalizations, but these new statements turn out to be useful. They are for instance instrumental in our study of gradient percolation in \cite{N1}. Other new statements in the present paper concern arms ``with defects'' or the fact that the finite-size scaling characteristic length $L_{\epsilon}(p)$ remains of the same order of magnitude when $\epsilon$ varies in $(0,1/2)$ (Corollary \ref{lengths}) -- and not only for $\epsilon$ small enough. This last fact is used in \cite{NW} to study the ``off-critical'' regime for percolation.

\section{Percolation background}

Before turning to near-critical percolation in the next section, we review some general notations and properties concerning percolation. We also sketch the proof of some of them, for which small generalizations will be needed. We assume the reader is familiar with the standard fare associated with percolation, and we refer to the classic references \cite{Ke_book,G_book} for more details.

\subsection{Notations}

\subsubsection*{Setting}

Unless otherwise stated, we will focus on site percolation in two dimensions, on the triangular lattice. This lattice will be denoted by $\mathbb{T} = (\mathbb{V}^T,\mathbb{E}^T)$, where $\mathbb{V}^T$ is the set of vertices (or ``sites''), and $\mathbb{E}^T$ is the set of edges (or ``bonds''), connecting adjacent sites. We restrict ourselves to this lattice because it is at present the only one for which the critical regime has been proved to be conformal invariant in the scaling limit.

The usual (homogeneous) site percolation process of parameter $p$ can be defined by taking the different sites to be \emph{black} (or occupied) with probability $p$, and \emph{white} (vacant) with probability $1-p$, independently of each other. This gives rise to a product probability measure on the set of configurations, which is referred to as $\PP_p$, the corresponding expectation being $\mathbb{E}_p$. We usually represent it as a random (black or white) coloring of the faces of the dual hexagonal lattice (see Figure \ref{hexagons}).

\begin{figure}
\begin{center}
\includegraphics[width=7cm]{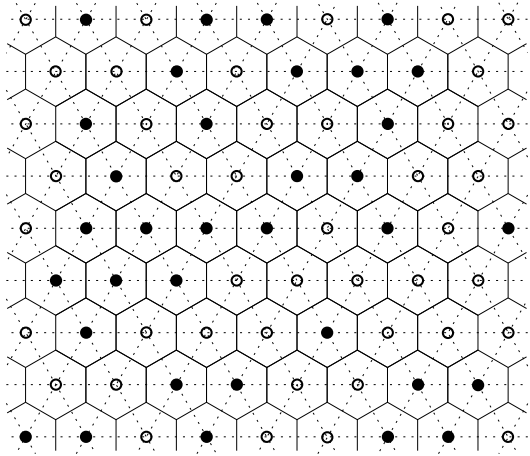}
\caption{Percolation on the triangular lattice can be viewed as a random coloring of the dual hexagonal lattice.}
\label{hexagons}
\end{center}
\end{figure}

More generally, we can associate to each family of parameters $\hat{p}=(\hat{p}_v)_v$ a product measure $\PPP$ for which each site $v$ is black with probability $\hat{p}_v$ and white with probability $1-\hat{p}_v$, independently of all other sites.

\subsubsection*{Coordinate system}

We sometimes use complex numbers to position points in the plane, but we mostly use oblique coordinates, with the origin in $0$ and the basis given by $1$ and $e^{i\pi/3}$, \emph{ie} we take the $x$--axis and its image under  rotation of angle $\pi/3$ (see Figure \ref{lattice}). For $a_1 \leq a_2$ and $b_1 \leq b_2$, the parallelogram $R$ of corners $a_j + b_k e^{i\pi/3}$ ($j,k = 1,2$) is thus denoted by $[a_1,a_2] \times [b_1,b_2]$, its interior being $\mathring{R} := ]a_1,a_2[ \times ]b_1,b_2[ = [a_1+1,a_2-1] \times [b_1+1,b_2-1]$ and its boundary $\partial R := R \setminus \mathring{R}$ the concatenation of the four boundary segments $\{a_i\} \times [b_1,b_2]$ and $[a_1,a_2] \times \{b_i\}$.

We denote by $\|z\|_{\infty}$ the infinity norm of a vertex $z$ as measured with respect to these two axes, and by $d$ the associated distance. For this norm, the set of points at distance at most $N$ from a site $z$ forms a rhombus $S_N(z)$ centered at this site and whose sides line up with the basis axes. Its boundary, the set of points at distance exactly $N$, is denoted by $\partial S_N(z)$, and its interior, the set of points at distance strictly less that $N$, by $\mathring{S}_N(z)$. To describe the percolation process, we often use $S_N := S_N(0)$ and call it the ``box of size $N$''. Note that it can also be written as $S_N=[-N,N] \times [-N,N]$. It will sometimes reveal useful to have noted that
\begin{equation}
|S_N(z)| \leq C_0 N^2
\end{equation}
for some universal constant $C_0$. For any two positive integers $n \leq N$, we also consider the annulus $S_{n,N}(z):=S_N(z) \setminus \mathring{S}_n(z)$, with the natural notation $S_{n,N} := S_{n,N}(0)$.

\begin{figure}
\begin{center}
\includegraphics[width=8.5cm]{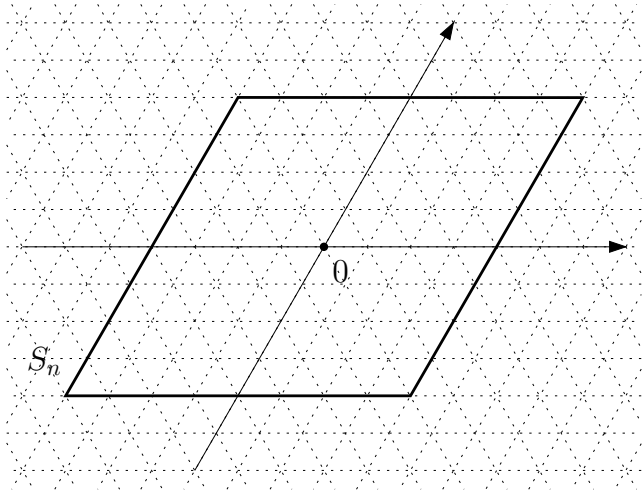}
\caption{We refer to oblique coordinates, and we denote by $S_n$ the ``box of size $n$''.}
\label{lattice}
\end{center}
\end{figure}

\subsubsection*{Connectivity properties}

Two sites $x$ and $y$ are said to be connected if there exists a black path, \emph{ie} a path consisting only of black sites, from $x$ to $y$. We denote it by $x \leadsto y$. Similarly, if there exists a white path from $x$ to $y$, these two sites are said to be $\ast$--connected, which we denote by $x \leadsto^{\ast} y$.

For notational convenience, we allow $y$ to be ``$\infty$'': we say that $x$ is connected to infinity ($x \leadsto \infty$) if there exists an infinite, self-avoiding and black path starting from $x$. We denote by
\begin{equation}
\theta(p) :=  \PP_p\big( 0 \leadsto \infty \big)
\end{equation}
the probability for $0$ (or any other site by translation invariance) to be connected to $\infty$.

To study the connectivity properties of a percolation realization, we often are interested in the connected components of black or white sites: the set of black sites connected to a site $x$ (empty if $x$ is white) is called the cluster of $x$, denoted by $C(x)$. We can define similarly $C^{\ast}(x)$ the white cluster of $x$. Note that $x \leadsto \infty$ is equivalent to the fact that $|C(x)|=\infty$.

If $A$ and $B$ are two sets of vertices, the notation 
$A \leadsto B$
is used to denote the event that some site in $A$ is connected to some site in $B$. If the connection is required to take place using exclusively the sites in some other set $C$, we write
$A \negthinspace \mathop{_{\ ^{\leadsto}\ }^{~~C}} \negthinspace B$.

\subsubsection*{Crossings}

A left-right (or horizontal) crossing of the parallelogram $[a_1,a_2] \times [b_1,b_2]$ is simply a black path connecting its left side to its right side. However, this definition implies that the existence of a crossing in two boxes sharing just a side are not completely independent: it will actually be more convenient to relax (by convention) the condition on its extremities. In other words, we request such a crossing path to be composed only of sites in $]a_1,a_2[ \times ]b_1,b_2[$ which are black, with the exception of its two extremities on the sides of the parallelogram, which can be either black or white. The existence of such a horizontal crossing is denoted by $\mathcal{C}_H([a_1,a_2] \times [b_1,b_2])$. We define likewise top-bottom (or vertical) crossings and denote their existence by $\mathcal{C}_V([a_1,a_2] \times [b_1,b_2])$, and also white crossings, the existence of which we denote by $\mathcal{C}^{\ast}_H$ and $\mathcal{C}^{\ast}_V$.

More generally, the same definition applies for crossings of annuli $S_{n,N}(z)$, from the internal boundary $\partial S_n(z)$ to the external one $\partial S_N(z)$, or even in more general domains $\mathcal{D}$, from one part of the boundary to another part.

\subsubsection*{Asymptotic behavior}

We use the standard notations to express that two quantities are asymptotically equivalent. For two positive functions $f$ and $g$, the notation $f \asymp g$ means that $f$ and $g$ remain of the same order of magnitude, in other words that there exist two positive and finite constants $C_1$ and $C_2$ such that $C_1 g \leq f \leq C_2 g$ (so that the ratio between $f$ and $g$ is bounded away from $0$ and $+ \infty$), while $f \approx g$ means that ${\log f} / {\log g} \to 1$ (``logarithmic equivalence'') -- either when $p \to 1/2$ or when $n \to \infty$, which will be clear from the context. This weaker equivalence is generally the one obtained for quantities behaving like power laws.

\subsection{General properties}

On the triangular lattice, it is known since \cite{Ke0} that percolation features a phase transition at $p=1/2$, called the \emph{critical point}: this means that
\begin{itemize}
\item[$\bullet$] When $p<1/2$, there is (a.s.) no infinite cluster (\emph{sub-critical} regime), or equivalently $\theta(p)=0$.
\item[$\bullet$] When $p>1/2$, there is (a.s.) an infinite cluster (\emph{super-critical} regime), or equivalently $\theta(p)>0$. Furthermore, the infinite cluster turns out to be \emph{unique} in this case.
\end{itemize}

In sub- and super-critical regimes, ``correlations'' decay very fast, this is the so-called \emph{exponential decay} property:
\begin{itemize}
\item[$\bullet$] For any $p<1/2$,
$$\exists C_1, C_2(p)>0, \quad \PP_p(0 \leadsto \partial S_n) \leq C_1 e^{-C_2(p) n}.$$
\item[$\bullet$] We can deduce from it that for any $p>1/2$,
\begin{align*}
\PP_p(0 \leadsto \partial S_n, |C(0)|<\infty) & \\
& \hspace{-2cm} \leq \PP_p(\text{$\exists$ white circuit surrounding $0$ and a site on $\partial S_n$})\\
& \hspace{-2cm} \leq C'_1 e^{-C'_2(p) n}
\end{align*}
for some $C'_1(p), C'_2(p) > 0$.
\end{itemize}

We would like to stress the fact that the speed at which these correlations vanish is governed by a constant $C_2$ which depends on $p$ -- it becomes slower and slower as $p$ approaches $1/2$. To study what happens near the critical point, we need to control this speed for different values of $p$: we will derive in Section \ref{expsection} an exponential decay property that is uniform in $p$.

The intermediate regime at $p=1/2$ is called the \emph{critical} regime. It is known for the triangular lattice that there is no infinite cluster at criticality: $\theta(1/2)=0$. Hence to summarize,
$$\theta(p) > 0 \quad \emph{iff} \quad p>1/2.$$
Correlations no longer decay exponentially fast in this critical  regime, but (as we will see) just like power laws. For instance, non trivial random scaling limits -- with fractal structures -- arise. This particular regime has some very strong symmetry property (conformal invariance) which allows to describe it very precisely.

\subsection{Some technical tools}

\subsubsection*{Monotone events}

We use the standard terminology associated with events: an event $A$ is increasing if ``it still holds when we add black sites'', and decreasing if it satisfies the same property, but when we add white sites.

Recall also the usual coupling of the percolation processes for different values of $p$: we associate to the different sites $x$ i.i.d. random variables $U_x$ uniform on $[0,1]$, and for any $p$, we obtain the measure $\PP_p$ by declaring each site $x$ to be black if $U_x \leq p$, and white otherwise. This coupling shows for instance that
$$p \mapsto \PP_p(A)$$
is a non-decreasing function of $p$ when $A$ is an increasing event. More generally, we can represent in this way any product measure $\PPP$.

\subsubsection*{Correlation inequalities}

The two most common inequalities for percolation concern monotone events: if $A$ and $B$ are increasing events, we have (\cite{BK,G_book})
\begin{itemize}
\item[1.] the Harris-FKG inequality:
$$\PP(A \cap B) \geq \PP(A) \PP(B).$$
\item[2.] the BK inequality:
$$\PP(A \circ B) \leq \PP(A) \PP(B),$$
$A \circ B$ meaning as usual that $A$ and $B$ occur ``disjointly''.
\end{itemize}

In the paper \cite{BK} where they proved the BK inequality, Van den Berg and Kesten also conjectured that this inequality holds in a more general fashion, for any pair of cylindrical events $A$ and $B$: if we define $A \square B$ the disjoint occurrence of $A$ and $B$ in this situation, we have
\begin{equation}
\PP(A \square B) \leq \PP(A) \PP(B).
\end{equation}
This was later proved by Reimer \cite{R}, and it is now known as Reimer's inequality. We will also use the following inequality:
\begin{equation}
\PP_{1/2}(A \circ B) \leq \PP_{1/2}(A \cap \tilde{B}),
\end{equation}
where $\tilde{B}$ denotes the ``opposite'' of $B$. This inequality is an intermediate step in Reimer's proof. On this subject, the reader can consult the nice review \cite{BCR}.

\subsubsection*{Russo's formula}

Russo's formula allows to study how probabilities of events vary when the percolation parameter $p$ varies. Recall that for an increasing event $A$, the event ``$v$ is pivotal for $A$'' is composed of the configurations $\omega$ such that if we make $v$ black, $A$ occurs, and if we make it white, $A$ does not occur. Note that by definition, this event is independent of the particular state of $v$. An analog definition applies for decreasing events.

\begin{theorem}[Russo's formula] \label{russo}
Let $A$ be an increasing event, depending only on the sites contained in some finite set $S$. Then
\begin{equation}
\frac{d}{dp} \PP_p(A) = \sum_{v \in S} \PP_p(\text{$v$ is pivotal for $A$}).
\end{equation}
\end{theorem}

We now quickly remind the reader how to prove this formula, since we will later (in Section \ref{nearcritical}) generalize it a little.
\begin{proof}
We allow the parameters $\hat{p}_v$ ($v \in S$) to vary independently, which amounts to consider the more general function $\mathcal{P}: \hat{p} = (\hat{p}_v)_{v \in S} \mapsto \PPP(A)$. This is clearly a smooth function (it is polynomial), and $\PP_p(A) = \mathcal{P}(p,\ldots, p)$. Now using the standard coupling, we see that for any site $w$, for a small variation $\epsilon>0$ in $w$,
\begin{equation}
\PPP^{+\epsilon}(A) - \PPP(A) = \epsilon \times \PPP(\text{$w$ is pivotal for $A$}),
\end{equation}
so that
$$\frac{\partial}{\partial \hat{p}_w} \PPP(A) = \PPP(\text{$w$ is pivotal for $A$}).$$
Russo's formula now follows readily by expressing $\frac{d}{dp} \PP_p(A)$ in terms of the previous partial derivatives:
\begin{align*}
\frac{d}{dp} \PP_p(A) & = \sum_{v \in S} \bigg( \frac{\partial}{\partial \hat{p}_v} \PPP(A) \bigg)_{\hat{p}= (p,\ldots,p)}\\
& = \sum_{v \in S} \PP_p(\text{$v$ is pivotal for $A$}).
\end{align*}
\end{proof}

\subsubsection*{Russo-Seymour-Welsh theory}

For symmetry reasons, we have:
\begin{equation}\label{sym}
\forall n, \quad \PP_{1/2}(\mathcal{C}_H([0,n] \times [0,n])) = 1/2.
\end{equation}
In other words, the probability of crossing a $n \times n$ box is the same on every scale. In particular, this probability is bounded from below: this is the starting point of the so-called Russo-Seymour-Welsh theory (see \cite{G_book,Ke_book}), that provides lower bounds for crossings in parallelograms of fixed aspect ratio $\tau \times 1$ ($\tau \geq 1$) in the ``hard direction''.

\begin{theorem}[Russo-Seymour-Welsh] \label{thm_RSW}
There exist universal non-decreasing functions $f_k(.)$ ($k \geq 2$), that stay positive on $(0,1)$ and verify: if for some parameter $p$ the probability of crossing a $n \times n$ box is at least $\delta_1$, then the probability of crossing a $k n \times n$ parallelogram is at least $f_k(\delta_1)$. Moreover, these functions can be chosen satisfying the additional property: $f_k(\delta) \to 1$ as $\delta \to 1$, with $f_k(1 - \epsilon) = 1 - C_k \epsilon^{\alpha_k} + o(\epsilon^{\alpha_k})$ for some $C_k, \alpha_k >0$.
\end{theorem}

If for instance $p>1/2$, we know that when $n$ gets very large, the probability $\delta_1$ of crossing a $n \times n$ rhombus becomes closer and closer to $1$, and the additional property tells that the probability of crossing a $k n \times n$ parallelogram tends to $1$ \emph{as a function of $\delta_1$}.

Combined with Eq.(\ref{sym}), the theorem entails:
\begin{corollary}
For each $k \geq 1$, there exists some $\delta_k > 0$ such that
\begin{equation}
\forall n, \quad \PP_{1/2}(\mathcal{C}_H([0,k n] \times [0,n])) \geq \delta_k.
\end{equation}
\end{corollary}

Note that only going from rhombi to parallelograms of aspect ratio slightly larger than $1$ is difficult. For instance, once the result is known for $2n \times n$ parallelograms, the construction of Figure \ref{RSW_construction} shows that we can take
\begin{equation} \label{RSW_k}
f_k(\delta) = \delta^{k-2} (f_2(\delta))^{k-1}.
\end{equation}

\begin{figure}
\begin{center}
\includegraphics[width=12cm]{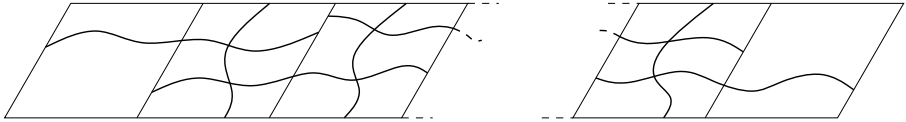}
\caption{\label{RSW_construction}This construction shows that we can take $f_k(\delta) = \delta^{k-2} (f_2(\delta))^{k-1}$.}
\end{center}
\end{figure}

\begin{proof}
The proof goes along the same lines as the one given by Grimmett \cite{G_book} for the square lattice. We briefly review it to indicate the small adaptations to be made on the triangular lattice. We hope Figure \ref{RSW_theory} will make things clear. For an account on RSW theory in a general setting, the reader should consult Chapter 6 of \cite{Ke_book}.

\begin{figure}
\begin{center}
\includegraphics[width=11cm]{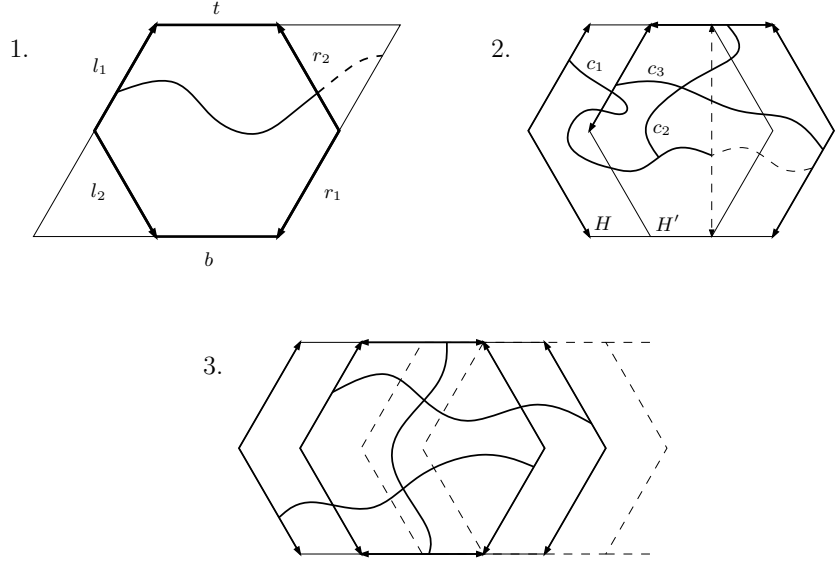}
\caption{\label{RSW_theory}Proof of the RSW theorem on the triangular lattice.}
\end{center}
\end{figure}

We work with hexagons, since they exhibit more symmetries. Note that a crossing of an $N \times N$ rhombus induces a left-right crossing of an hexagon of side length $N/2$ (see Figure \ref{RSW_theory}.1). We then apply the ``square-root trick'' -- used recurrently during the proof -- to the four events $\{i \leadsto j\}$: one of them occurs with probability at least $1 - (1-\delta)^{1/4}$. This implies that
\begin{equation}
\PP(l_1 \leadsto r_1) = \PP(l_2 \leadsto r_2) \geq (1 - (1-\delta)^{1/4})^2 =: \tau(\delta).
\end{equation}
(if $\PP(l_1 \leadsto r_1) = \PP(l_2 \leadsto r_2) \geq 1 - (1-\delta)^{1/4}$ we are OK, otherwise we just combine a crossing $l_1 \leadsto r_2$ and a crossing $l_2 \leadsto r_1$).

Now take two hexagons $H$, $H'$ as on Figure \ref{RSW_theory}.2 (with obvious notation for their sides). With probability at least $1 - (1-\delta)^{1/2}$ there exists a left-right crossing in $H$ whose last intersection with $l'_1 \cup l'_2$ is on $l'_2$. Assume this is the case, and condition on the lowest left-right crossing in $H$: with probability at least $1 - (1-\tau(\delta))^{1/2}$ it is connected to $t'$ in $H'$. We then use a crossing from $l'_1$ to $r'_1 \cup r'_2$, occurring with probability at least $1 - (1-\delta)^{1/2}$, to obtain
$$\PP(l_1 \cup l_2 \leadsto r'_1 \cup r'_2) \geq (1 - (1-\delta)^{1/2}) \times (1 - (1-\tau(\delta))^{1/2}) \times (1 - (1-\delta)^{1/2}).$$
The hard part is done: it now suffices to use $t$ successive ``longer hexagons'', and $t-1$ top-bottom crossings of regular hexagons, for $t$ large enough (see Figure \ref{RSW_theory}.3). We construct in such a way a left right-crossing of a $2N \times N$ parallelogram, with probability at least
\begin{equation}
f_2(\delta) := (1 - (1-\delta)^{1/2})^{2t} (1 - (1-\tau(\delta))^{1/2})^{2t - 1}.
\end{equation}
When $\delta$ tends to $1$, $\tau(\delta)$, and thus $f_2(\delta)$, also tend to $1$. Moreover, it is not hard to convince oneself that $f_2$ admits near $\delta = 1$ an asymptotic development of the form
\begin{equation}
f_2(1 - \epsilon) = 1 - C \epsilon^{1/8} + o(\epsilon^{1/8}).
\end{equation}
Eq.(\ref{RSW_k}) then provides the desired conclusion for any $k \geq 2$.
\end{proof}

\section{Near-critical percolation overview}

\subsection{Characteristic length}

We will use throughout the paper a certain ``characteristic length'' $L(p)$ defined in terms of crossing probabilities, or ``sponge-crossing probabilities''. This length is often convenient to work with, and it has been used in many papers concerning finite-size scaling, e.g. \cite {CCF,CCFS,BCKS1,BCKS2}.

Consider the rhombi $[0,n] \times [0,n]$. At $p=1/2$, $\PP_{p}(\mathcal{C}_H([0,n] \times [0,n])) = 1/2$. When $p<1/2$ (sub-critical regime), this probability tends to $0$ when $n$ goes to infinity, and it tends to $1$ when $p>1/2$ (super-critical regime). We introduce a quantity that measures the scale up to which these crossing probabilities remain bounded away from $0$ and $1$: for each $\epsilon_0 \in (0,1/2)$, we define
\begin{equation}
L_{\epsilon_0}(p)=\left\lbrace
\begin{array}{c l}
\min\{n \text{\: s.t. \:} \PP_p(\mathcal{C}_H([0,n] \times [0,n])) \leq \epsilon_0 \} & \text{when}\: p<1/2 \\[2mm]
\min\{n \text{\: s.t. \:} \PP_p(\mathcal{C}_H^*([0,n] \times [0,n])) \leq \epsilon_0 \} & \text{when}\: p>1/2 \\
\end{array} \right.
\end{equation}

Hence by definition,
\begin{equation}
\PP_p(\mathcal{C}_H([0,L_{\epsilon_0}(p)-1] \times [0,L_{\epsilon_0}(p)-1])) \geq \epsilon_0
\end{equation}
and
\begin{equation}
\PP_p(\mathcal{C}_H([0,L_{\epsilon_0}(p)] \times [0,L_{\epsilon_0}(p)])) \leq \epsilon_0
\end{equation}
if $p<1/2$, and the same with $\ast$'s if $p>1/2$.

Note that by symmetry, we also have directly $L_{\epsilon_0}(p)=L_{\epsilon_0}(1-p)$. Since $\PP_{1/2}(\mathcal{C}_H([0,n] \times [0,n]))$ is equal to $1/2$ on every scale, we will take the convention $L(1/2) = +\infty$, so that in the following, the expression ``for any $n \leq L(p)$'' must be interpreted as ``for any $n$'' when $p=1/2$. This convention is also consistent with the following property.

\begin{proposition}
For any fixed $\epsilon_0 \in (0,1/2)$, $L_{\epsilon_0}(p) \to + \infty$ when $p \to 1/2$.
\end{proposition}

\begin{proof}
Was it not the case, we could find an integer $N$ and a sequence $p_k \to 1/2$, say $p_k < 1/2$, such that for each $k$, $L_{\epsilon_0}(p_k) = N$, which would imply
$$\PP_{p_k}(\mathcal{C}_H([0,N] \times [0,N])) \leq \epsilon_0.$$
This contradicts the fact that
$$\PP_{p_k}(\mathcal{C}_H([0,N] \times [0,N])) \to 1/2,$$
the function $p \mapsto \PP_{p}(\mathcal{C}_H([0,N] \times [0,N]))$ being continuous (it is polynomial in $p$).
\end{proof}

\subsection{Russo-Seymour-Welsh type estimates}

When studying near-critical percolation, we will have to consider product measures $\PPP$ more general than simply the measures $\PP_p$ ($p \in [0,1]$), with associated parameters $\hat{p}_v$ which are allowed to depend on the site $v$, but have to remain between $p$ and $1-p$. We will say that $\PPP$ is ``between $\PP_p$ and $\PP_{1-p}$''.

The Russo-Seymour-Welsh theory implies that for each $k \geq 1$, there exists some $\delta_k = \delta_k(\epsilon_0) > 0$ (depending only on $\epsilon_0$) such that for all $p$, $\PPP$ between $\PP_p$ and $\PP_{1-p}$,
\begin{equation}
\forall n \leq L_{\epsilon_0}(p), \quad \PPP(\mathcal{C}_H([0,k n] \times [0,n])) \geq \delta_k,
\end{equation}
and for symmetry reasons this bound is also valid for horizontal white crossings.

These estimates for crossing probabilities will be the basic building blocks on which most further considerations are built. They imply that when $n$ is not larger than $L(p)$, things can still be compared to critical percolation: roughly speaking, $L(p)$ is the scale up to which percolation can be considered as ``almost critical''.

In the other direction, we will see in Section \ref{expsection} that $L(p)$ is also the scale at which percolation starts to look sub- or super-critical. Assume for instance that $p>1/2$, we know that
$$\PP_p(\mathcal{C}_H([0,L_{\epsilon_0}(p)] \times [0,L_{\epsilon_0}(p)])) \geq 1 - \epsilon_0.$$
Then using RSW (Theorem \ref{thm_RSW}), we get that
$$\PP_p(\mathcal{C}_H([0,2 L_{\epsilon_0}(p)] \times [0,L_{\epsilon_0}(p)])) \geq 1 - \epsilon_1,$$
where $1-\epsilon_1 = f_2(1-\epsilon_0)$ can be made arbitrarily close to $1$ by taking $\epsilon_0$ sufficiently small. This will be useful in the proof of Lemma \ref{explem} (but actually only for it).

\subsection{Outline of the paper}

In the following, we fix some value of $\epsilon_0$ in $(0,1/2)$. For notational convenience, we forget about the dependence on $\epsilon_0$. We will see later (Corollary \ref{lengths}) that the particular choice of $\epsilon_0$ is actually not relevant, in the sense that for any two $\epsilon_0$, $\epsilon'_0$, the corresponding lengths are of the same order of magnitude.

In Section \ref{sec_separation} we define the so-called arm events. On a scale $L(p)$, the RSW property, which we know remains true, allows to derive separation results for these arms. Section \ref{sec_critical} is devoted to critical percolation, in particular how arm exponents -- describing the asymptotic behavior of arm events -- can be computed. In Section \ref{nearcritical} we study how arm events are affected when we make vary the parameter $p$: if we stay on a scale $L(p)$, the picture does not change too much. It can be used to describe the characteristic functions, which we do in Section \ref{charac}. Finally, Section \ref{sec_remarks} concludes the paper with some remarks and possible applications.

With the exception of this last section, the organization follows the implication between the different results: each section depends on the previous ones. A limited number of results can however be obtained directly, we will indicate it clearly when this is the case.

\section{Arm separation} \label{sec_separation}

We will see that when studying critical and near-critical percolation, certain exceptional events play a central role: the arm events, referring to the existence of some number of crossings (``arms'') of the annuli $S_{n,N}$ ($n < N$), the color of each crossing (black or white) being prescribed. These events are useful because they can be combined together, and they will prove to be instrumental for studying more complex events. Their asymptotic behavior can be described precisely using $SLE_6$ (see next section), allowing to derive further estimates, especially on the characteristic functions.

\subsection{Arm events}

Let us consider an integer $j \ge 1$. A color sequence $\sigma$ is a sequence $(\sigma_1,\ldots,\sigma_j)$ of ``black'' and ``white'' of length $j$. We use the letters ``$W$'' and ``$B$'' to encode the colors: the sequence $(\text{black},\text{white},\text{black})$ is thus denoted by ``$BWB$''. Only the cyclic order of the arms is relevant, and we identify two sequences if they are the same up to a cyclic permutation: for instance, the two sequences ``$BWBW$'' and ``$WBWB$'' are the same, but they are different from ``$BBWW$''. The resulting set is denoted by $\Sj$. For any color sequence $\sigma$, we also introduce $\tilde{\sigma} = (\tilde{\sigma}_1,\ldots,\tilde{\sigma}_j)$ the inverted sequence, where each color is replaced by its opposite.

For any two positive integers $n \leq N$, we define the event 
\begin{equation}
A_{j,\sigma}(n,N) := \{ \partial S_n \leadsto^{j,\sigma} \partial S_N\}
\end{equation}
that there exist $j$ disjoint monochromatic arms in the annulus $S_{n,N}$, whose colors are those prescribed by $\sigma$ (when taken in counterclockwise order). We denote such an ordered set of crossings by $\mathcal{C} = \{c_i\}_{1 \leq i \leq j}$, and we say it to be ``$\sigma$-colored''. Recall that by convention, we have relaxed the color prescription for the extremities of the $c_i$'s. Hence for $j=1$ and $\sigma=B$, $A_{j,\sigma}(0,N)$ just denotes the existence of a black path $0 \leadsto \partial S_N$.

\begin{figure}
\begin{center}
\includegraphics[width=10cm]{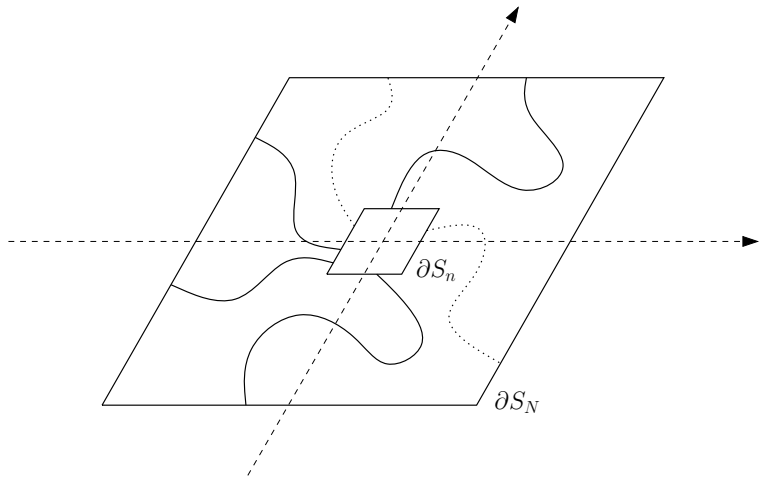}
\caption{\label{arms}The event $A_{6,\sigma}(n,N)$, with $\sigma=BBBWBW$.}
\end{center}
\end{figure}

Note that a combinatorial objection due to discreteness can arise: if $j$ is too large compared to $n$, the event $A_{j,\sigma}(n,N)$ can be void, just because the arms do not have enough space on $\partial S_{n}$ to arrive all together. For instance $A_{j,\sigma}(0,N)=\varnothing$ if $j \geq 7$. In fact, we just have to check that $n$ is large enough so that the number of sites touching the exterior of $|\partial S_{n}|$ (\emph{ie} $|\partial S_{n+1}|$ with the acute corners removed) is at least $j$: if this is true, we can then draw straight lines heading toward the exterior. For each positive integer $j$, we thus introduce $n_0(j)$ the least such nonnegative integer, and we have
$$\forall N \geq n_0(j), \quad A_{j,\sigma}(n_0(j),N) \neq \varnothing.$$
Note that $n_0(j)=0$ for $j=1,\ldots,6$, and that $n_0(j) \leq j$. For asymptotics, the exact choice of $n$ is not relevant since anyway, for any fixed $n_1, n_2 \geq n_0(j)$,
$$\PPP(A_{j,\sigma}(n_1,N)) \asymp \PPP(A_{j,\sigma}(n_2,N)).$$

\begin{remark}
Note that Reimer's inequality implies that for any two integers $j$, $j'$, and two color sequences $\sigma$, $\sigma'$ of these lengths, we have:
\begin{equation}
\label{csq_Reimer}
\PPP(A_{j+j',\sigma\sigma'}(n,N)) \leq \PPP(A_{j,\sigma}(n,N)) \PPP(A_{j',\sigma'}(n,N))
\end{equation}
for any $\PPP$, $n \leq N$ (denoting by $\sigma\sigma'$ the concatenation of $\sigma$ and $\sigma'$).
\end{remark}

\subsection{Well-separateness}

We now impose some restrictions on the events $A_{j,\sigma}(n,N)$. Our main goal is to prove that we can separate macroscopically (the extremities of) any sequence of arms: with this additional condition, the probability of $A_{j,\sigma}(n,N)$ does not decrease from more than a (universal) constant factor. This result is not really surprising, but we will need it recurrently for technical purposes.

Let us now give a precise meaning to the property of being ``separated'' for sets of crossings. In the following, we will actually consider crossings in different domain shapes. We first state the definition for a parallelogram of fixed ($1 \times \tau$) aspect ratio, and explain how to adapt it in other cases.

We first require that the extremities of these crossings are distant from each other. We also need to add a condition ensuring that the crossings can easily be extended: we impose the existence of ``free spaces'' at their extremities, which will allow then to construct longer extensions. This leads to the following definition, similar to Kesten's ``fences'' \cite{Ke4}.
\begin{definition}
Consider some $M \times \tau M$ parallelogram $R = [a_1,a_1+M] \times [b_1,b_1+\tau M]$, and $\mathcal{C} = \{c_i\}_{1 \leq i \leq j}$ a ($\sigma$-colored) set of $j$ disjoint left-right crossings. Introduce $z_i$ the extremity of $c_i$ on the right side of the parallelogram, and for some $\eta \in (0,1]$, the parallelogram $r_i = z_i + [0,\sqrt{\eta} M] \times [-\eta M,\eta M]$, attached to $R$ on its right side.

We say that $\mathcal{C}$ is \emph{well-separated} at scale $\eta$ (on the right side) if the two following conditions are satisfied:
\begin{enumerate}[1.]
\item The extremity $z_i$ of each crossing is not too close from the other ones:
\begin{equation}
\forall i \neq j, \quad \text{dist}(z_i,z_j) \geq 2 \sqrt{\eta} M,
\end{equation}
nor from the top and bottom right corners $Z_+, Z_-$ of $R$:
\begin{equation}
\forall i, \quad \text{dist}(z_i,Z_{\pm}) \geq 2 \sqrt{\eta} M.
\end{equation}

\item Each $r_i$ is crossed vertically by some crossing $\tilde{c}_i$ of the same color as $c_i$, and
\begin{equation}
c_i \leadsto \tilde{c}_i \quad  \text{in $\mathring{S}_{\sqrt{\eta} M}(z_i)$}.
\end{equation} 
\end{enumerate}
\end{definition}

\begin{figure}
\begin{center}
\includegraphics[width=7cm]{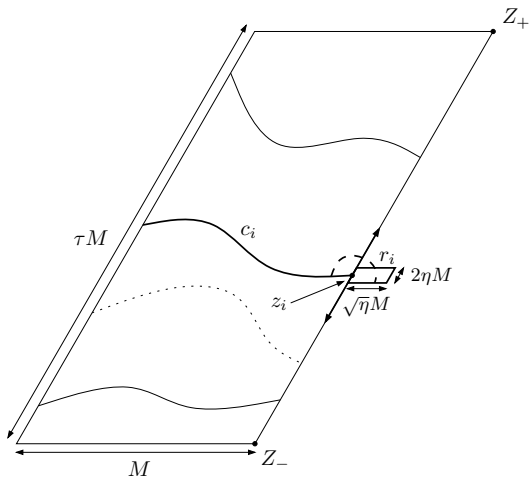}
\caption{Well-separateness for a set of crossings $\mathcal{C} = \{c_i\}$.}
\label{well_separateness}
\end{center}
\end{figure}

For the second condition, we of course require the path connecting $c_i$ and $\tilde{c}_i$ to be of the same color as these two crossings. The crossing $\tilde{c}_i$ is thus some small extension of $c_i$ on the right side of $R$. The free spaces $r_i$ will allow us to use locally an FKG-type inequality to further extend the $c_i$'s on the right.

\begin{definition}
We say that a set $\mathcal{C} = \{c_i\}_{1 \leq i \leq j}$ of $j$ disjoint left-right crossings of $R$ can be \emph{made well-separated} on the right side if there exists another set $\mathcal{C}' = \{c'_i\}_{1 \leq i \leq j}$ of $j$ disjoint crossings that is well-separated on the right side, such that $c'_i$ has the same color as $c_i$, and the same extremity on the left side.
\end{definition}

The same definitions apply for well-separateness on the left side, and also for top-bottom crossings. Consider now a set of crossings of an annulus $S_{n,N}$. We can divide this set into four subsets, according to the side of $\partial S_N$ on which they arrive. Take for instance the set of crossings arriving on the right side: we say it to be \emph{well-separated} if, as before, the extremities of these crossings on $\partial S_N$ are distant from each other and from the top-right and bottom-right corners, and if there exist free spaces to extend them. Then, we say that a set of crossings of $S_{n,N}$ is well-separated on the external boundary if each of the four previous sets is itself well-separated. Note that requiring the extremities to be not too close from the corners ensures that they are not too close from the extremities of the crossings arriving on other sides either. We take the same definition for the internal boundary $\partial S_n$: in this case, taking the extremities away from the corners also ensures that the free spaces are included in $S_n$ and do not intersect each other.

We are in position to define our first sub-event of $A_{j,\sigma}(n,N)$: for any $\eta, \eta' \in (0,1)$,
\begin{equation}
\At_{j,\sigma}^{\eta / \eta'}(n,N) := \{ \partial S_n \leadsto_{j,\sigma}^{\eta / \eta'} \partial S_N\}
\end{equation}
denotes the event $A_{j,\sigma}(n,N)$ with the additional condition that the set of $j$ arms is well-separated at scale $\eta$ on $\partial S_n$, and at scale $\eta'$ on $\partial S_N$.

\bigskip

We can even prescribe the ``landing areas'' of the different arms, \emph{ie} the position of their extremities. We introduce for that some last definition:
\begin{definition}
Consider $\partial S_N$ for some integer $N$: a \emph{landing sequence} $\{I_i\}_{1 \leq i \leq j}$ on $\partial S_N$ is a sequence of disjoint sub-intervals $I_1,\ldots,I_j$ on $\partial S_N$ in counterclockwise order. It is said to be $\eta$-separated if for each $i$\footnote{As usual, we consider cyclic indices, so that here for instance $I_{j+1} = I_1$.},
\begin{enumerate}[1.]
\item $\text{dist}(I_i,I_{i+1}) \geq 2 \sqrt{\eta} N$,
\item and $\text{dist}(I_i,Z_{\pm}) \geq 2 \sqrt{\eta} N$.

\noindent It is called a landing sequence of size $\eta$ if the additional property
\item $\text{length}(I_i) \geq \eta N$
\end{enumerate}
is also satisfied.
\end{definition}

We identify two landing sequences on $\partial S_N$ and $\partial S_{N'}$ if they are identical up to a dilation. This leads to the following sub-event of $\At_{j,\sigma}^{\eta / \eta'}(n,N)$: for two landing sequences $I=\{I_i\}_{1 \leq i \leq j}$ and $I'=\{I'_i\}_{1 \leq i' \leq j}$,
\begin{equation}
\Att_{j,\sigma}^{\eta,I / \eta',I'}(n,N) := \{ \partial S_n \leadsto_{j,\sigma}^{\eta,I / \eta',I'} \partial S_N\}
\end{equation}
denotes the event $\At_{j,\sigma}^{\eta / \eta'}(n,N)$, with the additional requirement on the set of crossings $\{c_i\}_{1 \leq i \leq j}$ that for each $i$, the extremities $z_i$ and $z'_i$ of $c_i$ on (respectively) $\partial S_n$ and $\partial S_N$ satisfy $z_i \in I_i$ and $z'_i \in I'_i$.

We will also have use for another intermediate event between $A$ and $\Att$: $\Ab_{j,\sigma}^{I/I'}(n,N)$, for which we impose the landing areas $I / I'$ of the different arms without requiring the existence of the free spaces. We do not ask \emph{a priori} the sub-intervals to be $\eta$-separated either, just to be disjoint. Note that if they are $\eta / \eta'$-separated, then the extremities of the different crossings will be $\eta / \eta'$-separated too.

To summarize:
\begin{displaymath}
\xymatrix{
& A_{j,\sigma}(n,N) = \{\text{$j$ arms $\partial S_n \leadsto \partial S_N$, color $\sigma$}\} \ar[dl]|{\text{separated at scale $\eta / \eta'$ + small extensions}} \ar[dr]|{\text{landing areas $I / I'$}} & \\
\At_{j,\sigma}^{\eta / \eta'}(n,N) \ar[dr]|{\text{landing areas $I / I'$}} & & \Ab_{j,\sigma}^{I / I'}(n,N) \ar[dl]|{\text{small extensions (if $I / I'$ are $\eta / \eta'$-separated)}} \\
& \Att_{j,\sigma}^{\eta,I / \eta',I'}(n,N) & }
\end{displaymath}

\begin{remark}
If we take for instance alternating colors ($\bar{\sigma} = BWBW$), and as landing areas $\bar{I}_1, \ldots ,\bar{I}_4$ the (resp.) right, top, left and bottom sides of $\partial S_N$, the $4$-arm event $\Ab_{4,\bar{\sigma}}^{./\bar{I}}(0,N)$ (the ``.'' meaning that we do not put any condition on the internal boundary) is then the event that $0$ is pivotal for the existence of a left-right crossing of $S_N$.
\end{remark}

\subsection{Statement of the results}

\subsubsection*{Main result}

Our main separation result is the following:

\begin{theorem} \label{armsep}
Fix an integer $j \geq 1$, some color sequence $\sigma \in \Sj$ and $\eta_0,\eta'_0 \in (0,1)$. Then we have
\begin{equation}
\PPP\big(\Att_{j,\sigma}^{\eta,I / \eta',I'}(n,N)\big) \asymp \PPP\big(A_{j,\sigma}(n,N)\big)
\end{equation}
uniformly in all landing sequences $I / I'$ of size $\eta / \eta'$, with $\eta \geq \eta_0$ and $\eta' \geq \eta'_0$, $p$, $\PPP$ between $\PP_p$ and $\PP_{1-p}$, $n \leq N \leq L(p)$.
\end{theorem}

\subsubsection*{First relations}

Before turning to the proof of this theorem, we list some direct consequences of the RSW estimates that will be needed.

\begin{proposition} \label{easy_prop}
Fix $j \geq 1$, $\sigma \in \Sj$ and $\eta_0,\eta'_0 \in (0,1)$.

\begin{enumerate}
\item \emph{``Extendability''}: We have
$$\PPP\big(\Att_{j,\sigma}^{\eta,I / \tilde{\eta}',\It'}(n,2N)\big), \: \PPP\big(\Att_{j,\sigma}^{\tilde{\eta},\It / \eta',I'}(n/2,N)\big) \asymp \PPP\big(\Att_{j,\sigma}^{\eta,I / \eta',I'}(n,N)\big)$$
uniformly in $p$, $\PPP$ between $\PP_p$ and $\PP_{1-p}$, $n \leq N \leq L(p)$, and all landing sequences $I / I'$ (resp. $\It / \It'$) of size $\eta / \eta'$ (resp. $\tilde{\eta} / \tilde{\eta}'$) larger than $\eta_0 / \eta'_0$. In other words: ``once well-separated, the arms can easily be extended''. \label{extend}

\item \emph{``Quasi-multiplicativity''}: We have
$$\PPP(\Att_{j,\sigma}^{. / \eta,I_{\eta}}(n_1,n_2/4)) \PPP(\Att_{j,\sigma}^{\eta',I_{\eta'} / .}(n_2,n_3)) \asymp \PPP(A_{j,\sigma}(n_1,n_3))$$ 
uniformly in $p$, $\PPP$ between $\PP_p$ and $\PP_{1-p}$, $n_0(j) \leq n_1 < n_2 < n_3 \leq L(p)$ with $n_2 \geq 4 n_1$, and all landing sequences $I / I'$ of size $\eta / \eta'$ larger than $\eta_0 / \eta'_0$.
\label{quasi_mult2}

\item For any $\eta,\eta' >0$, there exists a constant $C = C(\eta,\eta')>0$ with the following property: for any $p$, $\PPP$ between $\PP_p$ and $\PP_{1-p}$, $n \leq N \leq L(p)$, there exist $I$ and $I'$ of size $\eta$ and $\eta'$ (they may depend on all the parameters mentioned) such that
$$\PPP\big(\Att_{j,\sigma}^{\eta,I / \eta',I'}(n,N)\big) \geq C \: \PPP\big(\At_{j,\sigma}^{\eta,\eta'}(n,N)\big).$$
\label{existI}

\end{enumerate}

\end{proposition}

\begin{proof}
The proof relies of gluing arguments based on RSW constructions. However, the events considered are not monotone when $\sigma$ is non-constant (there is at least one black arm and one white arm). We will thus need a slight generalization of the FKG inequality for events ``locally monotone''.

\begin{lemma}
Consider $A^+$, $\tilde{A}^+$ two increasing events, and $A^-$, $\tilde{A}^-$ two decreasing events. Assume that there exist three disjoint finite sets of vertices $\mathcal{A}$, $\mathcal{A}^+$ and $\mathcal{A}^-$ such that $A^+$, $A^-$, $\tilde{A}^+$ and $\tilde{A}^-$ depend only on the sites in, respectively, $\mathcal{A} \cup \mathcal{A}^+$, $\mathcal{A} \cup \mathcal{A}^-$, $\mathcal{A}^+$ and $\mathcal{A}^-$. Then we have
\begin{equation}
\PPP(\tilde{A}^+ \cap \tilde{A}^- | A^+ \cap A^-) \geq \PPP(\tilde{A}^+) \PPP(\tilde{A}^-)
\end{equation}
for any product measure $\PPP$.
\end{lemma}

\begin{proof}
Conditionally on the configuration $\omega_{\mathcal{A}}$ in $\mathcal{A}$, the events $A^+ \cap \tilde{A}^+$ and $A^- \cap \tilde{A}^-$ are independent, so that
$$\PPP(A^+ \cap \tilde{A}^+ \cap A^- \cap \tilde{A}^- | \omega_{\mathcal{A}}) = \PPP(A^+ \cap \tilde{A}^+ | \omega_{\mathcal{A}}) \PPP(A^- \cap \tilde{A}^- | \omega_{\mathcal{A}}).$$
The FKG inequality implies that
\begin{align*}
\PPP(A^+ \cap \tilde{A}^+ | \omega_{\mathcal{A}}) & \geq \PPP(A^+ | \omega_{\mathcal{A}}) \PPP(\tilde{A}^+ | \omega_{\mathcal{A}})\\
& = \PPP(A^+ | \omega_{\mathcal{A}}) \PPP(\tilde{A}^+)
\end{align*}
and similarly with $A^-$ and $\tilde{A}^-$. Hence,
\begin{align*}
\PPP(A^+ \cap \tilde{A}^+ \cap A^- \cap \tilde{A}^- | \omega_{\mathcal{A}}) & \geq \PPP(A^+ | \omega_{\mathcal{A}}) \PPP(\tilde{A}^+) \PPP(A^- | \omega_{\mathcal{A}}) \PPP(\tilde{A}^-)\\
& = \PPP(A^+ \cap A^- | \omega_{\mathcal{A}}) \PPP(\tilde{A}^+) \PPP(\tilde{A}^-).
\end{align*}
The conclusion follows by summing over all configurations $\omega_{\mathcal{A}}$.
\end{proof}

Once this lemma at our disposal, items \ref{extend}. and \ref{quasi_mult2}. are straightforward. For item \ref{existI}., we consider a consider a covering of $\partial S_n$ (resp. $\partial S_N$) with at most $8 \eta^{-1}$ (resp. $8 \eta'^{-1}$) intervals $(I)$ of length $\eta$ (resp. $(I')$ of length $\eta'$). Then for some $I$, $I'$,
$$\PPP\big(\Att_{j,\sigma}^{\eta,I / \eta',I'}(n,N)\big) \geq (8 \eta^{-1})^{-1} (8 \eta'^{-1})^{-1} \PPP\big(\At_{j,\sigma}^{\eta,\eta'}(n,N)\big).$$
\end{proof}

We also have the following a-priori bounds for the arm events:
\begin{proposition}
Fix some $j \geq 1$, $\sigma \in \Sj$ and $\eta_0,\eta'_0 \in (0,1)$. Then there exist some exponents $0 < \alpha_j, \alpha' < \infty$, as well as constants $0 < C_j,C' < \infty$, such that
\begin{equation} \label{apriori}
C_j \bigg(\frac{n}{N}\bigg)^{\alpha_j} \leq \PPP\big(\Att_{j,\sigma}^{\eta,I / \eta',I'}(n,N)\big) \leq C' \bigg(\frac{n}{N}\bigg)^{\alpha'}
\end{equation}
uniformly in $p$, $\PPP$ between $\PP_p$ and $\PP_{1-p}$, $n \leq N \leq L(p)$, and all landing sequences $I / I'$ of size $\eta / \eta'$ larger than $\eta_0 / \eta'_0$.
\end{proposition}

The lower bound comes from iterating item \ref{extend}. The upper bound can be obtained by using concentric annuli: in each of them, RSW implies that there is a probability bounded away from zero to observe a black circuit, preventing the existence of a white arm (consider a white circuit instead if $\sigma = BB \ldots B$).

\subsection{Proof of the main result}

Assume that $A_{j,\sigma}(n,N)$ is satisfied: our goal is to link this event to the event $\Att_{j,\sigma}^{\eta_0,I_{\eta_0} / \eta'_0,I_{\eta'_0}}(n,N)$, for some fixed scales $\eta_0, \eta'_0$.

\begin{proof}
First note that it suffices to prove the result for $n$, $N$ which are powers of two: then we would have, if $k$, $K$ are such that $2^{k-1} < n \leq 2^k$ and $2^K \leq n < 2^{K+1}$,
\begin{align*}
\PPP\big(A_{j,\sigma}(n,N)\big) & \leq \PPP\big(A_{j,\sigma}(2^k,2^K)\big)\\
& \leq C_1 \PPP\big(\Att_{j,\sigma}^{\eta,I / \eta',I'}(2^k,2^K)\big)\\
& \leq C_2 \PPP\big(\Att_{j,\sigma}^{\eta,I / \eta',I'}(n,N)\big).\\
\end{align*}

We have to deal with the extremities of the $j$ arms on the internal boundary $\partial S_n$, and on the external boundary $\partial S_N$.

\subsubsection*{1. External extremities}

Let us begin with the external boundary. In the course of proof, we will have use for the intermediate event $\At_{j,\sigma}^{. / \eta'}(n,N)$ that there exists a set of $j$ arms that is well-separated on the external side $\partial S_N$ only, and also the event $\Att_{j,\sigma}^{. / \eta',I'}(n,N)$ associated to some landing sequence $I'$ on $\partial S_N$. Each of the $j$ arms induces in $S_{2^{K-1},2^K}$ a crossing of one of the four U-shaped regions $U_{2^{K-1}}^{1,\textrm{ext}},\ldots,U_{2^{K-1}}^{4,\textrm{ext}}$ depicted in Figure \ref{U_shapes_ext}. The ``ext'' indicates that a crossing of such a region connects the two marked parts of the boundary. For the internal extremities, we will use the same regions, but we distinguish different parts of the boundary. The key observation is the following:

\begin{figure}
\begin{center}
\includegraphics[width=11cm]{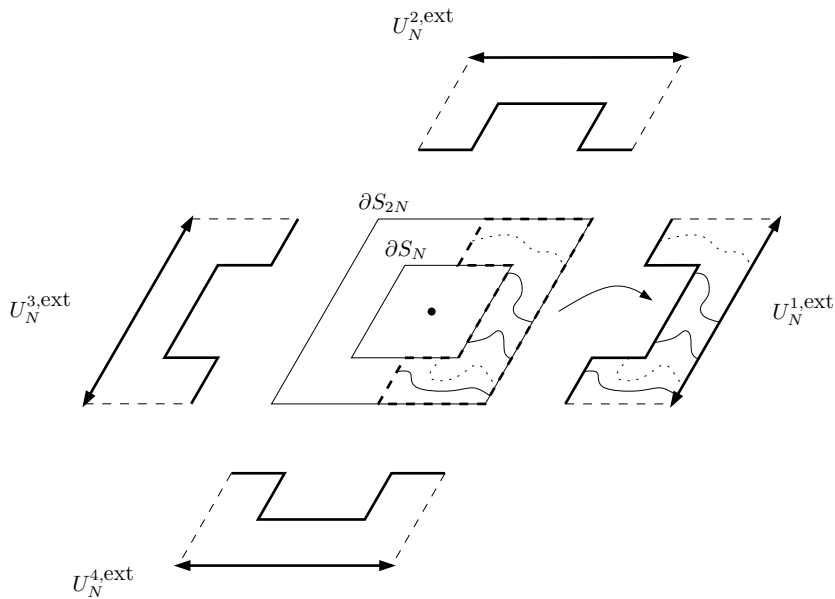}
\caption{\label{U_shapes_ext}The four U-shaped regions that we use for the external extremities.}
\end{center}
\end{figure}

\noindent \textbf{In a U-shaped region, any set of disjoint crossings can be made well-separated with high probability.}

More precisely, if we take such an $N \times 4N$ domain, the probability that any set of disjoint crossings can be made $\eta$-well-separated (on the external boundary) can be made arbitrarily close to $1$ by choosing $\eta$ sufficiently small, uniformly in $N$. We prove the following lemma, which implies that on \emph{every} scale, with very high probability the $j$ arms can be made well-separated.

\begin{lemma} \label{well_sep}
For any $\delta>0$, there exists a size $\eta(\delta) >0$ such that for any $p$, any $\PPP$ between $\PP_p$ and $\PP_{1-p}$ and any $N \leq L(p)$: in the domain $U_{N}^{1,\textrm{ext}}$,
\begin{equation}
\PPP(\text{Any set of disjoint crossings can be made $\eta$-well-separated}) \geq 1 - \delta.
\end{equation}
\end{lemma}

\begin{proof}
First we note that there cannot be too many disjoint crossings in $U_{N}^{1,\textrm{ext}}$. Indeed, the probability of crossing this domain is less than some $1 - \delta'$ (by RSW): combined with the BK inequality, this implies that the probability of observing at least $h$ crossings is less than
\begin{equation}
(1-\delta')^h.
\end{equation}
We thus take $T$ such that this quantity is less than $\delta/4$.

Consider for the moment any $\eta \in (0,1)$ (we will see during the proof how to choose it). We note that we can put disjoint annuli around $Z_-$ and $Z_+$ to prevent crossings from arriving there. Consider $Z_-$ for instance, and look at the disjoint annuli centered on $Z_-$ of the form $S_{2^{l-1},2^l}(Z_-)$, with $\eta^{3/8} \leq 2^{l-1} < 2^l \leq \sqrt{\eta}$ (see Figure \ref{Separation}). We can take at least $- C_4 \log \eta$ such disjoint annuli for some universal constant $C_4 > 0$, and with probability at least $1-(1-\delta'')^{-C_4 \log \eta}$ there exists a black circuit in one of the annuli. Consider then the annuli $S_{2^{l-1},2^l}(Z_-)$, with $\eta^{1/4} \leq 2^{l-1} < 2^l \leq \eta^{3/8}$: with probability at least $1- (1-\delta'')^{-C'_4 \log \eta}$ we observe a white circuit in one of them. If two circuits as described exist, we say that $Z_-$ is ``protected''. The same reasoning applies for $Z_+$.

Consider now the following construction: take $c_1$ the lowest (\emph{ie} closest to the bottom side) monochromatic crossing (which can be either black or white), then $c_2$ the lowest monochromatic crossing disjoint from $c_1$, and so on. The process stops after $t$ steps, and we denote by $\mathcal{C} = \{c_{u}\}_{1 \leq u \leq t}$ the set of crossings so-obtained. Of course, $\mathcal{C}$ can be void: we set $t=0$ in this case. We have
\begin{equation}
\PP(t \geq T) \leq (1-\delta')^T \leq \delta/4
\end{equation}
by definition of $T$. We denote by $z_u$ the extremity of $c_u$ on the right side, and by $\sigma_u \in \{B,W\}$ its color.

In order to get some independence and be able to apply the previous construction around the extremities of the crossings, we condition on the successive crossings. Consider some $u \in \{1,\ldots,T\}$ and some ordered sequence of crossings $\tilde{c}_1, \tilde{c}_2, \ldots, \tilde{c}_u$, together with colors $\tilde{\sigma}_1, \tilde{\sigma}_2, \ldots, \tilde{\sigma}_u$. The event $E_u := \{ t \geq u \text{ and } c_v = \tilde{c}_v, \: \sigma_v=\tilde{\sigma}_v \text{ for any } v \in \{1,\ldots,u\} \}$ is independent from the status of the sites above $\tilde{c}_u$. Hence, if we condition on $E_u$, percolation there remains unbiased and we can use the RSW theorem.

We now do the same construction as before. Look at the disjoint annuli centered on $z_u$ of the form $S_{2^{l-1},2^l}(z_u)$, with $\eta^{3/8} \leq 2^{l-1} < 2^l \leq \sqrt{\eta}$ on one hand, and with $\eta^{1/4} \leq 2^{l-1} < 2^l \leq \eta^{3/8}$ on the other hand. Assume for instance that $\tilde{\sigma}_u=B$. With probability at least $1-(1-\delta'')^{-C''_4 \log \eta}$ we observe a white circuit in one of the annuli in the first set, preventing other disjoint black crossings to arrive near $z_u$, and also a black one in the second set, preventing white crossings to arrive. Moreover, by considering a black circuit in the annuli $S_{2^{l-1},2^l}(z_u)$ with $\sqrt{\eta} \leq 2^{l-1} < 2^l \leq \eta^{3/4}$, we can construct a small extension of $c_u$. If the three circuits described exist, $c_u$ is said to be ``protected from above''. Summing over all possibilities for $\tilde{c}_i$, $\tilde{\sigma}_i$ ($1 \leq i \leq u$), we get that for some $C'''_4$,
\begin{equation}
\PP(\text{$t \geq u$ and $c_u$ is not protected from above}) \leq (1-\delta'')^{-C'''_4 \log \eta}.
\end{equation}

\begin{figure}
\begin{center}
\includegraphics[width=8cm]{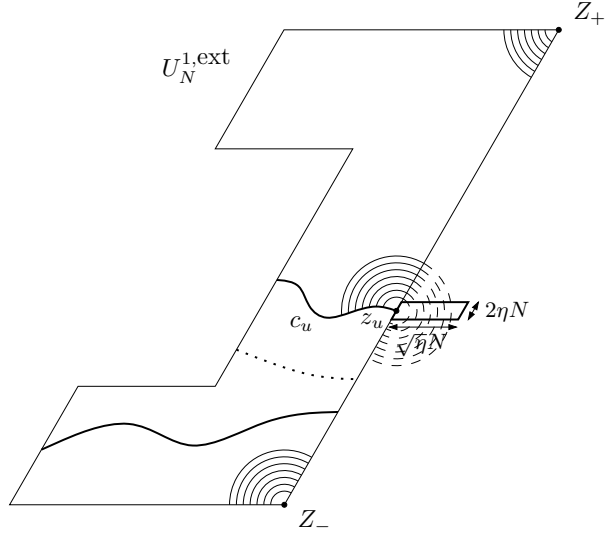}
\caption{\label{Separation}We apply RSW in concentric annuli around $Z_-$ and $Z_+$, and then around the extremity $z_u$ of each crossing $c_u$.}
\end{center}
\end{figure}

Now for our set of crossings $\mathcal{C}$,
\begin{align*}
\PP(\text{$\mathcal{C}$ is not $\eta$-well-separated}) & \\
& \hspace{-3cm} \leq \PP(t \geq T) + \sum_{u=1}^{T-1} \PP(\text{$t \geq u$ and $c_u$ is not protected from above} )\\
& \hspace{-2.5cm} + \PP(\text{$Z_-$ is not protected}) + \PP(\text{$Z_+$ is not protected}).
\end{align*}
First, each term in the sum, as well as the last two terms, are less than $(1-\delta'')^{-C'''_4 \log \eta}$. We also have
$\PP(t \geq T) \leq \delta/4$, so that the right-hand side is at most
\begin{equation}
(T+1) (1-\delta'')^{-C'''_4 \log \eta} + \frac{\delta}{4}.
\end{equation}
It is less than $\delta$ if we choose $\eta$ sufficiently small ($T$ is fixed).

We now assume that $\mathcal{C}$ is $\eta$-well-separated, and prove that any other set $\mathcal{C}' = \{c'_{u}\}_{1 \leq u \leq t'}$ of $t'$ ($\leq t$) disjoint crossings (we take it ordered) can also be made $\eta$-well-separated. For that purpose, we replace recursively the tip of each $c'_v$ by the tip of one of the $c_u$'s. If we take $c'_1$ for instance, it has to cross at least one of the $c_v$ (by maximality of $\mathcal{C}$). Let us call $c_{v_1}$ the lowest one: still by maximality, $c'_1$ cannot go below it. Take the piece of $c'_1$ between its extremity $z'_1$ and its last intersection $a_1$ with $c_{v_1}$, and replace it with the corresponding piece of $c_{v_1}$: this gives $c''_1$. This new crossing has the same extremity as $c_{v_1}$ on the right-side, and it is not hard to check that it is connected to the small extension $\tilde{c}_{v_1}$ of $c_{v_1}$ on the external side. Indeed, this extension is connected by a path that touches $c_{v_1}$ in, say, $b_1$: either $b_1$ is between $a_1$ and $z_1$, in which case $c''_1$ is automatically connected to $\tilde{c}_{v_1}$, otherwise $c'_1$ has to cross the connecting path before $a_1$ and $c''_1$ is also connected to $\tilde{c}_{v_1}$.

Consider then $c_2$, and $c_{v_2}$ the lowest crossing it intersects: necessarily $v_2 > v_1$ (since $c_1$ stays above $c_{v_1}$), and the same reasoning applies. The claim follows by continuing this procedure until $c'_{t'}$.

\end{proof}

\noindent \textbf{The arms are well-separated with positive probability.}

The idea is then to ``go down'' in successive concentric annuli, and to apply the lemma in each of them. We work with two different scales of separation:
\begin{itemize}
\item[$\bullet$] a fixed (macroscopic) scale $\eta'_0$ that we will use to extend arms, associated to a \emph{constant} extension cost.

\item[$\bullet$] another scale $\eta'$ which is \emph{very} small ($\eta' \ll \eta'_0$), so that the $j$ arms can be made well-separated at scale $\eta'$ with very high probability.
\end{itemize}

The proof goes as follows. Take some $\delta > 0$ very small (we will see later how small), and some $\eta' > 0$ associated to it by the lemma. We start from the scale $\partial S_{2^K}$ and look at the crossings induced by the $j$ arms. The previous lemma implies that with very high probability, these $j$ arms can be modified in $S_{2^{K-1},2^K}$ so that they are $\eta'$-well-separated. Otherwise, we go down to the next annulus: there still exist $j$ arms, and what happens in $S_{2^{K-1},2^K}$ is independent of what happens in $S_{2^{K-1}}$. On each scale, we have a very low probability to fail, and once the arms are separated on scale $\eta'$, we go backwards by using the scale $\eta'_0$, for which the cost of extension is constant.

More precisely, after one step we get $$A_{j,\sigma}(2^k,2^K) \subseteq \At_{j,\sigma}^{. / \eta'}(2^k,2^K) \cup \Big(\{\text{One of the four $U_{2^{K-1}}^{i,\textrm{ext}}$ fails}\} \cap A_{j,\sigma}(2^k,2^{K-1})\Big).$$
Hence, by independence of the two latter events,
$$\PPP(A_{j,\sigma}(2^k,2^K)) \leq \PPP(\At_{j,\sigma}^{./\eta'}(2^k,2^K)) + (4 \delta) \PPP(A_{j,\sigma}(2^k,2^{K-1})).$$
We then iterate this argument: after $K-k$ steps,
\begin{align*}
\PPP(A_{j,\sigma}(2^k,2^K)) & \\
& \hspace{-2cm} \leq \PPP(\At_{j,\sigma}^{. / \eta'}(2^k,2^K)) + (4 \delta) \PPP(\At_{j,\sigma}^{. / \eta'}(2^k,2^{K-1})) + (4 \delta)^2 \PPP(\At_{j,\sigma}^{. / \eta'}(2^k,2^{K-2})) + \ldots \\
& \hspace{-1.5cm} + (4 \delta)^{K-k-1} \PPP(\At_{j,\sigma}^{. / \eta'}(2^k,2^{k+1})) + (4 \delta)^{K-k}.
\end{align*}
We then use the size $\eta'_0$ to go backwards: if the crossings are $\eta'$-separated at some scale $m$, there exists some landing sequence $I_{\eta'}$ of size $\eta'$ where the probability of landing is comparable to the probability of just being $\eta'$-well-separated, and then we can reach $I_{\eta'_0}$ of size $\eta'_0$ on the next scale. More precisely, there exist universal constants $C_1(\eta')$, $C_2(\eta')$ depending only on $\eta'$ such that for all $1 \leq i' \leq i$, we can choose some $I_{\eta'}$ (which can depend on $i'$) such that
$$\PPP(\At_{j,\sigma}^{. / \eta'}(2^k,2^{K-i'})) \leq C_1(\eta') \PPP(\Att_{j,\sigma}^{. / \eta',I_{\eta'}}(2^k,2^{K-i'}))$$
and then go to $I_{\eta'_0}$ on the next scale with cost $C_2(\eta')$:
$$\PPP(\Att_{j,\sigma}^{. / \eta',I_{\eta'}}(2^k,2^{K-i'})) \leq C_2(\eta') \PPP(\Att_{j,\sigma}^{. / \eta'_0,I_{\eta'_0}}(2^k,2^{K-i'+1})).$$
Now for the size $\eta'_0$, going from $\partial S_m$ to $\partial S_{2m}$ has a cost $C'_0$ depending \emph{only} on $\eta'_0$ on each scale $m$, we have thus
$$\PPP(\At_{j,\sigma}^{. / \eta'}(2^k,2^{K-i'})) \leq C_1(\eta') C_2(\eta') C_0^{i'-1} \PPP(\Att_{j,\sigma}^{. / \eta'_0,I_{\eta'_0}}(2^k,2^K)).$$
There remains a problem with the first term $\PPP(\At_{j,\sigma}^{. / \eta'}(2^k,2^K))$\ldots So assume that we have started from $2^{K-1}$ instead, so that the annulus $S_{2^{K-1},2^K}$ remains free:
\begin{align*}
\PPP(A_{j,\sigma}(2^k,2^K)) & \\
& \hspace{-2cm} \leq \PPP(A_{j,\sigma}(2^k,2^{K-1}))\\
& \hspace{-2cm} \leq \PPP(\At_{j,\sigma}^{. / \eta'}(2^k,2^{K-1})) + (4 \delta) \PPP(\At_{j,\sigma}^{. / \eta'}(2^k,2^{K-2})) + (4 \delta)^2 \PPP(\At_{j,\sigma}^{. / \eta'}(2^k,2^{K-3})) + \ldots \\
& \hspace{-1cm} + (4 \delta)^{K-k-2} \PPP(\At_{j,\sigma}^{. / \eta'}(2^k,2^{k+1})) + (4 \delta)^{K-k-1} \\
& \hspace{-2cm} \leq C_1(\eta') C_2(\eta') \bigg[ 1 + (4 \delta C_0) + \ldots + (4 \delta C_0)^{K-k-1} \bigg] \PPP(\Att_{j,\sigma}^{. / \eta'_0,I_{\eta'_0}}(2^k,2^K)).
\end{align*}
Now $C_0$ is fixed as was noticed before, so we may have taken $\delta$ such that $4 \delta C_0 < 1/2$, so that
$$C_1(\eta') C_2(\eta') \bigg[ 1 + (4 \delta C_0) + \ldots + (4 \delta C_0)^{K-k-1} \bigg] \leq C_3(\eta')$$
for some $C_3(\eta')$. We have thus reached the desired conclusion for external extremities:
$$\PPP(A_{j,\sigma}(2^k,2^K)) \leq C_3(\eta') \PPP(\Att_{j,\sigma}^{. / \eta'_0,I_{\eta'_0}}(2^k,2^{K})).$$

\subsubsection*{2. Internal extremities}

\begin{figure}
\begin{center}
\includegraphics[width=11cm]{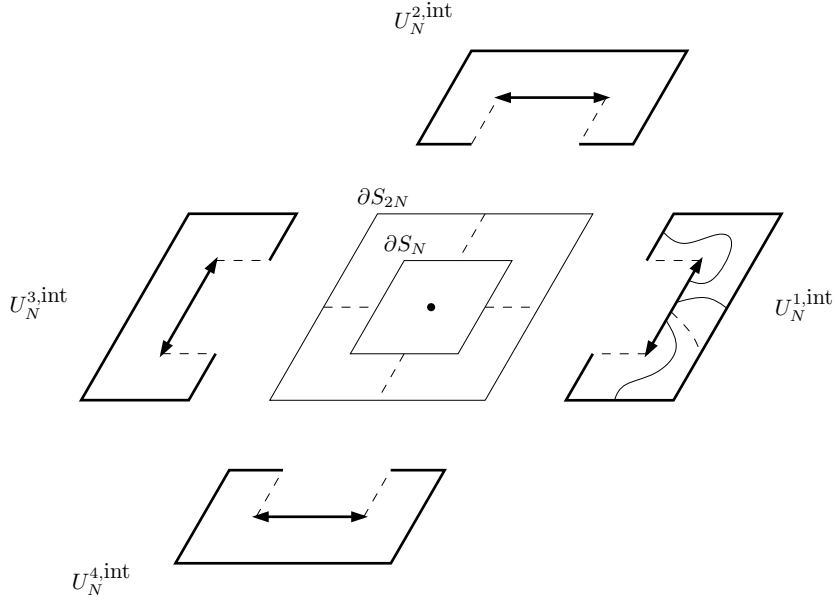}
\caption{\label{U_shapes_int}For the internal extremities, we consider the same domains but we mark different parts of the boundary.}
\end{center}
\end{figure}

The reasoning is the same for internal extremities, except that we work in the other direction, from $\partial S_{2^k}$ toward the interior. If we consider the domains $U_{N}^{i,\textrm{int}}$ having the same shapes as the $U_{N}^{i,\textrm{ext}}$ domains, but with different parts of the boundary distinguished (see Figure \ref{U_shapes_int}), then the lemma remains true. Hence,
\begin{align*}
\PPP(\Att_{j,\sigma}^{. / \eta'_0,I_{\eta'_0}}(2^k,2^K)) & \leq \PPP(\Att_{j,\sigma}^{. / \eta'_0,I_{\eta'_0}}(2^{k+1},2^K))\\
& \hspace{-2.5cm}\leq \PPP(\Att_{j,\sigma}^{\eta,. / \eta'_0,I_{\eta'_0}}(2^{k+1},2^K)) + (4 \delta) \PPP(\Att_{j,\sigma}^{\eta,. / \eta'_0,I_{\eta'_0}}(2^{k+2},2^K)) + \ldots \\
& \hspace{-1.5cm} + (4 \delta)^{K-k-2} \PPP(\Att_{j,\sigma}^{\eta,. / \eta'_0,I_{\eta'_0}}(2^{K-1},2^K)) + (4 \delta)^{K-k-1}\\
& \hspace{-2.5cm} \leq C_1(\eta) C_2(\eta) \bigg[ 1 + (4 \delta C_0) + \ldots + (4 \delta C_0)^{K-k-1} \bigg] \PPP(\Att_{j,\sigma}^{\eta_0,I_{\eta_0} / \eta'_0,I_{\eta'_0}}(2^k,2^K))
\end{align*}
and the conclusion follows.
\end{proof}

\subsection{Some consequences}

We now state some important consequences of the previous theorem.

\subsubsection*{Extendability}

\begin{proposition}
Take $j \geq 1$ and a color sequence $\sigma \in \Sj$. Then
\begin{equation}
\PPP(A_{j,\sigma}(n,2N)), \: \PPP(A_{j,\sigma}(n/2,N)) \asymp \PPP(A_{j,\sigma}(n,N))
\end{equation}
uniformly in $p$, $\PPP$ between $\PP_p$ and $\PP_{1-p}$ and $n_0(j) \leq n \leq N \leq L(p)$.
\end{proposition}

\begin{proof}
This proposition comes directly from combining the arm separation theorem with the extendability property of the $\Att$ events (item \ref{extend}. of Proposition \ref{easy_prop}).
\end{proof}

\subsubsection*{Quasi-multiplicativity}

\begin{proposition}
Take $j \geq 1$ and a color sequence $\sigma \in \Sj$. Then
\begin{equation}
\PPP(A_{j,\sigma}(n_1,n_2)) \PPP(A_{j,\sigma}(n_2,n_3)) \asymp \PPP(A_{j,\sigma}(n_1,n_3))
\end{equation}
uniformly in $p$, $\PPP$ between $\PP_p$ and $\PP_{1-p}$ and $n_0(j) \leq n_1 < n_2 < n_3 \leq L(p)$.
\end{proposition}

\begin{proof}
On one hand, we have
$$\PPP(A_{j,\sigma}(n_1,n_3)) \leq \PPP(A_{j,\sigma}(n_1,n_2) \cap A_{j,\sigma}(n_2,n_3)) = \PPP(A_{j,\sigma}(n_1,n_2)) \PPP(A_{j,\sigma}(n_2,n_3))$$
by independence of the events $A_{j,\sigma}(n_1,n_2)$ and $A_{j,\sigma}(n_2,n_3)$.

On the other hand, we may assume that $n_2 \geq 8 n_1$. Then for some $\eta_0$, $I_{\eta_0}$, the previous results (separation and extendability) allow to use the quasi-multiplicativity for $\Att$ events (item \ref{quasi_mult2}. of Proposition \ref{easy_prop}):
\begin{align*}
\PPP(A_{j,\sigma}(n_1,n_2)) \PPP(A_{j,\sigma}(n_2,n_3)) & \asymp \PPP(A_{j,\sigma}(n_1,n_2/4)) \PPP(A_{j,\sigma}(n_2,n_3))\\
& \asymp \PPP(\Att_{j,\sigma}^{. / \eta_0,I_{\eta_0}}(n_1,n_2/4)) \PPP(\Att_{j,\sigma}^{\eta_0,I_{\eta_0} / .}(n_2,n_3))\\
& \asymp \PPP(A_{j,\sigma}(n_1,n_3)).
\end{align*}
\end{proof}

\subsubsection*{Arms with defects}

In some situations, the notion of arms that are completely monochromatic is too restrictive, and the following question arises quite naturally: do the probabilities change if we allow the arms to present some (fixed) number of ``defects'', \emph{ie} sites of the opposite color?

We define $A^{(d)}_{j,\sigma}(n,N)$ the event that there exist $j$ arms $a_1,\ldots,a_j$ from $\partial S_n$ to $\partial S_N$ with the property: for any $i \in \{1,\ldots,j\}$, $a_i$ contains at most $d$ sites of color $\tilde{\sigma}_i$. The quasi-multiplicativity property entails the following result, which will be needed for the proof of Theorem \ref{armnear}:

\begin{proposition} \label{defects}
Let $j \geq 1$ and $\sigma \in \Sj$. Fix also some number $d$ of defects. Then we have
\begin{equation}
\PPP\big( A^{(d)}_{j,\sigma}(n,N) \big) \asymp (1+\log(N/n))^d \PPP\big( A_{j,\sigma}(n,N) \big)
\end{equation}
uniformly in $p$, $\PPP$ between $\PP_p$ and $\PP_{1-p}$ and $n_0(j) \leq n \leq N \leq L(p)$.
\end{proposition}

Actually, we will only need the upper bound on $\PPP\big( A^{(d)}_{j,\sigma}(n,N) \big)$. For instance, we will see in the next section that the arm events decay like power laws at the critical point. This proposition thus implies, in particular, that the ``arm with defects'' events are described by the same exponents: allowing defects just adds a logarithmic correction.

\begin{proof}
We introduce a logarithmic division of the annulus $S_{n,N}$: we take $k$ and $K$ such that $2^{k-1} < n \leq 2^k$ and $2^K \leq N < 2^{K+1}$. Roughly speaking, we ``take away'' the annuli where the defects take place, and ``glue'' the pieces of arms in the remaining annuli by using the quasi-multiplicativity property.

Let us begin with the upper bound: we proceed by induction on $d$. The property clearly holds for $d=0$. Take some $d \geq 1$: by considering the first annuli $S_{2^i,2^{i+1}}$ where a defect occurs, we get
\begin{equation}
\PPP(A^{(d)}_{j,\sigma}(n,N)) \leq \sum_{i=k}^{K-1} \PPP(A_{j,\sigma}(2^k,2^i)) \PPP(A^{(d-1)}_{j,\sigma}(2^{i+1},2^K)).
\end{equation}
We have $\PPP(A^{(d-1)}_{j,\sigma}(2^{i+1},2^K)) \leq C_{d-1} (1+\log(N/n))^{d-1} \PPP(A_{j,\sigma}(2^{i+1},2^K))$ thanks to the induction hypothesis, and by quasi-multiplicativity,
\begin{align*}
\PPP(A^{(d)}_{j,\sigma}(n,N)) & \leq (1+\log(N/n))^{d-1} C_{d-1} \sum_{i=k}^{K-1} \PPP(A_{j,\sigma}(2^k,2^i)) \PPP(A_{j,\sigma}(2^{i+1},2^K))\\
& \leq C_{d-1} (1+\log(N/n))^{d-1} \sum_{i=k}^{K-1} C' \PPP(A_{j,\sigma}(2^k,2^K))\\
& \leq C_d (1+\log(N/n))^{d-1} (K-k) \PPP(A_{j,\sigma}(2^k,2^K)),
\end{align*}
which gives the desired upper bound.

For the lower bound, note that for any $k \leq i_0 < i_1 < \ldots < i_d < i_{d+1} = K$, $A_{j,\sigma}(n,N)$ $\supseteq$ $A_{j,\sigma}(2^{k-1},2^{K+1})$ $\supseteq$ $A_{j,\sigma}(2^{k-1},2^{K+1})$ $\cap$ $\{$Each of the $j$ arms has exactly one defect in each of the annuli $S_{2^{i_r},2^{i_r+1}}$$\}$, so that for $K-k \geq d+1$,
\begin{align*}
\PPP(A^{(d)}_{j,\sigma}(n,N)) & \geq \sum_{k=i_0 < i_1 < i_2 < \ldots < i_d < i_{d+1}=K} C_d \prod_{r=0}^d \PPP(A_{j,\sigma}(2^{i_r+1},2^{i_{r+1}}))\\
& \geq C'_d \binom{K-k-1}{d} \PPP(A_{j,\sigma}(2^{k-1},2^{K+1}))\\
& \geq C''_d (K-k)^d \PPP(A_{j,\sigma}(2^{k-1},2^{K+1})),
\end{align*}
and our lower bound follows.
\end{proof}

\subsubsection*{Remark: more general annuli}

We will sometimes need to consider more general arm events, in annuli of the form $R \setminus r$, for non-necessarily concentric parallelograms $r \subseteq \mathring{R}$. Items \ref{extend}. and \ref{quasi_mult2}. of Proposition \ref{easy_prop} can easily be extended. Separateness and well-separateness can be defined in the same way for these arm events, and for any $\tau>1$, we can get results uniform in the usual parameters and in parallelograms $r$, $R$ such that $S_n \subseteq r \subseteq S_{\tau n}$ and $S_{N/\tau} \subseteq R \subseteq S_{N}$ for some $n,N \leq L(p)$:
\begin{equation} \label{non_concentric}
\PPP(\partial r \leadsto^{j,\sigma} \partial R) \asymp \PPP(\partial S_n \leadsto^{j,\sigma} \partial S_N),
\end{equation}
and similarly with separateness conditions on the external boundary or on the internal one.

\subsection{Arms in the half-plane}

So far, we have been interested in arm events in the whole plane: we can define in the same way the event $B_{j,\sigma}(n,N)$ that there exist $j$ arms that \emph{stay in the upper half-plane $\mathbb{H}$}, of colors prescribed by $\sigma \in \Sj$ and connecting $\partial S'_n$ to $\partial S'_N$, with the notation $\partial S'_n = (\partial S_n) \cap \mathbb{H}$. These events appear naturally when we look at arms near a boundary.

For the sake of completeness, let us just mention that all the results stated here remain true for arms in the half-plane. In fact, there is a natural way to order the different arms, which makes this case easier. We will not use these events in the following, and we leave the details to the reader.

\section{Consequences for critical percolation} \label{sec_critical}

When studying the phase transition of percolation, the critical regime plays a very special role. It possesses a strong property of \emph{conformal invariance} in the scaling limit. This particularity, first observed by physicists (\cite{P,BPZ1,BPZ2}), has been proved by Smirnov in \cite{Sm}, and later extended by Camia and Newman in \cite{CN1}. It allows to link the critical regime to the $SLE$ processes (with parameter $6$ here) introduced by Schramm in \cite{Sc}, and thus to use computations made for these processes (\cite{LSW1,LSW2}).

In the next sections, we will see why our description of critical percolation yields in turn a good description of near-critical percolation (which does not feature a priori any sort of conformal invariance), in particular how the characteristic functions behave through the phase transition.

\subsection{Arm exponents for critical percolation}

\subsubsection*{Color switching}

We focus here on the probabilities of arm events at the critical point. For arms in the half-plane, a nice combinatorial argument (noticed in \cite {AAD, SmW}) shows that once fixed the number $j$ of arms, prescribing the color sequence $\sigma$ does not change the probability. This is the so-called ``color exchange trick'':

\begin{proposition} \label{exchange2}
Let $j \geq 1$ be any fixed integer. If $\sigma,\sigma'$ are two color sequences, then for any $n'_0(j) \leq n \leq N$,
\begin{equation}
\PP_{1/2}(B_{j,\sigma}(n,N)) = \PP_{1/2}(B_{j,\sigma'}(n,N)).
\end{equation}
\end{proposition}

\begin{proof}
The proof relies on the fact that there is a canonical way to order the arms. If we condition on the $i$ left-most arms, percolation in the remaining domain is unbiased, so that we can ``flip'' the sites there: for any color sequence $\sigma$, if we denote by
$$\tilde{\sigma}^{(i)}=(\sigma_1,\ldots,\sigma_i,\tilde{\sigma}_{i+1},\ldots,\tilde{\sigma}_{j})$$
the sequence with the same $i$ first colors, and the remaining ones flipped, then
$$\PP_{1/2}(B_{j,\sigma}(n,N)) = \PP_{1/2}(B_{j,\tilde{\sigma}^{(i)}}(n,N)).$$
It is not hard to convince oneself that for any two sequences $\sigma, \sigma'$, we can go from $\sigma$ to $\sigma'$ in a finite number of such operations.
\end{proof}

This result is not as direct in the whole plane case, since there is no canonical ordering any more. However, the argument can be adapted to prove that the probabilities change only by a constant factor, as long as there is an interface, \emph{ie} as long as $\sigma$ contains at least one white arm and one black arm.

\begin{proposition} \label{exchange}
Let $j \geq 1$ be any fixed integer. If $\sigma,\sigma' \in \Sj$ are two \emph{non-constant} color sequences (\emph{ie} both colors are present), then
\begin{equation}
\PP_{1/2}(A_{j,\sigma}(n,N)) \asymp \PP_{1/2}(A_{j,\sigma'}(n,N))
\end{equation}
uniformly in $n_0(j) \leq n \leq N$.
\end{proposition}

\begin{proof}
Assume that $\sigma_1 = B$ and $\sigma_2 = W$, and fix some landing sequence $I$. If we replace the event $A_{j,\sigma}(n,N)$ by the strengthened event $\Ab_{j,\sigma}^{I/.}(n,N)$, we are allowed to condition on the black arm arriving on $I_1$ and on the white arm arriving on $I_2$ that are closest to each other: if we choose for instance $I$ such that the point $(N,0)$ is between $I_1$ and $I_2$, these two arms can be determined via an exploration process starting at $(N,0)$. We can then ``flip'' the remaining region. More generally, we can condition on any set of consecutive arms including these two arms, and the result follows for the same reasons as in the half-plane case.
\end{proof}

We would like to stress the fact that for the reasoning, we crucially need two arms of opposite colors. In fact, the preceding result is expected to be false if $\sigma$ is constant and $\sigma'$ non-constant (the two probabilities \emph{not} being of the same order of magnitude), which is quite surprising at first sight.

\subsubsection*{Derivation of the exponents}

The link with $SLE_6$ makes it possible to prove the existence of the (multichromatic) ``arm exponents'', and derive their values (\cite{LSW4,SmW}).

\begin{theorem} \label{armcrit}
Fix some $j \geq 1$. Then for any non-constant color sequence $\sigma \in \Sj$,
\begin{equation}
\PP_{1/2}\big(A_{j,\sigma}(n_0(j),N)\big) \approx N^{-\alpha_j}
\end{equation}
when $N \to \infty$, with
\begin{itemize}
\item[$\bullet$] $\alpha_1 = 5/48$,
\item[$\bullet$] and for $j \geq 2$, $\alpha_j = (j^2-1)/12$.
\end{itemize}
\end{theorem}

Let us sketch very briefly how it is proved. Consider the discrete (radial) exploration process in a unit disc: using the property of conformal invariance in the scaling limit, we can prove that this process converges toward a radial $SLE_6$, for which we can compute disconnection probabilities. It implies that
$$\PP_{1/2}(A_{j,\sigma}(\eta n,n)) \to g_j(\eta),$$
for some function $g_j(\eta) \sim \eta^{\alpha_j}$ as $\eta \to 0$. Then, the quasi-multiplicativity property in concentric annuli of fixed modulus provides the desired result.

As mentioned, this theorem is believed to be \emph{false} for constant $\sigma$, \emph{ie} when the arms are all of the same color. In this case, the probability should be smaller, or equivalently the exponent (assuming its existence) larger. Hence for each $j=2,3,\ldots$, there are two different arm exponents, the multichromatic $j$-arm exponent $\alpha_j$ given by the previous formula (most often simply called the $j$-arm exponent) and the monochromatic $j$-arm exponent $\alpha'_j$, for which no closed formula is currently known, nor even predicted. The only result proved so far concerns the case $j=2$: as shown in \cite{LSW4}, the monochromatic $2$-arm exponent can be expressed as the leading eigenvalue of some (complicated) differential operator. Numerically, it has been found (see \cite{AAD}) to be approximately $\alpha'_2 \simeq 0.35\ldots$

Note also that the derivation using $SLE_6$ only provides a logarithmic equivalence. However, there are reasons to believe that a stronger equivalence holds, a ``$\asymp$'': for instance we know that this is the case for the ``universal exponents'' computed in the next sub-section.

We will often relate events to combinations of arm events, that in turn can be linked (see next section) to arm events at the critical point $p=1/2$. It will thus be convenient to introduce the following notation, with $\sigma_j = BWBW\ldots$: for any $n_0(j) \leq n < N$,
\begin{equation}
\pi_j(n|N) := \PP_{1/2}(A_{j,\sigma_j}(n,N))
\end{equation}
($\asymp \PP_{1/2}(A_{j,\sigma}(n,N))$ for any non-constant $\sigma$), and in particular
\begin{equation}
\pi_j(N) := \PP_{1/2}(A_{j,\sigma_j}(n_0(j),N)) \quad (\approx N^{-\alpha_j}).
\end{equation}

Note that with this notation, the a-priori bound and the quasi-multiplicativity property take the aesthetic forms
\begin{equation}\label{pi1}
C (n/N)^{\alpha_j} \leq \pi_j(n|N) \leq C' (n/N)^{\alpha'},
\end{equation}
\begin{equation}\label{pi2}
\text{and} \:\: \pi_j(n_1|n_2) \pi_j(n_2|n_3) \asymp \pi_j(n_1|n_3).
\end{equation}

Let us mention that we can derive in the same way exponents for arms in the upper half-plane, the ``half-plane exponents'':

\begin{theorem} \label{armcrit2}
Fix some $j \geq 1$. Then for any sequence of colors $\sigma$,
\begin{equation}
\PP_{1/2}\big(B_{j,\sigma}(n'_0(j),N)\big) \approx N^{-\beta_j}
\end{equation}
when $N \to \infty$, with
$$\beta_j = j(j+1)/6.$$
\end{theorem}

\begin{remark}
As mentioned earlier, the triangular lattice is at present the only lattice for which conformal invariance in the scaling limit has been proved, and as a consequence the only lattice for which the existence and the values of the arm exponents have been established -- with the noteworthy exception of the three ``universal'' exponents that we are going to derive.
\end{remark}

\subsubsection*{Note: fractality of various sets}

These arm exponents can be used to measure the size (Hausdorff dimension) of various sets describing percolation clusters. In physics literature for instance (see e.g. \cite{AAD}), a set $S$ is said to be fractal of dimension $D_S$ if the density of points in $S$ within a box of size $n$ decays as $n^{-x_S}$, with $x_S = 2 - D_S$ (in 2D). The co-dimension $x_S$ is related to arm exponents in many cases:
\begin{itemize}
\item[$\bullet$] The $1$-arm exponent is related to the existence of long connections, from the center of a box to its boundary. It will thus measure the size of ``big'' clusters, like the incipient infinite cluster (IIC) as defined by Kesten (\cite{Ke2}), which scales as $n^{(2-5/48)}=n^{91/48}$.

\item[$\bullet$] The monochromatic $2$-arm exponent describes the size of the ``backbone'' of a cluster. The fact that this backbone is much thinner than the cluster itself was used by Kesten \cite{Ke5} to prove that the random walk on the IIC is sub-diffusive (while it has been proved to converge toward a Brownian Motion on a super-critical infinite cluster).

\item[$\bullet$] The multichromatic $2$-arm exponent is related to the boundaries (hulls) of big clusters, which are thus of fractal dimension $2-\alpha_2 = 7/4$.

\item[$\bullet$] The $3$-arm exponent concerns the external (accessible) perimeter of a cluster, which is the accessible part of the boundary: one excludes ``fjords'' which are connected to the exterior only by $1$-site wide passages. The dimension of this frontier is $2-\alpha_3 = 4/3$. These two latter exponents can be observed on random interfaces, numerically and in ``real-life'' experiments as well (see \cite{Sa1,Sa2} for instance).

\item[$\bullet$] As mentioned earlier, the $4$-arm exponent with alternating colors counts the pivotal (singly-connecting) sites (often called ``red'' sites in physics literature). This set can be viewed as the contact points between two distinct (large) clusters, its dimension is $2-\alpha_4=3/4$. We will relate this exponent to the correlation length exponent $\nu$ in Section \ref{charac}.
\end{itemize}

\subsection{Universal exponents}

We will now examine as a complement some particular exponents, for which heuristic predictions and elementary derivations exist, namely $\beta_2=1$, $\beta_3=2$ and $\alpha_5=2$. They are all integers, and they were established before the complete derivation using the $SLE_6$ (and actually they provide crucial a-priori estimates to prove the convergence toward $SLE_6$). Moreover, the equivalence that we get is stronger: we can replace the ``$\approx$'' by a ``$\asymp$''.

\begin{theorem} \label{armuniv}
When $N \to \infty$,
\begin{enumerate}
\item For any $\sigma \in \mathfrak{S}_2$,
$$\PP_{1/2}\big(B_{2,\sigma}(0,N)\big) \asymp N^{-1}.$$ \label{it2arms}

\item For any $\sigma \in \mathfrak{S}_3$,
$$\PP_{1/2}\big(B_{3,\sigma}(0,N)\big) \asymp N^{-2}.$$ \label{it3arms}

\item For any non-constant $\sigma \in \tilde{\mathfrak{S}}_5$,
$$\PP_{1/2}\big(A_{5,\sigma}(0,N)\big) \asymp N^{-2}.$$ \label{it5arms}
\end{enumerate}
\end{theorem}

\begin{proof}
We give a complete proof only for item \ref{it5arms}., since we will not need the two first ones -- we will however sketch at the end how to derive them.

\begin{figure}
\begin{center}
\includegraphics[width=8cm]{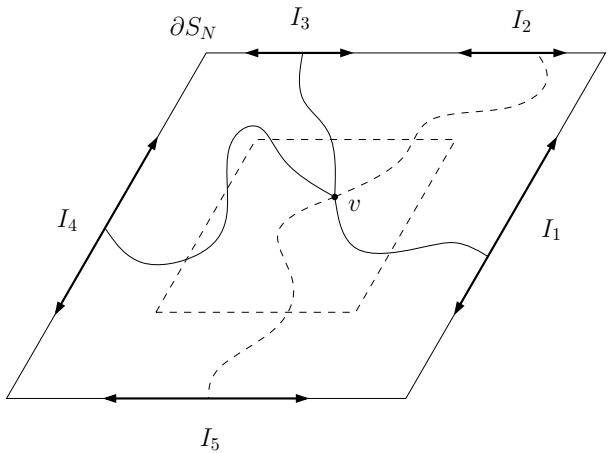}
\caption{\label{5arms}The landing sequence $I_1,\ldots,I_5$.}
\end{center}
\end{figure}

Heuristically, we can prove that the $5$-arm sites can be seen as particular points on the boundary of two big black clusters, and that consequently their number is of order $1$ in $S_{N/2}$. Then it suffices to use that the different sites in $S_{N/2}$ produce contributions of the same order. This argument can be made rigorous by proving that the number of ``macroscopic'' clusters has an exponential tail: we refer to the first exercise sheet in \cite{W2} for more details. We propose here a more direct -- but less elementary -- proof using the separation lemmas.

By color switching, it is sufficient to prove the claim for $\sigma=BWBBW$. In light of our previous results, it is clear that
$$\PP_{1/2}\big(v \leadsto^{5,\sigma} \partial S_N) \asymp \PP_{1/2}\big(0 \leadsto^{5,\sigma} \partial S_N)$$
uniformly in $N$, $v \in S_{N/2}$. It is thus enough to prove that the number of such $5$-arm sites in $S_{N/2}$ is of order $1$. 

Let us consider the upper bound first. Take the particular landing sequence $I_1,\ldots,I_5$ depicted on the figure, and consider the event
$$A_v := \{v \leadsto_{5,\sigma}^I \partial S_N\} \cap \{ \text{$v$ is black} \}.$$
Note that $\PP_{1/2}(A_v) = \frac{1}{2} \PP(v \leadsto_{5,\sigma}^I \partial S_N)$ since the existence of the arms is independent of the status of $v$, so that $\PP_{1/2}(A_v) \asymp \PP_{1/2}\big(0 \leadsto^{5,\sigma} \partial S_N)$. We claim that $A_v$ can occur for at most one site $v$. Indeed, assume that $A_v$ and $A_w$ occur, and denote by $r_1,\ldots,r_5$ and $r'_1,\ldots,r'_5$ the corresponding arms. Since $r_1 \cup r_4 \cup \{v\}$ separates $I_3$ from $I_5$, necessarily $w \in r_1 \cup r_4 \cup \{v\}$. Similarly, $w \in r_2 \cup r_4 \cup \{v\}$: since $r_1 \cap r_2 = \varnothing$, we get that $w \in r_4 \cup \{v\}$. But only one arm can ``go through'' $r_3 \cup r_5$: the arm $r'_1 \cup \{w\}$ from $w$ to $I_1$ has to contain $v$, and so does $r'_2 \cup \{w\}$. Since $r'_1 \cap r'_2 = \varnothing$, we get finally $v = w$.

Consequently,
\begin{equation}
1 \geq \PP_{1/2}\big( \cup_{v \in S_{N/2}} A_v \big) = \sum_{v \in S_{N/2}} \PP_{1/2}(A_v) \asymp N^2 \PP_{1/2}\big(0 \leadsto^{5,\sigma} \partial S_N),
\end{equation}
which provides the upper bound.

Let us turn to the lower bound. We perform a construction showing that a $5$-arm site appears with positive probability, by using multiple applications of RSW. With probability at least $\delta_{16}^2 > 0$, there is a black horizontal crossing in the strip $[-N,N] \times [0,N/8]$, together with a white one in $[-N,N] \times [-N/8,0]$. Assume this is the case, and condition on the lowest black left-right crossing $c$. We note that any site on this crossing has already $3$ arms, $2$ black arms and a white one. On the other hand, the percolation in the region above it remains unbiased.

Now, still using RSW, with positive probability $c$ is connected to the top side by a black path included in $[-N/8,0] \times [-N,N]$, and another white path included in $[0,N/8] \times [-N,N]$. Assume these paths exist, and denote by $v_1$ and $v_2$ the respective sites on $c$ where they arrive. Follow $c$ from left to right, and consider the last vertex $v$ before $v_2$ that is connected to the top side: it is not hard to see that there is a white arm from $v$ to the top side, and that $v \in S_{N/2}$, since $v$ is between $v_1$ and $v_2$. Hence,
\begin{equation}
\PP_{1/2}\big( \cup_{v \in S_{N/2}} \{v \leadsto^{5,\sigma} \partial S_N\} \big) \geq C
\end{equation}
for some universal constant $C > 0$. Since we also have
\begin{align*}
\PP_{1/2}\big( \cup_{v \in S_{N/2}} \{v \leadsto^{5,\sigma} \partial S_N\} \big) & \leq \sum_{v \in S_{N/2}} \PP_{1/2}\big(v \leadsto^{5,\sigma} \partial S_N) \\
& \leq C' N^2 \PP_{1/2}\big(0 \leadsto^5 \partial S_N),
\end{align*}
the desired lower bound follows.

\medskip

We now explain briefly how to obtain the two half-plane exponents (items \ref{it2arms}. and \ref{it3arms}.). We again use the arm separation theorem, but note that \cite{W2} contains elementary proofs for them too. For the $2$-arm exponent in the half-plane, we take $\sigma=BW$ and remark that if we fix two landing areas $I_1$ and $I_2$ on $\partial S'_N$, at most one site on the segment $[-N/2,N/2] \times \{0\}$ is connected by two arms to $I_1$ and $I_2$. On the other hand, a $2$-arm site can be constructed by considering a black path from $[-N/2,0] \times \{0\}$ to $I_1$ and a white path from $[0,N/2] \times \{0\}$ to $I_2$. Then the right-most site on $[-N/2,N/2] \times \{0\}$ connected by a black arm to $I_1$ is a $2$-arm site. Several applications of RSW allow to conclude.

For the $3$-arm exponent, we take three landing areas $I_1$, $I_2$ and $I_3$, and $\sigma=BWB$. It is not hard to construct a $3$-arm site by taking a black crossing from $I_1$ to $I_3$ and considering the closest to $I_2$. We can then force it to be in $S_{N/2} \cap \mathbb{H}$ by a RSW construction. For the upper bound, we first notice that if we require the arms to stay strictly positive, the probability remains of the same order of magnitude. We then use that at most one site in $S_{N/2} \cap \mathbb{H}$ is connected to the landing areas by three positive arms.
\end{proof}

The proofs given here only require RSW-type considerations (including separation of arms). As a consequence, they also apply to near-critical percolation. It is clear for $\PP_p$, on scales $N \leq L(p)$, but a priori only for the color sequences we have used in the proofs (resp. $\sigma =$ $BW$, $BWB$ and $BWBBW$ -- and of course those we can deduce from them by the symmetry $p \leftrightarrow 1-p$): it is indeed not obvious that $\PP_p\big(0 \leadsto^{\sigma} \partial S_N\big) \asymp \PP_p\big(0 \leadsto^{\sigma'} \partial S_N\big)$ for two distinct non-constant $\sigma$ and $\sigma'$. This is essentially Theorem \ref{armnear}, its proof occupies a large part of the next section.

For a general measure $\PPP$ between $\PP_p$ and $\PP_{1-p}$, we similarly have to be careful: we do not know whether $\PPP(v \leadsto^5 \partial S_N)$ remains of the same order of magnitude when $v$ varies. This also comes from Theorem \ref{armnear}, but in the course of its proof we will need an a-priori estimate on the probability of $5$ arms, so temporarily we will be content with a weaker statement that does not use its conclusion:
\begin{lemma} \label{sum5arms}
For $\sigma=BWBBW$ ($= \sigma_5$), we have uniformly in $p$, $\PPP$ between $\PP_p$ and $\PP_{1-p}$ and $N \leq L(p)$:
\begin{equation}
\sum_{v \in S_{N/2}} \PPP\big(v \leadsto^{5,\sigma} \partial S_N\big) \asymp 1.
\end{equation}
\end{lemma}

\begin{remark}
We would like to mention that these estimates for critical and near-critical percolation remain valid on other lattices too, like the square lattice (see the discussion in the last section) -- at least for the color sequences that we have used in the proofs, no analog of the color exchange trick being available (to our knowledge).
\end{remark}

\section{Consequences for near-critical percolation\label{nearcritical}}

\subsection{Arm exponents for near-critical percolation}

We would like now to study how the events $A_{j,\sigma}(n,N)$ are affected by a variation of the parameter $p$. We have defined $L(p)$ in terms of crossing events to be the scale on which percolation can be considered as (approximately) critical, we would thus expect the probabilities of these events not to vary too much if $n,N$ remain below $L(p)$. This is what happens:

\begin{theorem} \label{armnear}
Let $j \geq 1$, $\sigma \in \Sj$ be as usual. Then we have
\begin{equation}
\PPP\big(A_{j,\sigma}(n,N)\big) \asymp \PPP'\big(A_{j,\sigma}(n,N)\big)
\end{equation}
uniformly in $p$, $\PPP$ and $\PPP'$ between $\PP_p$ and $\PP_{1-p}$, and $n_0(j) \leq n \leq N \leq L(p)$.
\end{theorem}

Note that if we take in particular $\PPP'=\PP_{1/2}$, we get that below the scale $L(p)$, the arm events remain roughly the same as at criticality:
$$\PPP\big(A_{j,\sigma}(n,N)\big) \asymp \PP_{1/2}\big(A_{j,\sigma}(n,N)\big).$$
This will be important to derive the critical exponents for the characteristic functions from the arm exponents at criticality.

\begin{remark}
Note that the property of exponential decay with respect to $L(p)$ (Lemma \ref{explem}), proved in Section \ref{expsection}, shows that we cannot hope for a similar result on a much larger range, so that $L(p)$ is the appropriate scale here: consider for instance $\PP_p$ with $p>1/2$, the probability to observe a white arm tends to $0$ exponentially fast (and thus much faster than at the critical point), while the probability to observe a certain number of disjoint black arms tends to a positive constant.
\end{remark}

\subsection{Proof of the theorem}

We want to compare the value of $\PPP(A_{j,\sigma}(n,N))$ for different measures $\PPP$. A natural way of doing this is to go from one to the other by using Russo's formula (Theorem \ref{russo}). But since for $j \geq 2$ and non-constant $\sigma$, the event $A_{j,\sigma}(n,N)$ is not monotone, we need a slight generalization of this formula, for events that can be expressed as the intersection of two monotone events, one increasing and one decreasing. We also allow the parameters $p_v$ to be differentiable functions of $t \in [0,1]$.

\begin{lemma}
Let $A^+$ and $A^-$ be two monotone events, respectively increasing and decreasing, depending only on the sites contained in some finite set of vertices $S$. Let $(\hat{p}_v)_{v \in S}$ be a family of differentiable functions $\hat{p}_v : t \in [0,1] \mapsto \hat{p}_v(t) \in [0,1]$, and denote by $(\PPP_t)_{0 \leq t \leq 1}$ the associated product measures. Then
\begin{align*}
\frac{d}{dt} \PPP_t(A^+ \cap A^-) & \\
& \hspace{-1.5cm} = \sum_{v \in S} \frac{d}{dt} \hat{p}_v(t) \Big[ \PPP_t(\text{$v$ is pivotal for $A^+$ but not for $A^-$, and $A^-$ occurs}) & \\
& - \PPP_t(\text{$v$ is pivotal for $A^-$ but not for $A^+$, and $A^+$ occurs})\Big].
\end{align*}
\end{lemma}

\begin{proof}
We adapt the proof of standard Russo's formula. We use the same function $\mathcal{P}$ of the parameters $(\hat{p}_v)_{v \in S}$, and we note that for a small variation $\epsilon>0$ in $w$,
\begin{align*}
\PPP^{+\epsilon}(A^+ \cap A^-) - \PPP(A^+ \cap A^-) & \\
& \hspace{-3cm}= \epsilon \times \PPP(\text{$w$ is pivotal for $A^+$ but not for $A^-$, and $A^-$ occurs}) & \\
& \hspace{-2.5cm} - \epsilon \times \PPP(\text{$w$ is pivotal for $A^-$ but not for $A^+$, and $A^+$ occurs}).
\end{align*}
Now, it suffices to compute the derivative of the function $t \mapsto \PPP_t(A^+ \cap A^-)$ by writing it as the composition of $t \mapsto (\hat{p}_v(t))$ and $(\hat{p}_v)_{v \in S} \mapsto \PPP(A)$.
\end{proof}

\begin{proof}[Proof of the theorem]

We now turn to the proof itself. It is divided into three main steps.

\subsubsection*{1. First simplifications}

Note first that by quasi-multiplicativity, we can restrict ourselves to $n = n_0(j)$. It also suffices to prove the result for some fixed $\PPP'$, with $\PPP$ varying: we thus assume that $p<1/2$, and take $\PPP'=\PP_p$. Denoting by $\hat{p}_v$ the parameters of $\PPP$, we have by hypothesis $\hat{p}_v \geq p$ for each site $v$. For technical reasons, we suppose that the sizes of annuli are powers of two: take $k_0$, $K$ such that $2^{k_0-1} < n_0 \leq 2^{k_0}$ and $2^K \leq N < 2^{K+1}$, then
$$\PP_p(A_{j,\sigma}(n_0,N)) \asymp \PP_p(A_{j,\sigma}(2^{k_0},2^K))$$
and the same is true for $\PPP$.

\begin{figure}
\begin{center}
\includegraphics[width=11cm]{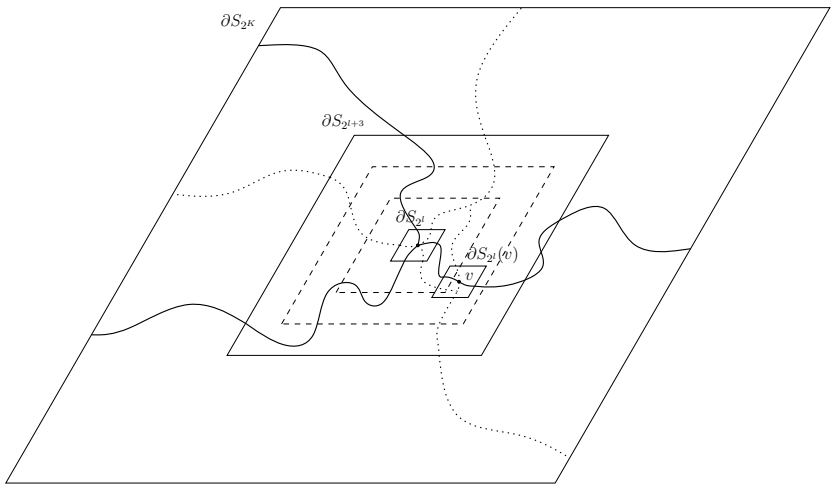}
\caption{If $v$ is pivotal, $4$ alternating arms arise locally.}
\label{local_arms}
\end{center}
\end{figure}

To estimate the change in probability when $p$ is replaced by $\hat{p}_v$, we will use the observation that the pivotal sites give rise to $4$ alternating arms locally (see Figure \ref{local_arms}). However, this does not work so nicely for the sites $v$ which are close to $\partial S_{2^{k_0}}$ or $\partial S_{2^K}$, so for the sake of simplicity we treat apart these sites. We perform the change $p \leadsto \hat{p}_v$ in $S_{2^{k_0},2^K} \setminus S_{2^{k_0+3},2^{K-3}}$. Note that the intermediate measure $\ppp$ so obtained is between $\PP_p$ and $\PP_{1-p}$, and that $\ppp(A_{j,\sigma}(2^{k_0+3},2^{K-3})) = \PP_p(A_{j,\sigma}(2^{k_0+3},2^{K-3}))$. We have
\begin{equation}
\ppp(A_{j,\sigma}(2^{k_0},2^K)) \asymp \ppp(A_{j,\sigma}(2^{k_0+3},2^{K-3}))
\end{equation}
and also
\begin{equation}
\PP_p(A_{j,\sigma}(2^{k_0},2^K)) \asymp \PP_p(A_{j,\sigma}(2^{k_0+3},2^{K-3})),
\end{equation}
which shows that it would be enough to prove the result with $\ppp$ instead of $\PP_p$.

\subsubsection*{2. Make appear the logarithmic derivative of the probability by applying Russo's formula}

The event $A_{j,\sigma}(2^{k_0},2^K)$ cannot be directly written as an intersection like in Russo's formula, since the order of the different arms is prescribed. To fix this difficulty, we impose the landing areas of the different arms on $\partial S_{2^K}$, \emph{ie} we fix some landing sequence $I' = I'_1, \ldots ,I'_j$ and we consider the event $\Ab_{j,\sigma}^{./I'}(2^{k_0},2^K)$. Since we know that
\begin{equation}
\ppp\big(A_{j,\sigma}(2^{k_0},2^K)\big) \asymp \ppp\big(\Ab_{j,\sigma}^{. / I'}(2^{k_0},2^K)\big),
\end{equation}
and also with $\PPP$ instead of $\ppp$, it is enough to prove the result for this particular landing sequence.

We study successively three cases. We begin with the case of one arm, which is slightly more direct than the two next ones -- however, only small adaptations are needed. We then consider the special case where $j$ is even and $\sigma$ alternating: due to the fact that any arm is surrounded by two arms of opposite color, the local four arms are always of the right length. We finally prove the result for any $j$ and any $\sigma$: a technical complication arises in this case, for which the notion of ``arms with defects'' is needed.

\medskip

\noindent \textbf{Case 1: $j=1$}

We consider first the case of one arm, and assume for instance $\sigma = B$. We introduce the family of measures $(\ppp_t)_{t \in [0,1]}$ with parameters
$$\tilde{p}_v(t) = t \hat{p}_v + (1-t) p$$
in $S_{2^{k_0+3},2^{K-3}}$, corresponding to a linear interpolation between $p$ and $\hat{p}_v$. For future use, note that $\ppp_t$ is between $\PP_p$ and $\PP_{1-p}$ for any $t \in [0,1]$. We have $\frac{d}{dt} \tilde{p}_v(t) = \hat{p}_v - p$ if $v \in S_{2^{k_0+3},2^{K-3}}$ (and $0$ otherwise\ldots), generalized Russo's formula (with just an increasing event -- take for instance $A^- = \Omega$) thus gives:
\begin{align*}
\frac{d}{dt} \ppp_t\big( \Ab_{1,\sigma}^{. / I'}(2^{k_0},2^K) \big) = \sum_{v \in S_{2^{k_0+3},2^{K-3}}} (\hat{p}_v - p) \ppp_t(\text{$v$ is pivotal for $\Ab_{1,\sigma}^{. / I'}(2^{k_0},2^K)$}).
\end{align*}

The key remark is that the summand can be expressed in terms of arm events: for probabilities, being pivotal is approximately the same as having a black arm, and four arms locally around $v$. Indeed, $\big[ v$ is pivotal for $\Ab_{1,\sigma}^{. / I'}(2^{k_0},2^K)$ $\big]$ \emph{iff}
\begin{enumerate}[(1)]
\item There exists an arm $r_1$ from $\partial S_{2^{k_0}}$ to $I'_1$, with $v \in r_1$; $r_1$ is black, with a possible exception in $v$ (\emph{$\Ab_{1,\sigma}^{. / I'}(2^{k_0},2^K)$ occurs when $v$ is black}).
\item There exists a path $c_1$ passing through $v$ and separating $\partial S_{2^{k_0}}$ from $I'_1$ ($c_1$ may be either a circuit around $\partial S_{2^{k_0}}$ or a path with extremities on $\partial S_{2^K}$); $c_1$ is white, except possibly in $v$ (\emph{there is no black arm from $\partial S_{2^{k_0}}$ to $I'_1$ when $v$ is white}).
\end{enumerate}
Put now a rhombus $R(v)$ around $v$: if it does not contain $0$, then $v$ is connected to $\partial R(v)$ by $4$ arms of alternating colors. Indeed, $r_1$ provides two black arms, and $c_1$ two white arms.

Look at the pieces of the black arm outside of $R(v)$: if $R(v)$ is not too large, we can expect them to be sufficiently large to enable us to reconstitute the whole arm. We would like that the two white arms are a good approximation of the whole circuit too. We thus take $R(v)$ of size comparable to the distance $d(0,v)$: if $2^{l+1} < \|v\|_{\infty} \leq 2^{l+2}$, we take $R(v)=S_{2^l}(v)$. It is not hard to check that $R(v) \subseteq S_{2^l,2^{l+3}}$ for this particular choice of $R(v)$ (see Figure \ref{local_arms}), so that for any $t \in [0,1]$,
\begin{align*}
\ppp_t(\text{$v$ is pivotal for $\Ab_{1,\sigma}^{. / I'}(2^{k_0},2^K)$}) & \\
& \hspace{-3cm} \leq \ppp_t\big( \{ \partial S_{2^{k_0}} \leadsto \partial S_{2^l} \} \cap \{ \partial S_{2^{l+3}} \leadsto \partial S_{2^K} \} \cap \{ v \leadsto^{4,\sigma_4} \partial S_{2^l}(v) \}\big) \\
& \hspace{-3cm} = \ppp_t\big( \partial S_{2^{k_0}} \leadsto \partial S_{2^l} \big) \ppp_t \big( \partial S_{2^{l+3}} \leadsto \partial S_{2^K} \big) \ppp_t \big( v \leadsto^{4,\sigma_4} \partial S_{2^l}(v) \big)
\end{align*}
by independence of the three events, since they are defined in terms of sites in disjoint sets. We can then make appear the original event by joining together the two first terms, using quasi-multiplicativity and extendability\footnote{Note that in the case of one arm, the extendability property, as well as the quasi-multiplicativity, are direct consequences of RSW and do not require the separation lemmas.}:
\begin{equation}
\ppp_t\big( \partial S_{2^{k_0}} \leadsto \partial S_{2^l} \big) \ppp_t \big( \partial S_{2^{l+3}} \leadsto \partial S_{2^K} \big) \leq C_2 \ppp_t \big( \Ab_{1,\sigma}^{. / I'}(2^{k_0},2^K) \big)
\end{equation}
for some $C_2$ universal. Hence\footnote{As we will see in the next sub-section (Proposition \ref{converse}), the converse bound also holds: the estimate obtained gives the exact order of magnitude for the summand.},
\begin{equation}
\ppp_t(\text{$v$ is pivotal for $\Ab_{1,\sigma}^{. / I'}(2^{k_0},2^K)$}) \leq C_2 \ppp_t\big( \Ab_{1,\sigma}^{. / I'}(2^{k_0},2^K) \big) \ppp_t \big( v \leadsto^{4,\sigma_4} \partial S_{2^l}(v) \big).
\end{equation}

We thus get
\begin{align*}
\frac{d}{dt} \ppp_t\big( \Ab_{1,\sigma}^{. / I'}(2^{k_0},2^K) \big) & \\
& \hspace{-2cm} \leq C_2 \sum_{v \in S_{2^{k_0+3},2^{K-3}}} (\hat{p}_v - p) \ppp_t\big( \Ab_{1,\sigma}^{. / I'}(2^{k_0},2^K) \big) \ppp_t \big( v \leadsto^{4,\sigma_4} \partial S_{2^l}(v) \big).
\end{align*}
Now dividing by $\ppp_t\big( \Ab_{1,\sigma}^{. / I'}(2^{k_0},2^K) \big)$, we make appear its logarithmic derivative in the left-hand side,
\begin{equation}
\frac{d}{dt} \log\big[\ppp_t\big( \Ab_{1,\sigma}^{. / I'}(2^{k_0},2^K) \big)\big] \leq C_2 \sum_{v \in S_{2^{k_0+3},2^{K-3}}} (\hat{p}_v - p) \ppp_t \big( v \leadsto^{4,\sigma_4} \partial S_{2^l}(v) \big),
\end{equation}
it thus suffices to show that for some $C_3$ universal,
\begin{equation}
\label{univ}
\int_0^1 \sum_{v \in S_{2^{k_0+3},2^{K-3}}} (\hat{p}_v - p) \: \ppp_t \big( v \leadsto^{4,\sigma_4} \partial S_{2^l}(v) \big) dt \leq C_3.
\end{equation}
We will prove it in the next step, but before that, we turn to the two other cases: even if the computations need to be modified, it is still possible to reduce the proof to this inequality.

\medskip

\noindent \textbf{Case 2: $j$ even and $\sigma$ alternating}

In this case,
\begin{equation}
\Ab_{j,\sigma}^{. / I'}(2^{k_0},2^K) = A^+ \cap A^-
\end{equation}
with $A^+ = A^+(2^{k_0},2^K) = \{$There exist $j/2$ disjoint black arms $r_1 : \partial S_{2^{k_0}} \leadsto I'_1, \: r_3 : \partial S_{2^{k_0}} \leadsto I'_3 \ldots\}$ and $A^- = A^-(2^{k_0},2^K) = \{$There exist $j/2$ disjoint white arms $r_2 : \partial S_{2^{k_0}} \leadsto^{\ast} I'_2, \: r_4 : \partial S_{2^{k_0}} \leadsto^{\ast} I'_4 \ldots\}$.

We then perform the change $p \leadsto \hat{p}_v$ in $S_{2^{k_0+3},2^{K-3}}$ linearly as before, which gives rise to the family of measures $(\ppp_t)_{t \in [0,1]}$, and generalized Russo's formula reads
\begin{align*}
\frac{d}{dt} \ppp_t\big(\Ab_{j,\sigma}^{. / I'}(2^{k_0},2^K)\big) \\
& \hspace{-3cm} = \sum_{v \in S_{2^{k_0+3},2^{K-3}}} (\hat{p}_v - p) \Big[ \ppp_t(\text{$v$ is pivotal for $A^+$ but not for $A^-$, and $A^-$ occurs}) & \\
& - \ppp_t(\text{$v$ is pivotal for $A^-$ but not for $A^+$, and $A^+$ occurs})\Big].
\end{align*}

We note that $\big[ v$ is pivotal for $A^+(2^{k_0},2^K)$ but not $A^-(2^{k_0},2^K)$, and $A^-(2^{k_0},2^K)$ occurs $\big]$ \emph{iff} for some $i' \in \{1,3 \ldots ,j-1\}$,
\begin{enumerate}[(1)]
\item There exist $j$ disjoint monochromatic arms $r_1, \ldots ,r_j$ from $\partial S_{2^{k_0}}$ to $I'_1, \ldots ,I'_j$, with $v \in r_{i'}$; $r_2, r_4, \ldots$ are white, and $r_1, r_3, \ldots$ are black, with a possible exception for $r_{i'}$ in $v$ (\emph{the event $\Ab_{j,\sigma}^{. / I'}(2^{k_0},2^K)$ is satisfied when $v$ is black}).
\item There exists a path $c_{i'}$ separating $\partial S_{2^{k_0}}$ from $I'_{i'}$; this path is white, except possibly in $v$ (\emph{$\partial S_{2^{k_0}}$ and $I'_{i'}$ are separated when $v$ is white}).
\end{enumerate}
If we take the same rhombus $R(v) \subseteq S_{2^l,2^{l+3}}$ around $v$, then $v$ is still connected to $\partial R(v)$ by 4 arms of alternating colors. Indeed, $r_{i'}$ provides two black arms, and $c_{i'}$ (which can contain parts of $r_{i'-1}$ or $r_{i'+1}$ -- see Figure \ref{local_arms}) provides the two white arms.

Hence for any $t \in [0,1]$,
\begin{align*}
\ppp_t(\text{$v$ is pivotal for $A^+$ but not $A^-$, and $A^-$ occurs}) & \\
& \hspace{-6cm} \leq \ppp_t\big( A_{j,\sigma}(2^{k_0},2^l) \cap \Ab_{j,\sigma}^{. / I'}(2^{l+3},2^K) \cap \{ v \leadsto^{4,\sigma_4} \partial S_{2^l}(v) \}\big) \\
& \hspace{-6cm} = \ppp_t\big( A_{j,\sigma}(2^{k_0},2^l) \big) \ppp_t \big( \Ab_{j,\sigma}^{. / I'}(2^{l+3},2^K) \big) \ppp_t \big( v \leadsto^{4,\sigma_4} \partial S_{2^l}(v) \big)
\end{align*}
by independence of the three events. We join together the two first terms using extendability and quasi-multiplicativity:
\begin{equation}
\ppp_t\big( A_{j,\sigma}(2^{k_0},2^l) \big) \ppp_t \big( \Ab_{j,\sigma}^{. / I'}(2^{l+3},2^K) \big) \leq C_1 \ppp_t\big( \Ab_{j,\sigma}^{. / I'}(2^{k_0},2^K) \big)
\end{equation}
for some $C_2$ universal. We thus obtain
\begin{align*}
\ppp_t(\text{$v$ is pivotal for $A^+$ but not $A^-$, and $A^-$ occurs}) & \\
& \hspace{-4cm} \leq C_1 \ppp_t\big(\Ab_{j,\sigma}^{. / I'}(2^{k_0},2^K)\big) \ppp_t \big( v \leadsto^{4,\sigma_4} \partial S_{2^l}(v) \big).
\end{align*}

If we then do the same manipulation on the second term of the sum, we get
\begin{align*}
\bigg| \frac{d}{dt} \ppp_t\big(\Ab_{j,\sigma}^{. / I'}(2^{k_0},2^K)\big) \bigg| & \\
& \hspace{-2cm} \leq 2 C_1 \sum_{v \in S_{2^{k_0+3},2^{K-3}}} (\hat{p}_v - p) \ppp_t\big(\Ab_{j,\sigma}^{. / I'}(2^{k_0},2^K)\big) \ppp_t \big( v \leadsto^{4,\sigma_4} \partial S_{2^l}(v) \big),
\end{align*}
and if we divide by $\ppp_t\big(\Ab_{j,\sigma}^{. / I'}(2^{k_0},2^K)\big)$,
\begin{equation}
\bigg| \frac{d}{dt} \log\big[\ppp_t\big(\Ab_{j,\sigma}^{. / I'}(2^{k_0},2^K)\big)\big] \bigg| \leq 2 C_1 \sum_{v \in S_{2^{k_0+3},2^{K-3}}} (\hat{p}_v - p) \ppp_t \big( v \leadsto^{4,\sigma_4} \partial S_{2^l}(v) \big).
\end{equation}
As promised, we have thus reduced this case to Eq.(\ref{univ}).

\medskip

\noindent \textbf{Case 3: Any $j$, $\sigma$}

In the general case, a minor complication may arise, coming from consecutive arms of the same color: indeed, the property of being pivotal for a site $v$ does not always give rise to four arms in $R(v)$, but to some more complex event $E(v)$ (see Figure \ref{non_alternating}). If $v$ is on $r_i$, and this arm is black for instance, there are still two black arms coming from $r_i$, but the two white arms do not necessarily reach $\partial R(v)$, since they can encounter neighboring black arms.

\begin{figure}
\begin{center}
\includegraphics[width=9cm]{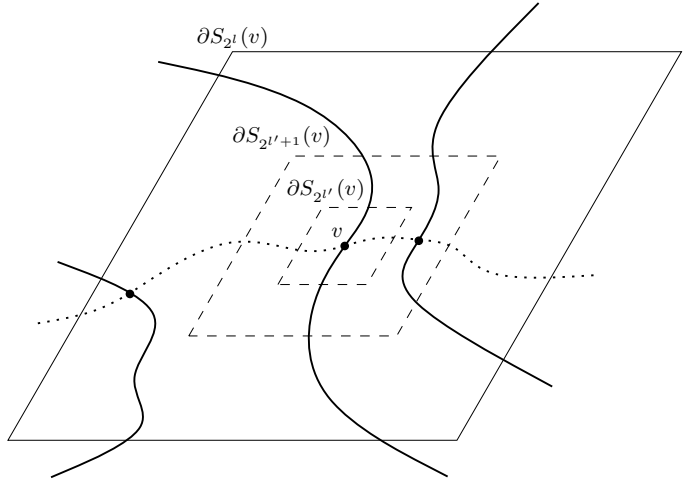}
\caption{More complex events may arise when $\sigma$ is not alternating.}
\label{non_alternating}
\end{center}
\end{figure}

We first introduce an event for which the property of being pivotal is easier to formulate. We group consecutive arms of the same color in ``packs'': if $(r_{i_q},r_{i_q+1},\ldots,r_{i_q+l_q-1})$ is such a sequence of arms, say black, we take an interval $\tilde{I}_q$ covering all the $I_i$ for $i_q \leq i \leq i_q+l_q-1$ and replace the condition ``$r_i \leadsto I_i$ for all $i_q \leq i \leq i_q+l_q-1$'' by ``$r_i \leadsto \tilde{I}_q$ for all $i_q \leq i \leq i_q+l_q-1$''. We construct in this way an event $\tilde{A} = \tilde{A}^+ \cap \tilde{A}^-$: since it is intermediate between $\Ab_{j,\sigma}^{. / I'}(2^{k_0},2^K)$ and $A_{j,\sigma}(2^{k_0},2^K)$, we have
$$\ppp_t(\tilde{A}) \asymp \ppp_t(\Ab_{j,\sigma}^{. / I'}(2^{k_0},2^K)).$$

This new definition allows to use Menger's theorem (see \cite{D_book}, Theorem 3.3.1): $\big[ v$ is pivotal for $\tilde{A}^+$ but not $\tilde{A}^-$, and $\tilde{A}^-$ occurs $\big]$ \emph{iff} for some arm $r_{i'}$ in a black pack $(r_{i_q},r_{i_q+1},\ldots,r_{i_q+l_q-1})$,
\begin{enumerate}[(1)]
\item There exist $j$ disjoint monochromatic arms $r_1, \ldots ,r_j$ from $\partial S_{2^{k_0}}$ to the $\tilde{I}_q$ (an appropriate number of arms for each of these intervals), with $v \in r_{i'}$; all of these arms are of the prescribed color, with a possible exception for $r_{i'}$ in $v$ (\emph{$\tilde{A}$ occurs when $v$ is black}).
\item There exists a path $c_{i'}$ separating $\partial S_{2^{k_0}}$ from $\tilde{I}_{q}$; this path is white, except in (at most) $l_q-1$ sites, and also possibly in $v$ (\emph{$\partial S_{2^{k_0}}$ and $\tilde{I}_{q}$ are separated when $v$ is white}).
\end{enumerate}

Now we take the same rhombus $R(v) \subseteq S_{2^l,2^{l+3}}$ around $v$: if there are four arms $v \leadsto^{4,\sigma_4} \partial S_{2^{l-1}}(v)$, we are OK. Otherwise, if $l'$, $1 \leq l' \leq l-2$, is such that the defect on $c_{i'}$ closest to $v$ is in $S_{2^{l'+1}}(v) \setminus S_{2^{l'}}(v)$, then there are $4$ alternating arms $v \leadsto^{4,\sigma_4} \partial S_{2^{l'}}(v)$, and also $6$ arms $\partial S_{2^{l'+1}}(v) \leadsto^{6,\sigma'_6} \partial S_{2^l}(v)$ having at most $j$ defects, with the notation $\sigma'_6=BBWBBW$. We denote by $E(v)$ the corresponding event: $E(v):=$ $\{$There exists $l' \in \{1,\ldots,l-2\}$ such that $v \leadsto^{4,\sigma_4} \partial S_{2^{l'}}(v)$ and $\partial S_{2^{l'+1}}(v) \leadsto^{6,\sigma'_6,(j)} \partial S_{2^l}(v)$$\}$ $\cup$ $\{$$v \leadsto^{4,\sigma_4} \partial S_{2^{l-1}}(v)$$\}$.

For the $6$ arms with defects, Proposition \ref{defects} applies and the probability remains roughly the same, with just an extra logarithmic correction:
\begin{align*}
\ppp_t \big( \partial S_{2^{l'+1}}(v) \leadsto^{6,\sigma'_6,(j)} \partial S_{2^l}(v) \big) & \\
& \hspace{-4cm} \leq C_1 (l-l')^j \ppp_t \big( \partial S_{2^{l'+1}}(v) \leadsto^{6,\sigma'_6} \partial S_{2^l}(v) \big)\\
& \hspace{-4cm} \leq C_1 (l-l')^j \ppp_t \big( \partial S_{2^{l'+1}}(v) \leadsto^{4,\sigma_4} \partial S_{2^l}(v) \big) \ppp_t \big( \partial S_{2^{l'+1}}(v) \leadsto^{BB} \partial S_{2^l}(v) \big)\\
& \hspace{-4cm} \leq C_2 (l-l')^j \ppp_t \big( \partial S_{2^{l'+1}}(v) \leadsto^{4,\sigma_4} \partial S_{2^l}(v) \big) 2^{-\alpha'(l-l')}
\end{align*}
using Reimer's inequality (its consequence Eq.(\ref{csq_Reimer})) and the a-priori bound for one arm (Eq.(\ref{apriori})).

It implies that
\begin{equation}
\ppp_t \big( E(v) \big) \leq C_5 \ppp_t \big( v \leadsto^{4,\sigma_4} \partial S_{2^l}(v) \big)
\end{equation}
for some universal constant $C_5$: indeed, by quasi-multiplicativity,
\begin{align*}
\sum_{l'=1}^{l-2} \ppp_t \big( v \leadsto^{4,\sigma_4} \partial S_{2^{l'}}(v) \big) \ppp_t \big( \partial S_{2^{l'+1}}(v) \leadsto^{6,\sigma'_6,(j)} \partial S_{2^l}(v) \big) &\\
& \hspace{-8.3cm} \leq C_2 \sum_{l'=1}^{l-2} \ppp_t \big( v \leadsto^{4,\sigma_4} \partial S_{2^{l'}}(v) \big) \ppp_t \big( \partial S_{2^{l'+1}}(v) \leadsto^{4,\sigma_4} \partial S_{2^l}(v) \big) (l-l')^j 2^{-\alpha'(l-l')}\\
& \hspace{-8.3cm} \leq C_3 \ppp_t \big( v \leadsto^{4,\sigma_4} \partial S_{2^l}(v) \big) \sum_{l'=1}^{l-2} (l-l')^j 2^{-\alpha'(l-l')}\\
& \hspace{-8.3cm} \leq C_4 \ppp_t \big( v \leadsto^{4,\sigma_4} \partial S_{2^l}(v) \big),
\end{align*}
since $\sum_{l'=1}^{l-2} (l-l')^j 2^{-\alpha'(l-l')} \leq \sum_{r=1}^{\infty} r^j 2^{-\alpha'r} < \infty$.

The reasoning is then identical:
\begin{align*}
\ppp_t(\text{$v$ is pivotal for $\tilde{A}^+$ but not $\tilde{A}^-$, and $\tilde{A}^-$ occurs}) & \\
& \hspace{-6cm} \leq \ppp_t\big( A_{j,\sigma}(2^{k_0},2^l) \big) \ppp_t \big( A_{j,\sigma}(2^{l+3},2^K) \big) \ppp_t \big( E(v) \big)\\
& \hspace{-6cm} \leq C_6 \ppp_t\big( A_{j,\sigma}(2^{k_0},2^K) \big) \ppp_t \big( v \leadsto^{4,\sigma_4} \partial S_{2^l}(v) \big),
\end{align*}
and using $\ppp_t\big( A_{j,\sigma}(2^{k_0},2^K) \big) \leq C_7 \ppp_t(\tilde{A})$, we get
\begin{equation}
\bigg| \frac{d}{dt} \log\big[\ppp_t\big( \tilde{A} \big)\big] \bigg| \leq C_8 \sum_{v \in S_{2^{k_0+3},2^{K-3}}} (\hat{p}_v - p) \ppp_t \big( v \leadsto^{4,\sigma_4} \partial S_{2^l}(v) \big).
\end{equation}
Once again, Eq.(\ref{univ}) would be sufficient.

\subsubsection*{3. Final summation}

We now prove Eq.(\ref{univ}), \emph{ie} that for some universal constant $C_1$,
\begin{equation}
\label{goal}
\int_0^1 \sum_{v \in S_{2^{k_0+3},2^{K-3}}} (\hat{p}_v - p) \: \ppp_t \big( v \leadsto^{4,\sigma_4} \partial S_{2^l}(v) \big) dt \leq C_1.
\end{equation}
Recall that Russo's formula allows to count $4$-arm sites: for any $N$ and any parameters $(\bar{p}_v)$ between $p$ and $1-p$,
\begin{equation}
\int_0^1 \sum_{v \in S_N} (\bar{p}_v - p) \: \bar{\PP}_t \big( v \leadsto^{4,\sigma_4} \partial S_N \big) dt = \bar{\PP}(\mathcal{C}_H(S_N)) - \PP_p (\mathcal{C}_H(S_N)) \leq 1.
\end{equation}
This is essentially the only relation we have at our disposal, the end of the proof consists in using it in a clever way.

Roughly speaking, when applied to $N=L(p)$, this relation gives that $(p-1/2) N^2 \pi_4(N) \leq 1$, since all the sites give a contribution of order
\begin{equation}
\label{heuristic}
\bar{\PP}_t \big( 0 \leadsto^{4,\sigma_4} \partial S_{N/2} \big) \asymp \pi_4(N).
\end{equation}
This corresponds more or less to the sites in the ``external annulus'' in Eq.(\ref{goal}). Now each time we get from an annulus to the next inside it, the probability to have $4$ arms is multiplied by $2^{\alpha_4} \approx 2^{5/4}$, while the number of sites is divided by $4$, so that things decay exponentially fast, and the sum of Eq.(\ref{goal}) is bounded by something like
$$\sum_{j=3}^{K-k_0-4} (2^{5/4-2})^{K-j} \leq \sum_{q=0}^{\infty} (2^{-3/4})^q < \infty.$$
We have to be more cautious, in particular Eq.(\ref{heuristic}) does not trivially hold, since we do not know at this point that the probability of having $4$ arms remains of the same order on a scale $L(p)$, and the estimate for $4$ arms only gives a logarithmic equivalence. The a-priori estimate coming from the $5$-arm exponent will allow us to circumvent these difficulties. We also need to take care of the boundary effects.

Assume that $v \in S_{2^{l+1},2^{l+2}}$ as before. We subdivide this annulus into $12$ sub-boxes of size $2^{l+1}$ (see figure) $R^i_{2^{l+1}}$ ($i=1,\ldots,12$). We then associate to each of these boxes a slightly enlarged box $R'^i_{2^{l+1}}$, of size $2^{l+2}$. At least one of these boxes contains $v$: we denote it by $R'(v)$. Since
$$\{v \leadsto^{4,\sigma_4} \partial S_{2^{l+1}}(v)\} \subseteq \{v \leadsto^{4,\sigma_4} \partial R'(v)\} \subseteq \{v \leadsto^{4,\sigma_4} \partial S_{2^l}(v)\},$$
we have
$$\ppp_t \big( v \leadsto^{4,\sigma_4} \partial S_{2^l}(v) \big) \asymp \ppp_t \big( v \leadsto^{4,\sigma_4} \partial R'(v) \big).$$
We thus have to find an upper bound for
\begin{equation} \label{big_sum}
\sum_{j=k_0+3}^{K-4} \sum_{i=1}^{12} \int_0^1 \sum_{v \in R'^i_{2^j}} (\hat{p}_v - p) \: \ppp_t \big( v \leadsto^{4,\sigma_4} \partial R'^i_{2^j} \big) dt.
\end{equation}

\begin{figure}
\begin{center}
\includegraphics[width=8cm]{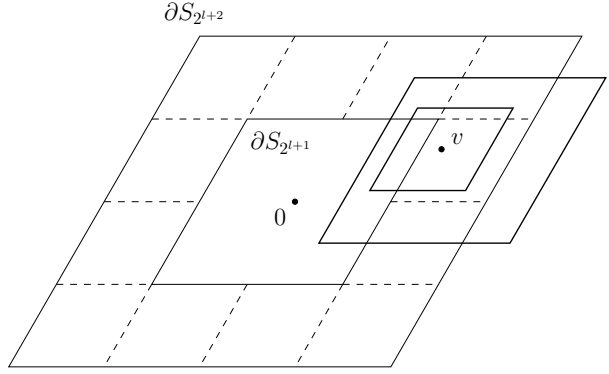}
\caption{\label{summation}We replace $R(v) = S_{2^l}(v)$ by one of the $R'^i_{2^{l+1}}$ ($i=1,\ldots,12$).}
\end{center}
\end{figure}

For that purpose, we will prove that for $i=1,\ldots,12$,
$$S_j^{i,(4)} := \sum_{v \in R'^i_{2^j}} (\hat{p}_v - p) \: \ppp_t \big( v \leadsto^{4,\sigma_4} \partial R'^i_{2^j} \big)$$
indeed decays fast when, starting from $j=K-4$, we make $j$ decrease. For that, we duplicate the parameters in the box $R'^i_{2^j}$ periodically inside $S_{2^{K-3}}$: this gives rise to a new measure $\bar{\PP}$ inside $S_{2^K}$ (to completely define it, simply take $\bar{p}_v = p$ outside of $S_{2^{K-3}}$). This measure contains $2^{2(K-j-3)}$ copies $(R'')$ of the original box, and we know that
\begin{equation} \label{csq_russo}
\int_0^1 \sum_{v \in S_{2^{K-3}}} (\bar{p}_v - p) \: \bar{\PP}_t \big( v \leadsto^{4,\sigma_4} \partial S_{2^K} \big) dt \leq 1.
\end{equation}
We have
\begin{align*}
\sum_{v \in S_{2^{K-3}}} (\bar{p}_v-p) \bar{\PP}_t(v \leadsto^{4,\sigma_4} \partial S_{2^K}) & \\
& \hspace{-3cm} \asymp \sum_{v \in S_{2^{K-3}}} (\bar{p}_v-p) \bar{\PP}_t \big( v \leadsto^{4,\sigma_4} \partial R'(v) \big) \bar{\PP}_t \big( \partial R'(v) \leadsto^{4,\sigma_4} \partial S_{2^K} \big)\\
& \hspace{-3cm} \asymp \bigg( \sum_{R''} \bar{\PP}_t \big( \partial R'' \leadsto^{4,\sigma_4} \partial S_{2^K} \big) \bigg) S_j^{i,(4)}.
\end{align*}
Hence, using Reimer's inequality and the a-priori bound for one arm,
\begin{align*}
\sum_{v \in S_{2^{K-3}}} (\bar{p}_v-p) \bar{\PP}_t(v \leadsto^{4,\sigma_4} \partial S_{2^K}) & \\
& \hspace{-3.5cm} \geq C_1 \bigg( \sum_{R''} \bar{\PP}_t \big( \partial R'' \leadsto^{5,\sigma_5} \partial S_{2^K} \big) \bar{\PP}_t \big( \partial R'' \leadsto \partial S_{2^K} \big)^{-1} \bigg) S_j^{i,(4)}\\
& \hspace{-3.5cm} \geq C_2 2^{\alpha'(K-j)} S_j^{i,(4)} \bigg( \sum_{R''} \bar{\PP}_t \big( \partial R'' \leadsto^{5,\sigma_5} \partial S_{2^K} \big) \bigg).
\end{align*}
The same manipulation for $5$ arms gives, with $\tilde{S}_j^{i,(5)} = \sum_{v \in R'^i_{2^j}} \ppp_t \big( v \leadsto^{5,\sigma_5} \partial R'^i_{2^j} \big)$,
\begin{equation}
\sum_{v \in S_{2^{K-3}}} \bar{\PP}_t(v \leadsto^{5,\sigma_5} \partial S_{2^K}) \asymp \bigg( \sum_{R''} \bar{\PP}_t \big( \partial R'' \leadsto^{5,\sigma_5} \partial S_{2^K} \big) \bigg) \tilde{S}_j^{i,(5)}.
\end{equation}
We know from Lemma \ref{sum5arms} that $\sum_{v \in S_{2^{K-3}}} \bar{\PP}_t(v \leadsto^{5,\sigma_5} \partial S_{2^K})$ and $\tilde{S}_j^{i,(5)} \asymp 1$, and thus
\begin{equation}
\sum_{R''} \bar{\PP}_t \big( \partial R'' \leadsto^{5,\sigma_5} \partial S_{2^K} \big) \asymp 1.
\end{equation}
This entails that
$$S_j^{i,(4)} \leq C_3 2^{- \alpha'(K-j)} \sum_{v \in S_{2^{K-3}}} (\bar{p}_v-p) \bar{\PP}_t(v \leadsto^{4,\sigma_4} \partial S_{2^K}),$$
and finally, by integrating and using Eq.(\ref{csq_russo}),
$$\int_0^1 \sum_{v \in R'^i_{2^j}} (\hat{p}_v - p) \: \ppp_t \big( v \leadsto^{4,\sigma_4} \partial R'^i_{2^j} \big) dt
\leq C_3 2^{- \alpha'(K-j)}.$$
The sum of Eq.(\ref{big_sum}) is thus less than
$$\sum_{j=k_0+3}^{K-4} 12 C_3 2^{- \alpha'(K-j)} \leq C_4 \sum_{r=0}^{\infty} 2^{- \alpha'r} < \infty,$$
which completes the proof.
\end{proof}

\begin{remark}
We will use this theorem in the next section to relate the so-called ``characteristic functions'' to the arm exponents at criticality. We will have use in fact only for the two cases $j=1$ and $j=4$, $\sigma=\sigma_4$: the general case (3rd case in the previous proof) will thus not be needed there. It is however of interest for other applications, for instance to say that for an interface in near-critical percolation, the dimension of the accessible perimeter is the same as at criticality: this requires the case $j=3$, $\sigma=\sigma_3$.
\end{remark}

\subsection{Some complements}

\subsubsection*{Theorem for more general annuli}

We will sometimes need a version of Theorem \ref{armnear} with non concentric rhombi:
\begin{equation}
\PPP(\partial r \leadsto \partial R) \asymp \PPP(\partial S_n \leadsto \partial S_N).
\end{equation}
for $S_{N/\tau} \subseteq R \subseteq S_N$ and $r \subseteq S_{N/2\tau}$. It results from the remark on more general annuli (Eq.(\ref{non_concentric})) combined with Theorem \ref{armnear} applied to $\PPP$ and translations of it. In particular, for any fixed $\eta > 0$,
\begin{equation}
\PPP(\partial S_n(v) \leadsto \partial S_N) \asymp \PPP(\partial S_n \leadsto \partial S_N)
\end{equation}
uniformly in $v \in S_{(1-\eta) N}$.

\subsubsection*{A complementary bound}

Following the same lines as in the previous proof, we can get a bound in the other direction:

\begin{proposition} \label{converse}
There exists some universal constant $\tilde{C} > 1$ such that for all $p > 1/2$,
\begin{equation}
\PP_p\big(0 \leadsto \partial S_{L(p)}\big) \geq \tilde{C} \: \PP_{1/2}\big(0 \leadsto \partial S_{L(p)}\big).
\end{equation}
\end{proposition}

In other words, the one-arm probability varies of a non-negligible amount, like the crossing probability: we begin to ``feel'' the super-critical behavior.

\begin{proof}
Take $K$ such that $2^K \leq L(p) < 2^{K+1}$ and $(\PPP_t)$ the linear interpolation between $\PP_{1/2}$ and $\PP_p$. By gluing arguments, for $A = \{ 0 \leadsto \partial S_{L(p)} \}$, for any $v \in S_{2^{K-4},2^{K-3}}$,
\begin{align*}
\ppp_t(\text{$v$ is pivotal for $A$}) & \\
& \hspace{-2cm} \geq C_1 \ppp_t\big( 0 \leadsto \partial S_{2^{K-5}} \big) \ppp_t \big( \partial S_{2^{K-2}} \leadsto \partial S_{L(p)} \big) \ppp_t \big( v \leadsto^{4,\sigma_4} \partial S_{2^{K-5}}(v) \big)\\
& \hspace{-2cm} \geq C_2 \ppp_t\big( 0 \leadsto \partial S_{2^K} \big) \ppp_t \big( v \leadsto^{4,\sigma_4} \partial S_{2^{K-5}}(v) \big),
\end{align*}
so that
\begin{align*}
\frac{d}{dt} \log \big[ \ppp_t( A ) \big] & \geq \sum_{v \in S_{2^{K-4},2^{K-3}}} (p - 1/2) \PPP_t \big( v \leadsto^{4,\sigma_4} \partial S_{2^{K-5}}(v) \big)\\
& \geq C_3 (p-1/2) L(p)^2 \PPP_t \big( 0 \leadsto^{4,\sigma_4} \partial S_{L(p)} \big),
\end{align*}
since each of the sites $v \in S_{2^{K-4},2^{K-3}}$ produces a contribution of order $\PPP_t \big( 0 \leadsto^{4,\sigma_4} \partial S_{L(p)} \big)$. Proposition \ref{L_exp}, proved later\footnote{This does not raise any problem since we have included this complementary bound only for the sake of completeness, and we will not use it later.}, allows to conclude.
\end{proof}

\section{Consequences for the characteristic functions\label{charac}}

\subsection{Different characteristic lengths\label{different}}

Roughly speaking, a \emph{characteristic length} is a quantity intended to measure a ``typical'' scale of the system. There may be several natural definitions of such a length, but we usually expect the different possible definitions to produce lengths that are of the same order of magnitude. For two-dimensional percolation, the three most common definitions are the following:

\subsubsection*{Finite-size scaling}

The lengths $L_{\epsilon}$ that we have used throughout the paper, introduced in \cite{CCF}, are known as ``finite-size scaling characteristic lengths'':
\begin{equation}
L_{\epsilon}(p)=\left\lbrace
\begin{array}{c l}
\min\{n \text{\: s.t. \:} \PP_p(\mathcal{C}_H([0,n] \times [0,n])) \leq \epsilon \} & \text{when}\: p<1/2, \\[2mm]
\min\{n \text{\: s.t. \:} \PP_p(\mathcal{C}_H^*([0,n] \times [0,n])) \leq \epsilon \} & \text{when}\: p>1/2. \\
\end{array} \right.
\end{equation}

\subsubsection*{Mean radius of a finite cluster}

The (quadratic) mean radius measures the ``typical'' size of a finite cluster. It can be defined by the formula
\begin{equation}
\xi(p) = \Bigg[ \frac{1}{\EE_p \big[|C(0)| ; |C(0)|<\infty\big]} \sum_x \|x\|_{\infty}^2 \PP_p\big(0 \leadsto x , |C(0)|<\infty\big) \Bigg]^{1/2}.
\end{equation}

\subsubsection*{Connection probabilities}

A third possible definition would be via the rate of decay of correlations. Take first $p<1/2$ for example. For two sites $x$ and $y$, we consider the connection probability between them
\begin{equation}
\tau_{x,y} := \PP_p\big( x \leadsto y \big),
\end{equation}
and then
\begin{equation}
\tau_n := \sup_{x \in \partial S_n} \tau_{0,x},
\end{equation}
the maximum connection probability between sites at distance $n$ (using translation invariance). For any $n, m \geq 0$, we have
$$\tau_{n+m} \geq \tau_n \tau_m,$$
in other words $(-\log \tau_n)_{n \geq 0}$ is sub-additive, which implies the existence of a constant $\tilde{\xi}(p)$ such that
\begin{equation}
- \frac{\log \tau_n}{n} \longrightarrow \frac{1}{\tilde{\xi}(p)}
\end{equation}
when $n \to \infty$. Note the following a-priori bound:
\begin{equation}
\PP_p\big( 0 \leadsto x \big) \leq e^{-\|x\|_{\infty}/\tilde{\xi}(p)}.
\end{equation}

For $p>1/2$, we simply use the symmetry $p \leftrightarrow 1-p$: we define
\begin{equation}
\tau^{\ast}_n := \sup_{x \in \partial S_n} \PP_p\big( 0 \leadsto^{\ast} x \big)
\end{equation}
and then $\tilde{\xi}(p)$ in the same way. We have in this case
\begin{equation}
\PP_p\big( 0 \leadsto^{\ast} x \big) \leq e^{-\|x\|_{\infty}/\tilde{\xi}(p)}.
\end{equation}

Note that the symmetry $p \leftrightarrow 1-p$ gives immediately
$$\tilde{\xi}(p) = \tilde{\xi}(1-p).$$

\subsubsection*{Relation between the different lengths}

As expected, these characteristic lengths turn out to be all of the same order of magnitude: we will prove in Section \ref{Lsection} that $L_{\epsilon} \asymp L_{\epsilon'}$ for any two $\epsilon, \epsilon' \in (0,1/2)$, in Section \ref{expsection} that $L \asymp \tilde{\xi}$, and in Section \ref{xi} that $L \asymp \xi$.

\subsection{Main critical exponents}

We focus here on three functions commonly used to describe the macroscopic behavior of percolation. We have already encountered some of them:

\begin{enumerate}[(i)]
\item $\xi(p) = \Bigg[ \frac{1}{\EE_p \big[|C(0)| ; |C(0)|<\infty\big]} \sum_x \|x\|_{\infty}^2 \PP_p\big(0 \leadsto x , |C(0)|<\infty\big) \Bigg]^{1/2}$ the mean radius of a finite cluster.

\item $\theta(p) := \PP_p(0 \leadsto \infty)$. This function can be viewed as the density of the infinite cluster $C_{\infty}$, in the following sense:
\begin{equation}
\frac{1}{|S_N|} \big|S_N \cap C_{\infty}\big| \stackrel{\text{a.s.}}{\longrightarrow} \theta(p)
\end{equation}
when $N \to \infty$.

\item $\chi(p) = \EE_p \big[|C(0)| ; |C(0)| < \infty\big]$ the average size of a finite cluster.
\end{enumerate}

\begin{theorem}[Critical exponents]
The following power-law estimates hold:

\begin{enumerate}[(i)]
\item When $p \to 1/2$,
\begin{equation}
\xi(p) \asymp L(p) \approx |p-1/2|^{-4/3}.
\end{equation}

\item When $p \to 1/2^+$,
\begin{equation}
\theta(p) \approx (p-1/2)^{5/36}.
\end{equation}

\item When $p \to 1/2$,
\begin{equation}
\chi(p) \approx |p-1/2|^{-43/18}.
\end{equation}
\end{enumerate}
\end{theorem}

The corresponding exponents are usually denoted by (respectively) $\nu$, $\beta$ and $\gamma$. This theorem is proved in the next sub-sections by combining the arm exponents for critical percolation with the estimates established for near-critical percolation.

\subsection{Critical exponent for $L$\label{Lsection}}

We derive here the exponent for $L_{\epsilon}(p)$ by counting the sites which are pivotal for the existence of a crossing in a box of size $L_{\epsilon}(p)$. These pivotal sites are exactly those for which the $4$-arm event $\Ab_{4,\sigma_4}^{./\bar{I}}$ with alternating colors ($\sigma_4 = BWBW$) and sides ($\bar{I}=$ right, top, left and bottom sides):

\begin{proposition}[\cite{Ke4,SmW}] \label{L_exp}
For any fixed $\epsilon \in (0,1/2)$, the following equivalence holds:
\begin {equation} \label{nbpivot}
|p-1/2| \big(L_{\epsilon}(p)\big)^2 \pi_4(L_{\epsilon}(p)) \asymp 1.
\end{equation}
\end{proposition}

Recall now the value $\alpha_4=5/4$ of the $4$-arm exponent, stated in Theorem \ref{armcrit}. If we plug it into Eq.(\ref{nbpivot}), we get the value of the characteristic length exponent: when $p \to 1/2$,
$$1 \approx |p-1/2| \big(L_{\epsilon}(p)\big)^2 \big(L_{\epsilon}(p)\big)^{-5/4} = |p-1/2| \big(L_{\epsilon}(p)\big)^{3/4},$$
so that indeed
$$L_{\epsilon}(p) \approx |p-1/2|^{-4/3}.$$

\begin{proof}
For symmetry reasons, we can assume that $p > 1/2$. The proof goes as follows. We first apply Russo's formula to estimate the variation in probability of the event $\mathcal{C}_H([0,L(p)] \times [0,L(p)])$ between $1/2$ and $p$, which makes appear the events $\bar{A}_{4,\sigma_4}^{./\bar{I}}$. By construction of $L(p)$, the variation of the crossing event is of order $1$, and the sites that are ``not too close to the boundary'' (such that none of the $4$ arms can become too small -- see Figure \ref{fig_pivotal}) each produce a contribution of the same order by Theorem \ref{armnear}: proving that they all together produce a non-negligible variation in the crossing probabilities will thus imply the result. For that, we need the following lemma:
\begin{lemma}
For any $\delta >0$, there exists $\eta_0 >0$ such that for all $p$, $\PPP$ between $\PP_p$ and $\PP_{1-p}$, we have: for any parallelogram $[0,n]\times[0,m]$ with sides $n, m \leq L(p)$ and aspect ratio less than $2$ (\emph{ie} such that $1/2 \leq n/m \leq 2$), for any $\eta \leq \eta_0$,
\begin{equation}
\big| \PPP (\mathcal{C}_H([0,n] \times [0,m])) - \PPP (\mathcal{C}_H([0,(1+\eta) n] \times [0,m])) \big| \leq \delta.
\end{equation}
\end{lemma}

\begin{figure}
\begin{center}
\includegraphics[width=9cm]{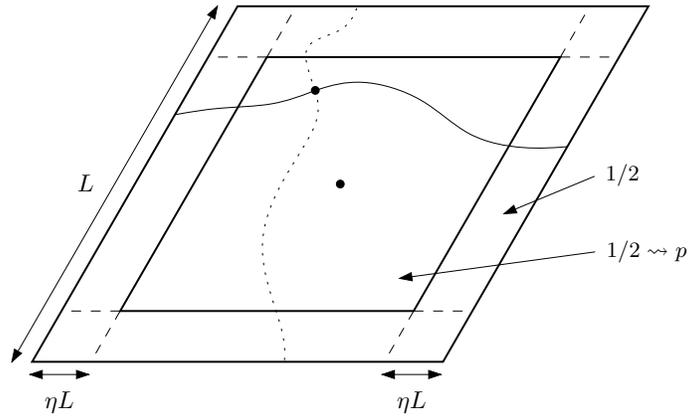}
\caption{We restrict to the sites at distance at least $\eta L$ from the boundary of $[0,L]^2$: these sites produce contributions of the same order, since the $4$ arms stay comparable in size.}
\label{fig_pivotal}
\end{center}
\end{figure}

\begin{proof}[Proof of lemma]
First, we clearly have
$$\PPP (\mathcal{C}_H([0,n] \times [0,m])) \geq \PPP (\mathcal{C}_H([0,(1+\eta) n] \times [0,m])).$$

\begin{figure}
\begin{center}
\includegraphics[width=7.5cm]{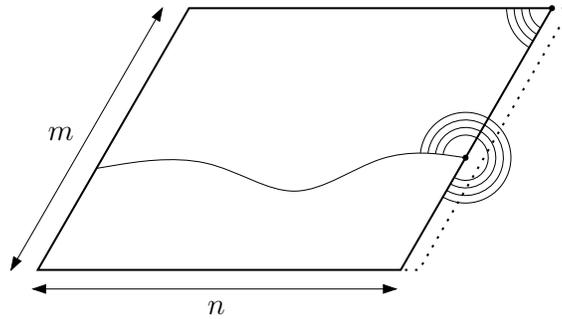}
\caption{We extend a crossing of $[0,n] \times [0,m]$ into a crossing of $[0,(1+\eta) n] \times [0,m]$ by applying RSW in concentric annuli.}
\label{extension}
\end{center}
\end{figure}

For the converse bound, we use the same idea as for Lemma \ref{well_sep}, we apply RSW in concentric annuli (see Figure \ref{extension}). By considering (parts of) annuli centered on the top right corner of $[0,n] \times [0,m]$, with radii between $\eta^{3/4} n$ and $\sqrt{\eta} n$, we see that the probability for a crossing to arrive at a distance less than $\eta^{3/4} n$ from this corner is at most $\delta/100$ for $\eta_0$ small enough. Assume this is not the case, condition on the lowest crossing and apply RSW in annuli between scales $\eta n$ and $\eta^{3/4} n$: if $\eta_0$ is sufficiently small, with probability at least $1 - \delta/100$, this crossing can be extended into a crossing of $[0,(1+\eta) n] \times [0,m]$.
\end{proof}

Let us return to the proof of the proposition. Take $\eta_0$ associated to $\delta = \epsilon / 100$ by the lemma, and assume that instead of performing the change $1/2 \leadsto p$ in the whole box $[0,L(p)]^2$, we make it only for the sites in the sub-box $[\eta L(p),(1-\eta) L(p)]^2$, for $\eta=\eta_0/4$. It amounts to consider the measure $\PPP^{(\eta)}$ with parameters 
\begin{equation}
\hat{p}^{(\eta)}_v=\left|
\begin{array}{c l}
p & \text{if}\: v \in [\eta L(p),(1-\eta) L(p)]^2, \\[2mm]
1/2 & \text{otherwise.}
\end{array} \right.
\end{equation}
We are going to prove that $\PPP^{(\eta)}(\mathcal{C}_H([0,L(p)]^2))$ and $\PP_p(\mathcal{C}_H([0,L(p)]^2))$ are very close by showing that they are both very close to $\PP_p(\mathcal{C}_H([\eta L(p),(1-\eta) L(p)]^2)) = \PPP^{(\eta)}(\mathcal{C}_H([\eta L(p),(1-\eta) L(p)]^2))$. Indeed, for any $\ppp \in \{\PPP^{(\eta)}, \PP_p\}$, we have by the lemma
\begin{align*}
\ppp(\mathcal{C}_H([0,L(p)]^2)) & \leq \ppp(\mathcal{C}_H([\eta L(p),(1-\eta) L(p)] \times [0,L(p)]))\\
& = 1 - \ppp(\mathcal{C}^*_V([\eta L(p),(1-\eta) L(p)] \times [0,L(p)]))\\
& \leq 1 - \big(\ppp(\mathcal{C}^*_V([\eta L(p),(1-\eta) L(p)]^2)) - 2 \delta\big)\\
& = \ppp(\mathcal{C}_H([\eta L(p),(1-\eta)L(p)]^2)) + 2 \delta
\end{align*}
and in the other way,
\begin{align*}
\ppp(\mathcal{C}_H([0,L(p)]^2)) & \geq \ppp(\mathcal{C}_H([\eta L(p),(1-\eta) L(p)] \times [0,L(p)])) - 2\delta\\
& = 1 - \ppp(\mathcal{C}^*_V([\eta L(p),(1-\eta) L(p)] \times [0,L(p)])) - 2\delta\\
& \geq 1 - \ppp(\mathcal{C}^*_V([\eta L(p),(1-\eta) L(p)]^2)) - 2\delta\\
& = \ppp(\mathcal{C}_H([\eta L(p),(1-\eta)L(p)]^2)) - 2\delta.
\end{align*}
The claim follows readily, in particular
\begin{equation}
\PPP^{(\eta)}(\mathcal{C}_H([0,L(p)]^2)) \geq \PP_p(\mathcal{C}_H([0,L(p)]^2)) - 4 \delta,
\end{equation}
which is at least $(1/2 + \epsilon) - 4 \delta \geq 1/2 + \epsilon/2$ by the very definition of $L(p)$. It shows as desired that the sites in $[\eta L(p),(1-\eta)L(p)]^2$ produce all together a non-negligible contribution.

Now, Russo's formula applied to the interpolating measures $(\PPP^{(\eta)}_t)_{t \in [0,1]}$ (with parameters $\hat{p}^{(\eta)}_v(t)= t \times \hat{p}^{(\eta)}_v + (1-t) \times 1/2$) and the event $\mathcal{C}_H([0,L(p)]^2)$ gives
\begin{align*}
\int_0^1 \sum_{v \in [\eta L(p),(1-\eta)L(p)]^2} \big(p - 1/2\big) \: \PPP^{(\eta)}_t \big( v \leadsto^{4,\sigma_4,\bar{I}} \partial [0,L(p)]^2 \big) dt & \\
& \hspace{-5cm} = \PPP^{(\eta)}(\mathcal{C}_H([0,L(p)]^2)) - \PP_{1/2}(\mathcal{C}_H([0,L(p)]^2)),
\end{align*}
and this quantity is at least $\epsilon/2$, and thus of order $1$.

Finally, it is not hard to see that once $\eta$ fixed, we have (uniformly in $p$, $\PPP$ between $\PP_p$ and $\PP_{1-p}$, and $v \in [\eta L(p),(1-\eta)L(p)]^2$)
\begin{align*}
\PPP \big( v \leadsto^{4,\sigma_4,\bar{I}} \partial [0,L(p)]^2 \big) & \asymp \PPP \big( v \leadsto^{4,\sigma_4} \partial S_{\eta L(p)}(v)\big)\\
& \asymp \PP_{1/2}(0 \leadsto^{4,\sigma_4} \partial S_{\eta L(p)})\\
& \asymp \PP_{1/2}(0 \leadsto^{4,\sigma_4} \partial S_{L(p)}),
\end{align*}
which yields the desired conclusion.
\end{proof}

\begin{remark}
Note that the intermediate lemma was required for the lower bound only, the upper bound can be obtained directly from Russo's formula. To get the lower bound, we could also have proved that for $n \leq L(p)$,
\begin{equation}
\sum_{x \in S_n} \PPP \big( x \leadsto^{4,\sigma_4} \partial S_n \big) \asymp n^2 \pi_4(n).
\end{equation}
Basically, it comes from the fact that when we get closer to $\partial S_N$, one of the arms is shorter, but the remaining arms have less space - the 3-arm exponent in the half-plane appears.
\end{remark}

All the results we have seen so far hold for any fixed value of $\epsilon$ in $(0,1/2)$, in particular the last Proposition. Combining it with the estimate for $4$ arms, we get an important corollary, that the behavior of $L_{\epsilon}$ does not depend on the value of $\epsilon$.

\begin{corollary} \label{lengths}
For any $\epsilon, \epsilon' \in (0,1/2)$,
\begin{equation}
L_{\epsilon}(p) \asymp L_{\epsilon'}(p).
\end{equation}
\end{corollary}

\begin{proof}
To fix ideas, assume that $\epsilon \leq \epsilon'$, so that $L_{\epsilon}(p) \geq L_{\epsilon'}(p)$, and we need to prove that $L_{\epsilon}(p) \leq C L_{\epsilon'}(p)$ for some constant $C$. We know that
$$|p-1/2| \big(L_{\epsilon}(p)\big)^2 \pi_4(L_{\epsilon}(p)) \asymp 1 \asymp |p-1/2| \big(L_{\epsilon'}(p)\big)^2 \pi_4(L_{\epsilon'}(p)),$$
hence for some constant $C_1$,
$$\frac{\big(L_{\epsilon}(p)\big)^2 \pi_4(L_{\epsilon}(p))}{\big(L_{\epsilon'}(p)\big)^2 \pi_4(L_{\epsilon'}(p))} \leq C_1.$$
This yields
$$\bigg( \frac{L_{\epsilon}(p)}{L_{\epsilon'}(p)} \bigg)^2 \leq C_1 \frac{\pi_4(L_{\epsilon'}(p))}{\pi_4(L_{\epsilon}(p))} \leq C_2 \big( \pi_4(L_{\epsilon'}(p),L_{\epsilon}(p)) \big)^{-1}$$
by quasi-multiplicativity. Now we use the a-priori bound for 4 arms given by the 5-arm exponent:
$$\pi_4(L_{\epsilon'}(p),L_{\epsilon}(p)) \geq \bigg( \frac{L_{\epsilon'}(p)}{L_{\epsilon}(p)} \bigg)^{-\alpha'} \pi_5(L_{\epsilon'}(p),L_{\epsilon}(p)) \geq C_3 \bigg( \frac{L_{\epsilon'}(p)}{L_{\epsilon}(p)} \bigg)^{2-\alpha'}.$$
Together with the previous equation, it implies the result: $$L_{\epsilon}(p) \leq (C_4)^{1/\alpha'} L_{\epsilon'}(p).$$
\end{proof}

\begin{remark}
In the other direction, a RSW construction shows that we can increase $L_{\epsilon}$ by any constant factor by choosing $\epsilon$ small enough.
\end{remark}

\subsection{Uniform exponential decay, critical exponent for $\theta$\label{expsection}}

Up to now, our reasonings (separation of arms, arm events in near-critical percolation, critical exponent for $L$) were based on RSW considerations on scales $n \leq L(p)$, so that critical and near-critical percolation could be handled simultaneously. In the other direction, the definition of $L(p)$ also implies that when $n > L(p)$, the picture starts to look like super/sub-critical percolation, supporting the choice of $L_(p)$ as the characteristic scale of the model.

More precisely, we prove a property of exponential decay uniform in $p$. This property will then be used to link $L$ with the other characteristic functions, and we will derive the following expressions of $\theta$, $\chi$ and $\xi$ as functions of $L$:
\begin{enumerate}[(i)]
\item $\theta(p) \asymp \pi_1(L(p))$,

\item $\chi(p) \asymp L(p)^2 \pi_1^2(L(p))$,

\item $\xi(p) \asymp L(p)$.
\end{enumerate}
The critical exponents for these three functions will follow readily, since we already know the exponent for $L$.

\subsubsection*{Uniform exponential decay}

The following lemma shows that correlations decay exponentially fast with respect to $L(p)$. It allows to control the speed for $p$ varying:
\begin{lemma}
\label{explem}
For any $\epsilon \in (0,1/2)$, there exist constants $C_i=C_i(\epsilon)>0$ such that for all $p < 1/2$, all $n$,
\begin{equation}
\label{expdecay}
\PP_p(\mathcal{C}_H([0,n] \times [0,n])) \leq C_1 e^{- C_2 n/L_{\epsilon}(p)}.
\end{equation}
\end{lemma}

\begin{proof}
We use a block argument: for each integer $n$,
\begin{equation}
\PP_p(\mathcal{C}_H([0,2n] \times [0,4n])) \leq C' [\PP_p(\mathcal{C}_H([0,n] \times [0,2n]))]^2,
\end{equation}
with $C' = 10^2$ some universal constant.

It suffices for that (see Figure \ref{fig_exp}) to divide the parallelogram $[0,2n] \times [0,4n]$ into 4 horizontal sub-parallelograms $[0,2n] \times [i n, (i+1) n]$ ($i=0, \ldots ,3$) and 6 vertical ones $[i n,(i+1) n] \times [j n,(j+2) n]$ ($i=0, 1 , j=0, 1 ,2$). Indeed, consider a horizontal crossing of the big parallelogram: by considering its pieces in the two regions $0 < x < n$ and $n < x < 2n$, we can extract from it two sub-paths crossing one of the sub-parallelograms ``in the easy way''. They are disjoint by construction, so the claim follows by using the BK inequality.

\begin{figure}
\begin{center}
\includegraphics[width=5cm]{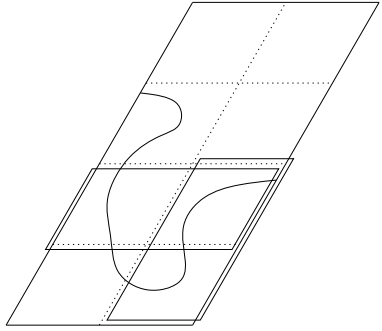}
\caption{Two of the small sub-parallelograms are crossed in the ``easy'' way.}
\label{fig_exp}
\end{center}
\end{figure}

We then obtain by iterating:
\begin{equation}
C' \PP_p(\mathcal{C}_H([0,2^k L(p)] \times [0,2^{k+1} L(p)])) \leq (C' \epsilon_1)^{2^k}
\end{equation}
as soon as $\epsilon_1 \geq \PP_p(\mathcal{C}_H([0,L(p)] \times [0,2 L(p)]))$.

Recall that by definition, $\PP_p(\mathcal{C}_H([0,L(p)] \times [0,L(p)])) \leq \epsilon_0$ if $\epsilon \leq \epsilon_0$. The RSW theory thus entails (Theorem \ref{thm_RSW}) that for all fixed $\epsilon_1 > 0$, we can take $\epsilon_0$ sufficiently small to get automatically (and independently of $p$) that
\begin{equation}
\PP_p(\mathcal{C}_H([0,L(p)] \times [0,2 L(p)])) \leq \epsilon_1.
\end{equation}

We now choose $\epsilon_1 = 1/(e^2 C')$. For each integer $n \geq L(p)$, we can define $k=k(n)$ such that $2^k \leq n/L(p) < 2^{k+1}$, and then,
\begin{align*}
\PP_p(\mathcal{C}_H([0,n] \times [0,n])) & \leq \PP_p(\mathcal{C}_H([0,2^k L(p)] \times [0,2^{k+1} L(p)]))\\
& \leq e^{-2^{k+1}}\\
& \leq e \times e^{-n/L(p)},
\end{align*}
which is also valid for $n < L(p)$, thanks to the extra factor $e$.

Hence, we have proved the property for any $\epsilon$ below some fixed value $\epsilon_0$ (given by RSW). The result for any $\epsilon \in (0,1/2)$ follows readily by using the equivalence of lengths for different values of $\epsilon$ (Corollary \ref{lengths}).
\end{proof}

We would like to stress the fact that we have not used any of the previous results until the last step, this exponential decay property could thus have been derived much earlier -- but only for values of $\epsilon$ small enough. It would for instance provide a more direct way to prove that
$$L_{\epsilon}(p) \asymp L_{\epsilon'}(p),$$
but still only for $\epsilon$, $\epsilon'$ less than some fixed value.

\begin{remark}
It will sometimes reveal more useful to know this property for crossings of longer parallelograms ``in the easy way'': we also have for any $k \geq 1$,
\begin{equation}
\label{explem2}
\PP_p(\mathcal{C}_H([0,n] \times [0,k n])) \leq C_1^{(k)} e^{- C_2^{(k)} n/L(p)}
\end{equation}
for some constants $C_i^{(k)}$ (depending on $k$ and $\epsilon$). This can be proved by combining the previous lemma with the fact that in Theorem \ref{thm_RSW}, we can take $f_k$ satisfying $f_k(1-\epsilon)=1 - C_k \epsilon^{\alpha_k} + o(\epsilon^{\alpha_k})$ for some $C_k,\alpha_k>0$.
\end{remark}

\subsubsection*{Consequence for $\theta$}

When $p>1/2$, we now show that at distance $L(p)$ from the origin, we are already ``not too far from infinity'': once we have reached this distance, there is a positive probability (bounded away from $0$ uniformly in $p$) to reach infinity.

\begin{corollary} \label{notfar}
We have
\begin{equation} \label{theta1}
\theta(p) = \PP_p\big( 0 \leadsto \infty \big) \asymp \PP_p\big( 0 \leadsto \partial S_{L(p)} \big)
\end{equation}
uniformly in $p>1/2$.
\end{corollary}

\begin{proof}
It suffices to consider overlapping parallelograms like in Figure \ref{theta}, each parallelogram twice larger than the previous one, so that the $k^\textrm{th}$ of them has a probability at least $1 - C_1 e^{-C_2 2^k}$ to present a crossing in the ``hard'' direction (thanks to the previous remark). Since $\prod_k (1 - C_1 e^{-C_2 2^k})> 0$, we are done.
\end{proof}

\begin{figure}
\begin{center}
\includegraphics[width=10cm]{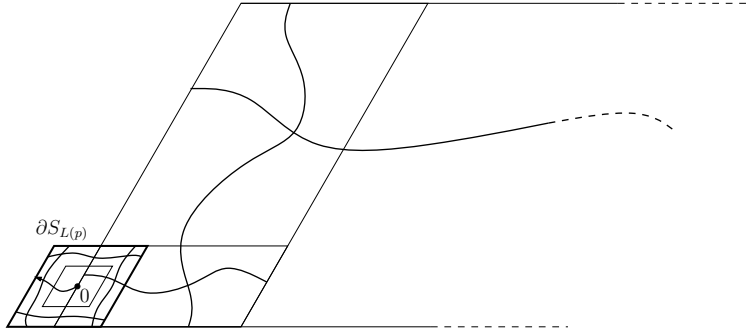}
\caption{We consider overlapping parallelograms, with size doubling at each step.}
\label{theta}
\end{center}
\end{figure}

Now, combining Eq.(\ref{theta1}) with Theorem \ref{armnear} gives (for $p>1/2$)
\begin{equation}
\theta(p) \asymp \PP_p\big( 0 \leadsto \partial S_{L(p)} \big) \asymp \PP_{1/2}\big( 0 \leadsto \partial S_{L(p)} \big) = \pi_1(L(p)).
\end{equation}
Using the $1$-arm exponent $\alpha_1 = 5/48$ stated in Theorem \ref{armcrit}, we get
\begin{equation}
\theta(p) \approx \big(L(p)\big)^{-5/48}
\end{equation}
as $p \to 1/2^+$. Together with the critical exponent for $L$ derived previously, this provides the critical exponent for $\theta$:
\begin{equation}
\theta(p) \approx \big((p - 1/2)^{-4/3}\big)^{-5/48} \approx (p - 1/2)^{5/36}.
\end{equation}

\subsubsection*{Equivalence of $L$ and $\tilde{\xi}$}

In the other direction, performing a RSW-type construction like in Figure \ref{RSW_construction} yields
\begin{equation}
\PP_p(\mathcal{C}_H([0,k L(p)] \times [0,L(p)])) \geq \delta_2^{k-1} \delta_1^{k-2} = C_1 e^{- C_2 k L(p)/ L(p)}
\end{equation}
so that $L(p)$ is the \emph{exact} speed of decaying. Once knowing this, it is easy to compare $L$ and $\tilde{\xi}$.

\begin{corollary}
For any fixed $\epsilon \in (0,1/2)$,
\begin{equation}
\tilde{\xi}(p) \asymp L_{\epsilon}(p).
\end{equation}
\end{corollary}

\begin{proof}
We exploit the previous remark: on one hand $L$ is the exact speed of decaying for crossings of rhombi, and on the other hand $\tilde{\xi}$ was defined to give the optimal bound for point-to-point connections.

More precisely, for any $x \in \partial S_n$ we have
\begin{align*}
\tau_{0,x} = \PP_p(0 \leadsto x) & \leq \PP_p(\mathcal{C}_H([0,n]\times[0,2n]))\\
& \leq C_1^{(2)} e^{- C_2^{(2)} n / L(p)}
\end{align*}
so that $\tau_{k L(p)} \leq C_1^{(2)} e^{- C_2^{(2)} k}$ and
$$- \frac{\log \tau_{k L(p)}}{k L(p)} \geq - \frac{1}{k L(p)} \big(\log C_1^{(2)} - C_2^{(2)} k\big) \xrightarrow[k \to \infty]{} \frac{C_2^{(2)}}{L(p)}.$$
Hence,
$$\frac{1}{\tilde{\xi}(p)} \geq \frac{C_2^{(2)}}{L(p)}$$
and finally $\tilde{\xi}(p) \leq C L(p)$.

Conversely, we know that $\PP_p(\mathcal{C}_H([0,k L(p)] \times [0,k L(p)])) \geq \tilde{C}_1 e^{-\tilde{C}_2 k}$ for some $\tilde{C}_i>0$. Consequently,
$$\tau_{k L(p)} \geq \frac{1}{(k L(p))^2} \PP_p(\mathcal{C}_H([0,k L(p)] \times [0,k L(p)])) \geq \frac{1}{(k L(p))^2} \tilde{C}_1 e^{-\tilde{C}_2 k},$$
which implies
$$- \frac{\log \tau_{k L(p)}}{k L(p)} \leq - \frac{1}{k L(p)} \big(\log \tilde{C}_1 - 2 \log (k L(p)) - \tilde{C}_2 k\big) \xrightarrow[k \to \infty]{} \frac{\tilde{C}_2}{L(p)},$$
whence the conclusion: $\tilde{\xi}(p) \geq C' L(p)$.
\end{proof}

\subsection{Further estimates, critical exponents for $\chi$ and $\xi$\label{xi}}

\subsubsection*{Estimates from critical percolation}

We start by stating some estimates that we will need. These estimates were originally derived for critical percolation (see e.g. \cite{Ke2,Ke3}), but for exactly the same reasons they also hold for near-critical percolation on scales $n \leq L(p)$:

\begin{lemma}
Uniformly in $p$, $\PPP$ between $\PP_p$ and $\PP_{1-p}$ and $n \leq L(p)$, we have
\begin{enumerate}[1.]
\item $\EEE \big[|x \in S_n : x \leadsto \partial S_n| \big] \asymp n^2 \pi_1(n)$. \label{it2}

\item For any $t \geq 0$,
$$\sum_{x \in S_n} \|x\|_{\infty}^t \PPP\big(0 \leadsto x\big) \asymp \sum_{x \in S_n} \|x\|_{\infty}^t \PPP\big(0 \leadsto^{S_n} x\big) \asymp n^{t+2} \pi_1^2(n).$$ \label{it1}
\end{enumerate}
\end{lemma}

Note that item \ref{it1}. implies in particular for $t=0$ that
$$\EEE \big[|x \in S_n : x \leadsto 0| \big] \asymp \EEE \big[|x \in S_n : x \leadsto^{S_n} 0| \big] \asymp n^2 \pi_1^2(n).$$

\begin{proof}
We will have use for the fact that we can take $\alpha=1/2$ for $j=1$ in Eq.(\ref{pi1}) (actually any $\alpha < 1$ would be enough for our purpose): for any integers $n < N$,
\begin{equation} \label{pi3}
\pi_1(n|N) \geq C (n/N)^{1/2}.
\end{equation}
This can be proved like (3.15) of \cite{BK}: just use blocks of size $n$ instead of individual sites to obtain that $\frac{N}{n} \pi_1^2(n|N)$ is bounded below by a constant.

\smallskip

\noindent \textbf{Proof of item 1.} We will use that
\begin{equation}
\EEE \big[|x \in S_n : x \leadsto 0| \big] = \sum_{x \in S_n} \PPP(x \leadsto \partial S_n).
\end{equation}

For the lower bound, it suffices to note that for any $x \in S_n$,
\begin{equation}
\PPP(x \leadsto \partial S_n) \geq \PPP(x \leadsto \partial S_{2n}(x)) = \PPP^x(0 \leadsto \partial S_{2n})
\end{equation}
(where $\PPP^x$ is the measure $\PPP$ translated by $x$), and that
\begin{equation}
\PPP^x(0 \leadsto \partial S_{2n}) \geq C_1 \PPP^x(0 \leadsto \partial S_n) \geq C_2 \pi_1(n)
\end{equation}
by extendability and Theorem \ref{armnear} for one arm.

For the upper bound, we sum over concentric rhombi around $0$:
\begin{align*}
\sum_{x \in S_n} \PPP(x \leadsto \partial S_n) & \leq \sum_{x \in S_n} \PPP(x \leadsto \partial S_{d(x,\partial S_n)}(x))\\
& \leq \sum_{j=1}^n C_1 n \PPP(0 \leadsto \partial S_j)
\end{align*}
using that there are at most $C_1 n$ sites at distance $j$ from $\partial S_n$. By Theorem \ref{armnear}, this last sum is at most
\begin{align*}
C_2 n \sum_{j=1}^n \pi_1(j) & \leq C_2 n \pi_1(n) \sum_{j=1}^n \frac{\pi_1(j)}{\pi_1(n)}\\
& \leq C_3 n \pi_1(n) \sum_{j=1}^n \pi_1(j|n)^{-1}
\end{align*}
by quasi-multiplicativity. Now Eq.(\ref{pi3}) says that $\pi_1(j|n) \geq (j/n)^{1/2}$, so that
$$\sum_{j=1}^n \pi_1(j|n)^{-1} \leq \sum_{j=1}^n (j/n)^{-1/2}= n^{1/2} \sum_{j=1}^n j^{-1/2} \leq C_4 n,$$
which gives the desired upper bound.

\smallskip

\noindent \textbf{Proof of item 2.}

Since
$$\sum_{x \in S_n} \|x\|_{\infty}^t \PPP\big(0 \leadsto^{S_n} x\big) \leq \sum_{x \in S_n} \|x\|_{\infty}^t \PPP\big(0 \leadsto x\big),$$
it suffices to prove the desired lower bound for the left-hand side, and the upper bound for the right-hand side.

Consider the lower bound first. We note (see Figure \ref{critical_estimate}) that if $0$ is connected to $\partial S_n$ and if there exists a black circuit in $S_{2n/3,n}$ (which occurs with probability at least $\delta_6^4$ by RSW), then any $x \in S_{n/3,2n/3}$ connected to $\partial S_{2n}(x)$ will be connected to $0$ in $S_n$. Using the FKG inequality, we thus get for such an $x$:
$$\PPP(0 \leadsto^{S_n} x) \geq \delta_6^4 \PPP(0 \leadsto \partial S_n) \PPP(x \leadsto \partial S_{2n}(x))$$
which is at least (still using extendability and Theorem \ref{armnear}) $C_1 \pi_1^2(n)$. Consequently,
\begin{align*}
\sum_{x \in S_n} \|x\|_{\infty}^t \PPP\big(0 \leadsto^{S_n} x\big) & \geq \sum_{x \in S_{n/3,2n/3}} \|x\|_{\infty}^t C_1 \pi_1^2(n)\\
& \geq C_2 n^2 (n/3)^t \pi_1^2(n).
\end{align*}

\begin{figure}
\begin{center}
\includegraphics[width=8cm]{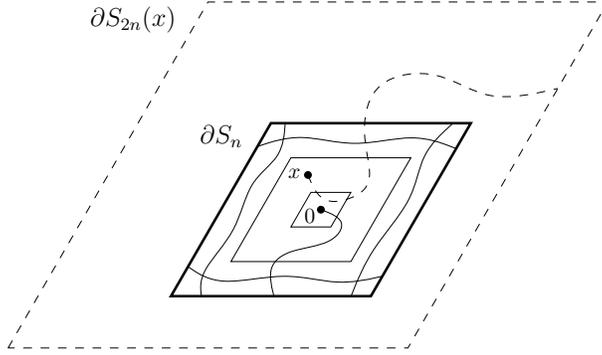}
\caption{With this construction, any site $x$ in $S_{n/3,2n/3}$ connected to a site at distance $2n$ is also connected to $0$ in $S_n$.}
\label{critical_estimate}
\end{center}
\end{figure}

Let us turn to the upper bound. We take a logarithmic division of $S_n$: define $k=k(n)$ so that $2^k < n \leq 2^{k+1}$, we have
\begin{equation}
\sum_{x \in S_n} \|x\|_{\infty}^t \PPP\big(0 \leadsto x\big) \leq C_1 + \sum_{j=3}^{k+1} \sum_{x \in S_{2^{j-1},2^j}} \|x\|_{\infty}^t \PPP\big(0 \leadsto x\big).
\end{equation}
Now for $x \in S_{2^{j-1},2^j}$, take the two boxes $S_{2^{j-2}}(0)$ and $S_{2^{j-2}}(x)$: since they are disjoint,
\begin{equation}
\PPP\big(0 \leadsto x\big) \leq \PPP\big(0 \leadsto S_{2^{j-2}}(0)\big) \PPP\big(x \leadsto S_{2^{j-2}}(x)\big),
\end{equation}
which is at most $C_2 \pi_1^2(2^{j-1})$ using the same arguments as before. Our sum is thus less than (since $|S_{2^{j-1},2^j}| \leq C_3 2^{2j}$)
\begin{equation}
\sum_{j=3}^{k+1} C_3 2^{2j} \times (2^j)^t \times (C_2 \pi_1^2(2^{j+1})) \leq C_4 2^{(2+t) k} \pi_1^2(2^k) \times \Bigg[ \sum_{j=3}^{k+1} 2^{(2+t) (j-k)} \frac{\pi_1^2(2^{j-1})}{\pi_1^2(2^k)} \Bigg].
\end{equation}
Now $2^{(2+t) k} \pi_1^2(2^k) \leq C_5 n^{2+t} \pi_1^2(n)$, and this yields the desired result, using as previously $\frac{\pi_1(2^j)}{\pi_1(2^k)} \leq C_6 \pi_1(2^j|2^k)^{-1} \leq C_6 2^{-(j-k)/2}$:
\begin{align*}
\sum_{j=3}^{k+1} 2^{(2+t) (j-k)} \frac{\pi_1^2(2^{j-1})}{\pi_1^2(2^k)} & \leq C_6 \sum_{j=3}^{k+1} 2^{(2+t) (j-k)} 2^{- (j-k)}\\
& \leq C_7 \sum_{l=0}^{k-3} 2^{-(1+t) l},
\end{align*}
and this sum is bounded by $\sum_{l=0}^{\infty} 2^{-(1+t) l} < \infty$.
\end{proof}

\subsubsection*{Main estimate}

The following lemma will allow us to link directly $\chi$ and $\xi$ with $L$. Roughly speaking, it relies on the fact that the sites at distance much larger than $L(p)$ from the origin have a negligible contribution, due to the exponential decay property, so that the sites in $S_{L(p)}$ produce a positive fraction of the total sum:

\begin{lemma} For any $t \geq 0$, we have
\begin{equation}
\sum_x \|x\|_{\infty}^t \PP_p\big(0 \leadsto x , |C(0)|<\infty\big) \asymp L(p)^{t+2} \pi_1^2(L(p))
\end{equation}
uniformly in $p$.
\end{lemma}

\begin{proof}

\noindent \textbf{Lower bound.}
The lower bound is a direct consequence of item \ref{it1}. above: indeed,
\begin{align*}
\sum_x \|x\|_{\infty}^t \PP_p\big(0 \leadsto x , |C(0)|<\infty\big) & \\
& \hspace{-2cm} \geq \PP_p(\text{$\exists$ white circuit in $S_{L,2L}$}) \sum_{x \in S_L} \|x\|_{\infty}^t \PP_p\big(0 \leadsto^{S_L} x\big) \\
& \hspace{-2cm} \geq \delta_4^4 \sum_{x \in S_L} \|x\|_{\infty}^t \PP_p\big(0 \leadsto^{S_L} x\big)
\end{align*}
by RSW, and item \ref{it1}. gives
$$\sum_{x \in S_L} \|x\|_{\infty}^t \PP_p\big(0 \leadsto^{S_L} x\big) \geq C L^{t+2} \pi_1^2(L).$$

\noindent \textbf{Upper bound.}
To get the upper bound, we cover the plane by translating $S_L$: we consider the family of rhombi $S_L(2 n_1 L, 2 n_2 L)$, for any two integers $n_1$ and $n_2$. By isolating the contribution of $S_L$, we get:
\begin{align*}
\sum_x \|x\|_{\infty}^t \PP_p\big(0 \leadsto x , |C(0)|<\infty\big) & \\
& \hspace{-4cm} \leq \sum_{x \in S_L} \|x\|_{\infty}^t \PP_p\big(0 \leadsto x , |C(0)|<\infty\big) \\
& \hspace{-3cm} + \sum_{(n_1,n_2) \neq (0,0)} \sum_{x \in S_L(2 n_1 L, 2 n_2 L)} \|x\|_{\infty}^t \PP_p\big(0 \leadsto x , |C(0)|<\infty\big).
\end{align*}

\begin{figure}
\begin{center}
\includegraphics[width=10cm]{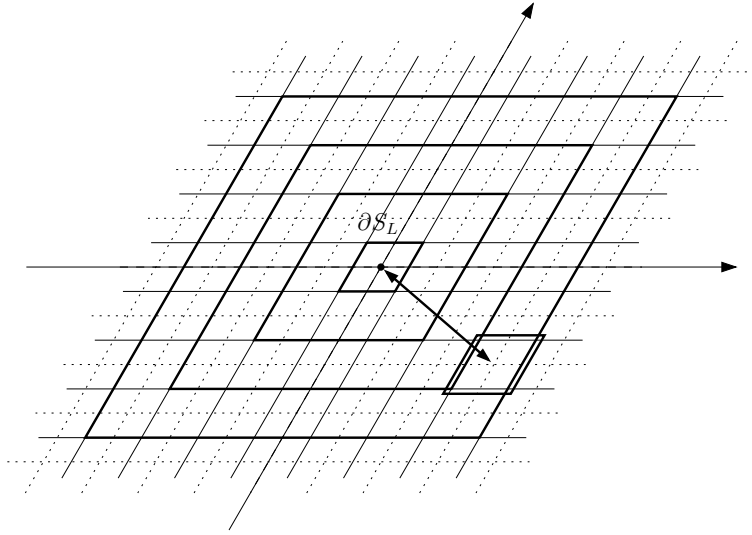}
\caption{For the upper bound, we cover the plane with rhombi of size $2L$ and sum their different contributions.}
\label{sum_rhombi}
\end{center}
\end{figure}

Using item \ref{it1}. above, we see that the rhombus $S_L$ gives a contribution
$$\sum_{x \in S_L} \|x\|_{\infty}^t \PP_p\big(0 \leadsto x , |C(0)|<\infty\big) \leq C L^{t+2} \pi_1^2(L),$$
which is of the right order of magnitude.

We now prove that each small rhombus outside of $S_L$ at distance $k L$ gives a contribution of order $\pi_1(L) \times L^t \times \EE_p \big[|x \in S_L : x \leadsto \partial S_L| \big] \asymp L^{t+2} \pi_1^2(L)$ (using item \ref{it2}.), multiplied by some quantity which decays exponentially fast in $k$ and will thus produce a series of finite sum. More precisely, if we regroup the rhombi into concentric annuli around $S_L$, we get that the previous summation is at most
\begin{align*}
\sum_{k=1}^{\infty} \sum_{\substack{(n_1,n_2)\\ \|(n_1,n_2)\|_{\infty}=k}} \sum_{x \in S_L(2 n_1 L, 2 n_2 L)} \|x\|_{\infty}^t \PP_p\big(0 \leadsto x , |C(0)|<\infty\big) & \\
& \hspace{-8cm} \leq \sum_{k=1}^{\infty} \sum_{\substack{(n_1,n_2)\\ \|(n_1,n_2)\|_{\infty}=k}} \sum_{x \in S_L(2 n_1 L, 2 n_2 L)} [(2k+1)L]^t \PP_p\big(0 \leadsto x , |C(0)|<\infty\big) & \\
& \hspace{-8cm} \leq \sum_{k=1}^{\infty} \sum_{\substack{(n_1,n_2)\\ \|(n_1,n_2)\|_{\infty}=k}} C' k^t L^t \: \EE_p \big[|C(0) \cap S_L(2 n_1 L, 2 n_2 L)| ; |C(0)|<\infty\big].
\end{align*}

Now we have to distinguish between the sub-critical and the super-critical cases: we are going to prove that in both cases,
$$\EE_p \big[|C(0) \cap S_L(2 n_1 L, 2 n_2 L)| ; |C(0)|<\infty\big] \leq C_1 L^2 \pi^2_1(L) e^{- C_2 k}$$
for some constants $C_1, C_2 >0$. When $p<1/2$, we will use that
\begin{equation}
\PP_p(\partial S_L \leadsto \partial S_{k L}) \leq C_3 e^{- C_4 k},
\end{equation}
which is a direct consequence of the exponential decay property Eq.(\ref{explem2}) for ``longer'' parallelograms. When $p>1/2$, we have an analog result, which can be deduced from the sub-critical case just like in the discrete case (just replace sites by translations of $S_L$):
\begin{align*}
\PP_p(\partial S_L \leadsto \partial S_{k L} , |C(0)| < \infty ) & \\
& \hspace{-4cm} \leq \PP_p(\text{$\exists$ white circuit surrounding a site on $\partial S_L$ and a site on $\partial S_{k L}$}) \\
& \hspace{-4cm} \leq C_5 e^{- C_6 k}.
\end{align*}

Assume first that $p<1/2$. By independence, we have ($\|(n_1,n_2)\|_{\infty} = k$)
\begin{align*}
\EE_p \big[|C(0) \cap S_L(2 n_1 L, 2 n_2 L)| ; |C(0)|<\infty\big] & \\
& \hspace{-5cm} \leq \PP_p (0 \leadsto \partial S_L) \EE_p \big[|x \in S_L(2 n_1 L, 2 n_2 L) : x \leadsto \partial S_L(2 n_1 L, 2 n_2 L)| \big] \\
& \hspace{-4cm} \times \PP_p (\partial S_L \leadsto \partial S_{(2k-1) L}) \\
& \hspace{-5cm} \leq \pi_1(L) \times (C L^2 \pi_1(L)) \times C'_3 e^{- C'_4 k}.
\end{align*}

If $p>1/2$, we write similarly (here we use FKG to separate the existence of a white circuit (decreasing) from the other terms (increasing), and then independence of the remaining terms)
\begin{align*}
\EE_p \big[|C(0) \cap S_L(2 n_1 L, 2 n_2 L)| ; |C(0)|<\infty\big] & \\
& \hspace{-6cm} \leq \PP_p (0 \leadsto \partial S_L) \EE_p \big[|x \in S_L(2 n_1 L, 2 n_2 L) : x \leadsto \partial S_L(2 n_1 L, 2 n_2 L)| \big] \\
& \hspace{-5.5cm} \times \PP_p(\text{$\exists$ white circuit surrounding a site on $\partial S_L$ and a site on $\partial S_{(2k-1) L}$}) \\
& \hspace{-6cm} \leq \pi_1(L) \times (C L^2 \pi_1(L)) \times C'_5 e^{- C'_6 k}.
\end{align*}

Since there are at most $C'' k$ rhombi at distance $k$ for some constant $C''$, the previous summation is in both cases less than
$$\sum_{k=1}^{\infty} C'' k \times C' k^t L^t \times C_1 L^2 \pi^2_1(L) e^{- C_2 k} \leq C''' \bigg(\sum_{k=1}^{\infty} k^{t+1} e^{- C_2 k}\bigg) L^{t+2} \pi^2_1(L),$$
which yields the desired upper bound, as $\sum_{k=1}^{\infty} k^{t+1} e^{- C_2 k} < \infty$.

\end{proof}

\subsubsection*{Critical exponents for $\chi$ and $\xi$}

The previous lemma reads for $t=0$:
\begin{proposition} We have
\begin{equation}
\chi(p) = \EE_p \big[|C(0)| ; |C(0)|<\infty\big] \asymp L(p)^2 \pi_1^2(L(p)).
\end{equation}
\end{proposition}
In other words, ``$\chi(p) \asymp \chi^{\text{near}}(p)$''. It provides the critical exponent for $\chi$:
\begin{equation}
\chi(p) \approx L(p)^2 \big[L(p)^{-5/48}\big]^2 \approx \big[|p-1/2|^{-4/3}\big]^{86/48} \approx |p-1/2|^{-43/18}.
\end{equation}

Recall that $\xi$ was defined via the formula
$$\xi(p) = \Bigg[ \frac{1}{\EE_p \big[|C(0)| ; |C(0)|<\infty\big]} \sum_x \|x\|_{\infty}^2 \PP_p\big(0 \leadsto x , |C(0)|<\infty\big) \Bigg]^{1/2}.$$
Using the last proposition and the lemma for $t=2$, we get
\begin{equation}
\xi(p) \asymp \Bigg[ \frac{L(p)^4 \pi_1^2(L(p))}{L(p)^2 \pi_1^2(L(p))} \Bigg]^{1/2} = L(p).
\end{equation}
We thus obtain the following proposition, announced in Section \ref{different}:
\begin{proposition}
We have
\begin{equation}
\xi(p) \asymp L(p).
\end{equation}
\end{proposition}
This implies in particular that
\begin{equation}
\xi(p) \approx |p-1/2|^{-4/3}.
\end{equation}

\section{Concluding remarks} \label{sec_remarks}

\subsection{Other lattices}

Most of the results presented here (the separation of arms, the theorem concerning arm events on a scale $L(p)$, the ``universal'' arm exponents, the relations between the different characteristic functions, etc.) come from RSW considerations or the exponential decay property, and remain true on other regular lattices like the square lattice. The triangular lattice has a property of self-duality which makes life easier, in general we have to consider the original lattice together with the matching lattice (obtained by ``filling'' each face with a complete graph): instead of black or white connections, we thus talk about primal and dual connections. We can also handle bond percolation in this way. We refer the reader to the original paper of Kesten \cite{Ke4} for more details, where results are proved in this more general setting. The only obstruction to get the critical exponents is actually the derivation of the arm exponents at the critical point $p=p_c$ (besides only two exponents are needed, for $1$ arm and $4$ alternating arms).

Consider site percolation on $\mathbb{Z}^2$ for instance. We know that $0 < \alpha_j,\alpha'_j,\beta_j < \infty$ for any $j \geq 1$. Hence the a-priori estimate
$$\mathbb{P}_{p_c}(0 \leadsto^{4,\sigma_4} \partial S_N) \geq N^{-2+\alpha}$$
for some $\alpha >0$, coming from the $5$-arm exponent, remains true: $\alpha_4 < 2$ (and in the same way $\alpha_6 > 2$). Combined with Proposition \ref{L_exp}, this leads to the weaker but nonetheless interesting statement
\begin{equation}
L(p) \leq |p - p_c|^{-A}
\end{equation}
for some $A >0$. Hence $\nu < \infty$, and then $\gamma < \infty$ (if these exponents exist). Using $\alpha_1 < \infty$, we also get $\beta < \infty$.

If we use a RSW construction in a box, we can make appear $3$-arm sites on the lowest crossing and deduce that $\alpha_1 \leq 1/3$. Here are rigorous bounds for the critical exponents in two dimensions:
$$\begin{array}{c|c}
\quad \text{triangular lattice} \quad & \quad \text{general rigorous bounds} \quad\\[2mm]
\hline\\
\quad \beta = 5/36 \quad & \quad 0 < \beta <1 \quad\\[1mm]
\quad \gamma = 43/18 \quad & \quad 8/5 \leq \gamma < \infty \quad\\[1mm]
\quad \nu = 4/3 \quad & \quad 1 < \nu < \infty \quad\\[1mm]
\end{array}$$
For more details, the reader can consult \cite{Ke4} and the references therein.

\subsection{Some related issues}

For the sake of completeness, let us just mention finally that the way the correlation length $L$ was defined also allows to use directly the compactness results of \cite{AB}. Indeed, the a-priori estimates on arm events coming from RSW considerations are exactly the hypothesis (H1) of this paper. This hypothesis entails that the curve cannot cross too many times any annulus, and thus cannot be too ``intricate'': this is Theorem 1, asserting the existence of Hölder parametrizations with high probability.

This regularity property then implies tightness, using (a version of) Arzela-Ascoli's theorem for continuous functions on a compact subset of the plane. We can thus show in this way the existence of scaling limits for near-critical percolation interfaces.

As a conclusion, let us also mention that the techniques presented here are important to study various models related to the critical regime, for instance Incipient Infinite Clusters \cite{CCD,J}, Dynamical Percolation \cite{ScSt}, Gradient Percolation \cite{N1}\ldots

\vfill

\section*{Acknowledgements}

I would like to thank W. Werner for suggesting me to look at these questions. He also had an important role in reading various intermediate versions of this paper. I also enjoyed many enlightening discussions with various people, including V. Beffara, J. van den Berg, F. Camia, L. Chayes, C. Garban and O. Schramm.

\end{document}